\documentclass[]{amsart}
\usepackage{fullpage}

\usepackage[utf8]{inputenc}
\usepackage[T1]{fontenc}

\usepackage{graphicx}
\usepackage{multirow}
\usepackage{booktabs}
\usepackage{textcomp}

\usepackage{amsmath}
\usepackage{amssymb}
\usepackage{amsfonts}
\usepackage{amsthm}
\usepackage{amscd}
\usepackage{mathrsfs}

\usepackage{csquotes}
\usepackage{xparse}

\usepackage{tikz}
\usetikzlibrary{patterns}

\theoremstyle{thmstyleone}
\newtheorem{theorem}{Theorem}[section]
\newtheorem{lemma}[theorem]{Lemma}

\newtheorem{proposition}[theorem]{Proposition}
\newtheorem{corollary}[theorem]{Corollary}

\theoremstyle{thmstylethree}

\newtheorem{convention}[theorem]{Convention}

\theoremstyle{thmstyletwo}
\newtheorem{remark}[theorem]{Remark}
\newtheorem{example}[theorem]{Example}

\makeatletter
\newcommand\suchthat{%
 \@ifstar
  {\mathrel{}\middle|\mathrel{}}
  {\mid}%
}
\makeatother

\newcommand{\bbN}{\mathbb N}
\newcommand{\bbR}{\mathbb R}
\newcommand{\bbZ}{\mathbb Z}

\newcommand{\bfH}{\mathbf H}
\newcommand{\bfL}{\mathbf L}

\newcommand{\bfu}{\mathbf u}

\newcommand{\calC}{\mathcal C}
\newcommand{\calP}{\mathcal P}
\newcommand{\calQ}{\mathcal Q}
\newcommand{\calT}{\mathcal T}
\newcommand{\calU}{\mathcal U}

\newcommand{\boldgamma}{\boldsymbol{\gamma}}
\newcommand{\boldnu}{\boldsymbol{\nu}}
\newcommand{\boldphi}{\boldsymbol{\phi}}

\newcommand{\dif}{\mathrm d}
\newcommand{\dx}{\;\dif x}

\newcommand{\card}[1]{\left|#1\right|}
\newcommand{\closure}[1]{\overline{#1}}
\newcommand{\interior}[1]{\operatorname{int} #1}

\newcommand{\Id}{\operatorname{Id}}
\newcommand{\nor}{\operatorname{nor}}
\newcommand{\supp}{\operatorname{supp}}
\newcommand{\vol}{\operatorname{vol}}

\renewcommand{\restriction}{|}

\DeclareMathOperator{\cardinality}{\#}

\newcommand{\trace}{\operatorname{tr}}
\newcommand{\Trace}{\operatorname{Tr}}
\newcommand{\Ext}{\operatorname{Ext}}

\newcommand{\normalfield}{\boldnu}
\newcommand{\contract}{\mathbin{\lrcorner}}

\newcommand{\exteriorderivative}{\mathrm d}

\newcommand{\BH}{\mathrm{BH}}

\newcommand{\shapemeasure}{\mu}
\newcommand{\comparisonconstant}{\rho}
\newcommand{\cardinalitybound}{\kappa}

\newcommand{\Sobdiff}{V}

\newcommand{\Localproj}{P}
\newcommand{\Projection}{\mathscr P}
\newcommand{\Poincare}{\mathscr P}
\newcommand{\Harmonic}{\mathscr H}
\newcommand{\Extension}{\mathscr E}

\newcommand{\grad}{\operatorname{grad}}
\newcommand{\curl}{\operatorname{curl}}

\newcommand{\divergence}{\operatorname{div}}

\newcommand{\supersimplices}{\nabla^{\uparrow}}
\newcommand{\subsimplices}{\Delta^{\downarrow}}

\newcommand{\Lagrange}{\mathsf P}
\newcommand{\Nedfst}{\mathsf{Ned}^{\mathrm i}}
\newcommand{\Nedsnd}{\mathsf{Ned}^{\mathrm{ii}}}
\newcommand{\BDM}{\mathsf{BDM}}
\newcommand{\RT}{\mathsf{RT}}

\newcommand{\ding}{u}
\newcommand{\dong}{w}

\usepackage[normalem]{ulem}
\usepackage[most]{tcolorbox}

\definecolor{OldColor}{HTML}{FF9900}
\definecolor{NewColor}{HTML}{0000CC}
\definecolor{MarkColor}{HTML}{FFF59D}

\NewDocumentCommand{\todobox}{ O{yellow} +m }{%
  \begin{tcolorbox}[
    enhanced,
    breakable,
    colback=#1!10,
    colframe=#1!70!black,
    boxrule=0.8pt,
    arc=2pt,
    left=6pt,
    right=6pt,
    top=4pt,
    bottom=4pt
  ]
  \textbf{TODO:}~#2
  \end{tcolorbox}
}

\NewDocumentCommand{\donebox}{ O{gray} +m }{%
  \begin{tcolorbox}[
    enhanced,
    breakable,
    colback=#1!10,
    colframe=#1!70!black,
    boxrule=0.8pt,
    arc=2pt,
    left=6pt,
    right=6pt,
    top=4pt,
    bottom=4pt
  ]
  \textbf{DONE:}~#2
  \end{tcolorbox}
}

\title[Averaging-based local projections in FEEC]
{Averaging-based local projections\\in finite element exterior calculus}

\author{Martin Licht}

\address{%
Department of Mathematics,
TU Dresden,
Zellescher Weg 12--14,
01069 Dresden,
Germany}

\email{martin\_werner.licht@tu-dresden.de}



\begin{document}

\begin{abstract}
We construct local projection operators onto spaces of finite element differential forms on simplicial meshes.
The projections are locally stable in Lebesgue and Sobolev--Slobodeckij norms, uniformly with respect to the mesh size,
enforce homogeneous partial boundary conditions, and satisfy piecewise Bramble--Hilbert estimates localized around element stars.
Special cases include the Ern--Guermond projection and a modified Cl\'ement-type interpolant that is idempotent.
These results lead to an equivalence between global and local best-approximation errors for curl- or divergence-conforming finite element spaces.
Our construction combines techniques for the Ern--Guermond and the Scott--Zhang projections within the framework of finite element exterior calculus, and addresses three different regularity regimes.
We instantiate the abstract projection for the divergence-conforming Brezzi--Douglas--Marini and Raviart--Thomas elements and for the curl-conforming N\'ed\'elec elements.
As an auxiliary result of independent merit, we contribute an extension operator for differential forms that preserves higher Sobolev--Slobodeckij regularity.
\end{abstract}

\maketitle

\section{Introduction}

Interpolation operators are indispensable to the error analysis of finite element methods. 
For example, they translate quasi-optimal error estimates into convergence rates in terms of the mesh size. 
Standard examples are the canonical interpolants, which are defined via the finite element degrees of freedom.  
They allow for local estimates of the approximation error between a function and its canonical finite element interpolation. 
While this mechanism is classical for Lagrange elements, it is also well-known for the curl-conforming and divergence-conforming vector-valued finite element spaces~\cite{brezzi_fortin1991,monk2003finite,brennerscott}.
The canonical interpolants, however, are restricted to sufficiently smooth scalar and vector fields. Since they are defined by evaluating the finite element degrees of freedom, they require enough regularity for traces to be meaningful.

Quasi-interpolation operators are defined on non-smooth functions, which makes them applicable across a broader spectrum of regularity regimes. These have long been known for scalar fields, where the best-known examples are the Cl\'ement interpolant~\cite{clement1975approximation}, which is defined on Lebesgue spaces, and the Scott--Zhang projection~\cite{scott1990finite}, which is defined on first-order Sobolev spaces.
The construction is standard for scalar fields, but quasi-interpolation operators for non-smooth vector fields were absent from the literature for a long time.

A major step was taken by Ern and Guermond~\cite{ern2017finite}, who introduced averaging-based quasi-interpolation operators for finite element spaces of vector fields, such as Brezzi--Douglas--Marini, N\'ed\'elec, and Raviart--Thomas spaces.  
This laid the foundation for finite element error analysis in low-regularity regimes that arise in numerical electromagnetism. 
Analogues of the Cl\'ement interpolant and the Scott--Zhang projection within the formalism of finite element exterior calculus~\cite{licht2021} generalize many results for scalar spaces to the setting of vector analysis. 
In particular, this has generalized the piecewise Bramble--Hilbert lemma~\cite{veeser2016approximating,camacho2015L2} to vector fields and differential forms.

In this manuscript, we study a quasi-interpolant for scalar and vector fields based on local weighted averaging.
We develop the averaging-based projection further and construct local projection operators onto conforming finite element spaces on simplicial meshes. 
The projection first computes local finite element approximations without imposing interelement conformity, 
and then averages the corresponding degrees of freedom over adjacent elements. 

The new interpolant has the following properties. 
First, it is locally stable in Lebesgue and Sobolev--Slobodeckij norms, uniformly with respect to the local mesh size. 
Second, it is a projection onto the finite element space. 
Third, it satisfies a {piecewise Bramble--Hilbert lemma} in different regularity regimes.
Finally, it can impose homogeneous boundary conditions on a fixed part of the domain boundary. 
Here, traces may be Dirichlet, tangential, or normal traces, depending on which space is considered.

The construction contains several familiar operators as special cases. 
With uniform averaging over all adjacent elements, we recover the Ern--Guermond projection. 
With a single adjacent element selected for each degree of freedom, we obtain a modified Cl\'ement-type interpolant. 
Unlike the classical Cl\'ement interpolant, this modified operator is idempotent. %

The piecewise Bramble--Hilbert lemma enables convergence rates for the finite element approximation of partial differential equations even in low-regularity settings, such as in computational electromagnetism with discontinuous, noisy coefficients or in problems involving domain approximations.

We derive our main result using the calculus of differential forms~\cite{hiptmair2002finite,AFW1,afwgeodecomp}, but we use vector analysis notation in this introduction.

\subsection{Interpolation of vector fields with low regularity}

We describe our projection and error estimates in more detail, using standard notation. 
Assume that $\Omega \subseteq \bbR^{3}$ is a bounded Lipschitz domain. 
Let $L^{2}(\Omega)$ denote the usual Lebesgue space of square-integrable functions, and let $H^{s}(\Omega)$ be the Sobolev--Slobodeckij space with smoothness $s \geq 0$. 
Assume that the domain is equipped with a triangulation $\calT$ and write $\Lagrange_{r}(\calT)$ for the continuous functions that are piecewise polynomials of degree $r$ with respect to the triangulation $\calT$.
Then $\Lagrange_{r}(\calT)$ is a subspace of $H^{1}(\Omega)$. 

In addition, write $H^{s}(\calT)$ for the functions that belong to $H^{s}(T)$ on each individual tetrahedron $T \in \calT$.
We recall that the Scott--Zhang projection $\Projection_{\rm SZ} : H^{1}(\Omega) \to \Lagrange_{r}(\calT)$ is bounded on the first-order Sobolev space and satisfies the following local estimates: 
for any $s \in [0,r+1]$, there exists $C > 0$ such that for any element $T$ of the triangulation and any $u \in H^{1}(\Omega) \cap H^{s}(\calT)$,
\begin{align}\label{intro:math:piecewise-bramble-hilbert}
    \| u - \Projection_{\rm SZ} u \|_{L^{2}(T)} 
    \leq 
    C 
    h^{s}_{T} 
    \sum_{\substack{ T' \in \calT \\ T \cap T' \neq \emptyset }} 
    \| u \|_{H^{s}(T')}
    .
\end{align}
Notably, this local error estimate only relies on global $H^{1}(\Omega)$-regularity and \emph{piecewise} higher Sobolev--Slobodeckij regularity.
This is in stark contrast to error estimates via the nodal interpolant~\cite{ern2017finite} or the Cl\'ement interpolant~\cite{clement1975approximation}, which require global higher regularity. The estimate \eqref{intro:math:piecewise-bramble-hilbert} has been recognized in recent years~\cite{veeser2016approximating,camacho2015L2,licht2021} and is now known as the \emph{piecewise Bramble--Hilbert lemma}. 
It is not present in the original analysis of the Scott--Zhang projection~\cite{scott1990finite}. 
It leads to the equivalence between \emph{global} and \emph{local}, or \emph{conforming} and \emph{non-conforming}, finite element approximation errors. 
It is helpful in situations where we approximate functions whose regularity on subdomains is higher than their global regularity.

The Scott--Zhang projection, alongside the Cl\'ement interpolant, has recently been generalized to the approximation of vector fields~\cite{licht2021}, using the formalism of finite element exterior calculus. 
We review its most important properties in the language of three-dimensional vector analysis. 
Let $\bfL^{2}(\Omega)$ denote the Hilbert space of vector fields with square-integrable coefficients, 
and write 
\begin{gather*}
    \bfH(\curl,\Omega)
    :=
    \left\{ \bfu \in \bfL^{2}(\Omega) \suchthat* \curl \bfu \in \bfL^{2}(\Omega) \right\}
    .
\end{gather*}
We write $\Nedsnd_{r}(\calT)$ for the N\'ed\'elec space of the second kind, i.e., the subspace of vector fields in $\bfH(\curl,\Omega)$ whose members have piecewise polynomial coefficients of degree $r$. 
In addition, $\bfH^{s}(\calT)$ with $s \geq 0$ denotes the Hilbert space of vector fields whose coefficients are in $H^{s}(\calT)$.
The curl-conforming Scott--Zhang projection $\Projection_{\rm SZ}^{\curl} : \bfH(\curl,\Omega) \to \Nedsnd_{r}(\calT)$ is bounded and satisfies the following estimate:
given $s \in [0,r+1]$ and $t \in [0,r]$, there exists $C > 0$ such that for any element $T$ of the triangulation and any $\bfu \in \bfH(\curl,\Omega)$ with $\bfu \in \bfH^{s}(\calT)$ and $\curl \bfu \in \bfH^{t}(\calT)$, 
\begin{align}\label{intro:math:piecewise-curl-bramble-hilbert}
    \| \bfu - \Projection_{\rm SZ}^{\curl} \bfu \|_{\bfL^{2}(T)} 
    \leq 
    C 
    \sum_{\substack{ T' \in \calT \\ T \cap T' \neq \emptyset }} 
    h^{s}_{T'} \| \bfu \|_{\bfH^{s}(T')} + h^{t+1}_{T'} \| \curl\bfu \|_{\bfH^{t}(T')}
    .
\end{align}
The estimate only requires global $\bfH(\curl,\Omega)$-regularity, and all higher Sobolev--Slobodeckij regularity is merely assumed piecewise. 
Again, this piecewise Bramble--Hilbert lemma establishes the equivalence of \emph{global} and \emph{local} finite element approximations.
Analogous results hold for divergence-conforming spaces.

This manuscript contributes an alternative to the Scott--Zhang projection, with similar properties but defined already on Lebesgue spaces. 
We define the averaging-based projection $\Projection^{\curl} : \bfL^{2}(\Omega) \to \Nedsnd_{r}(\calT)$ as the composition of two steps.
First, we project the vector field onto a piecewise finite element approximation, without imposing any continuity whatsoever between elements.
Second, we define a conforming finite element interpolant by averaging the degrees of freedom:
each degree of freedom of the conforming space is a weighted average of that degree of freedom over all adjacent piecewise approximations from the first step.
It is easily seen that the composition of these two stages is a bounded operator from the Lebesgue space onto a conforming finite element space. 
Moreover, the construction is idempotent and local.

We distinguish between three different types of error estimates, each for a different global regularity regime. 
First, we consider the \emph{Sobolev trace regime}, where scalar and vector fields have sufficient global Sobolev regularity to define traces in the classical Sobolev sense. 
We obtain a piecewise Bramble--Hilbert lemma of the following form:
given $m\geq 0$ with $m > \tfrac 1 2$ and $s \in [0,r+1]$, there exists $C > 0$ such that for any element $T$ of the triangulation and any $\bfu \in \bfH^{m}(\Omega) \cap \bfH^{s}(\calT)$, 
\begin{align}\label{intro:math:sobolev-trace-estimate}
    \| \bfu - \Projection^{\curl} \bfu \|_{\bfL^{2}(T)} 
    \leq 
    C 
    h^{s}_{T} 
    \sum_{\substack{ T' \in \calT \\ T \cap T' \neq \emptyset }} 
    \| \bfu \|_{\bfH^{s}(T')}
    .
\end{align}
Second, we consider a \emph{weak trace regime}, in which traces of vector fields are also defined using integration by parts without having to resort to Sobolev trace theory. 
This leads to the following estimate, which is a piecewise Bramble--Hilbert lemma:
given $s \in [0,r+1]$ and $t \in [0,r]$, there exists $C > 0$ such that for any element $T$ of the triangulation and any $\bfu \in \bfH(\curl,\Omega)$ with $\bfu \in \bfH^{s}(\calT)$ and $\curl \bfu \in \bfH^{t}(\calT)$, 
\begin{align}\label{intro:math:curl-conforming-estimate}
    \| \bfu - \Projection^{\curl} \bfu \|_{\bfL^{2}(T)} 
    \leq 
    C 
    \sum_{\substack{ T' \in \calT \\ T \cap T' \neq \emptyset }} 
    h^{s}_{T'} \| \bfu \|_{\bfH^{s}(T')} + h^{t+1}_{T'} \| \curl\bfu \|_{\bfH^{t}(T')}
    .
\end{align}
Third, we consider the global regularity regime. Here, we permit lower regularity, hence traces might be unavailable.
In this regime, any smoothness must be global rather than just piecewise, but it can be arbitrarily low: 
for $s \in [0,r+1]$, there exists $C > 0$ such that for any element $T$ of the triangulation and any $\bfu \in \bfH^{s}(\Omega)$, 
\begin{align}\label{intro:math:global-estimate}
    \| \bfu - \Projection^{\curl} \bfu \|_{\bfL^{2}(T)} 
    \leq 
    C 
    h^{s}_{T} 
    \| \bfu \|_{\bfH^{s}(\Omega_{T})}
    ,
\end{align}
where $\Omega_{T}$ is the extended neighborhood around $T$. 

The first two types of estimates for the averaging-based projection provide approximation error estimates in accordance with the piecewise regularity once some modest global regularity is established. 
In the third type of regime, the extra Sobolev--Slobodeckij regularity can be arbitrarily low but must be global, and we get an error estimate analogous to the Cl\'ement interpolant. 
The weak trace regime is the most natural setting for vector analysis but differs qualitatively from the Sobolev trace regime. 
This was the main incentive for the Scott--Zhang projections on differential forms~\cite{licht2021}. The present manuscript continues that endeavor. 

At the center of our analysis are \emph{rough degrees of freedom}, 
which are representations of the finite element degrees of freedom by integrals over elements and facets.
Rough degrees of freedom represented as boundary integrals over facets lead to estimates for classical Sobolev--Slobodeckij spaces.
These are common in the construction of quasi-interpolation operators on non-smooth fields~\cite{brennerscott}. 
For example, they are already present in the work of Scott and Zhang and in the construction by Ern and Guermond. 
The main difference from these works is that we employ rough degrees of freedom represented as volume integrals, based on an integration-by-parts formula, which enables error estimates on rough spaces of vector fields such as $\bfH(\curl,\Omega)$ instead of the Sobolev scale.

\subsection{Boundary conditions}

Boundary conditions in vector analysis have features not present in the scalar-valued setting.
Tangential and normal boundary conditions in vector analysis may be defined not only via Sobolev trace theory but also via a generalized integration-by-parts formula. This is also known as the distinction between strong and weak boundary traces~\cite{mitrea2022geometricI,mitrea2022geometricII}.
The standard Sobolev trace theory does not cover the boundary conditions in $\bfH(\curl,\Omega)$ and $\bfH(\divergence,\Omega)$ that are defined via an integration-by-parts formula. 
Generalizing interpolants and their error estimates to vector-valued finite element spaces must accommodate the weak trace theory.

Our averaging-based projection is bounded on Lebesgue spaces, and homogeneous boundary conditions are \enquote{hard-coded}: 
the finite element interpolation always satisfies the homogeneous boundary conditions on a fixed boundary part and the approximation error estimate is only valid for fields already satisfying these homogeneous boundary conditions.
This is a difference from the generalized Scott--Zhang projection, which requires regularity $\bfH(\curl,\Omega)$ and $\bfH(\divergence,\Omega)$ for traces to be well-defined.
The Scott--Zhang projection allows a more subtle treatment of boundary conditions and approximates the boundary values on any fixed boundary part. 
If the boundary values of the original field are zero, then the same is true for the Scott--Zhang projection, and its error estimates hold for fields that are sufficiently regular.

The general principle is that finite element quasi-interpolants that are defined on Lebesgue spaces, since they cannot inspect boundary values, can respect homogeneous boundary conditions only if these are \enquote{hard-coded} in the definition of the quasi-interpolant. Likewise, the approximation error estimate can then only hold when the approximation target satisfies the boundary condition. 
By contrast, if we want the quasi-interpolant to approximate inhomogeneous boundary conditions and satisfy approximation error estimates for targets that do not satisfy the boundary condition, then the quasi-interpolant can only be defined on spaces that are regular enough to admit traces.

\subsection{Special cases of the construction}

The averaging-based projection takes weighted averages of the degrees of freedom over each adjacent element. 
There is some latitude in the choice of weights. 
For example, if we always take the equally weighted average over all neighboring elements, we essentially obtain the Ern--Guermond projection as a special case. 
By contrast, if we take the degree of freedom from only one single neighboring element, then we may interpret this special case as a modified Cl\'ement interpolant. 

This modified Cl\'ement interpolant is of interest in its own right, even when only the interpolation of scalar functions is considered.
Recall that the original Cl\'ement interpolant, like its generalization to vector fields and differential forms, evaluates the degrees of freedom on patchwise projections onto polynomial spaces.
Instead, our modified Cl\'ement interpolant evaluates the degrees of freedom on the elementwise projections from a single neighboring element. 
We thus simplify the definition of the interpolant at the expense of a slightly more complicated analysis. 
In addition, the modified Cl\'ement interpolant is a projection and it satisfies a piecewise Bramble--Hilbert lemma. 
Recall that the original Cl\'ement interpolant admits a simpler error analysis but relies on higher Sobolev--Slobodeckij regularity on entire patches to establish a Bramble--Hilbert-like error estimate. Moreover, it is not idempotent. 

The error estimates for the classical Cl\'ement interpolant require polynomial approximation error estimates on subdomains.
Such local approximation error estimates might deteriorate close to the boundary, for example near re-entrant corners with sharp angles.
But the averaging-based projections, including our modified Cl\'ement interpolant, are not subject to this: 
the constants in the inequalities~\eqref{intro:math:sobolev-trace-estimate}, {}~\eqref{intro:math:curl-conforming-estimate}, and~\eqref{intro:math:global-estimate} depend only on the shape regularity of the elements.
This enables inequalities with stable constants for shape-regular triangulations, even when the domain is rough.
In practical terms, the constants in the inequalities are robust near re-entrant corners, which typically allow shape-regular meshes.

\subsection{Relation to previous work}

We review finite element interpolation operators and place this research into a wider context. 
The a~priori error analysis of finite element methods leads to quasi-optimal error estimates,
but this alone does not yet establish convergence rates in terms of the mesh size. 
The familiar convergence rates are established by estimating the approximation errors of finite element interpolators.

The standard analysis of the approximation error uses the canonical interpolant, defined via the degrees of freedom.
This analysis is classical for the nodal interpolant onto Lagrange elements, and analogues for curl- or divergence-conforming finite element spaces have long been known~\cite[Chapter III.3]{brezzi_fortin1991}, \cite{hiptmair2002finite}, \cite[Chapter 5]{monk2003finite}, ~\cite{ciarlet2016approximation}. 
However, the canonical interpolant is only defined when the function to be interpolated is sufficiently smooth:
then traces on lower-dimensional simplices are defined, so that, e.g., point evaluations or edge integrals are meaningful. 
Even if the smoothness permits the use of the canonical interpolants, the stability and error analysis of these interpolants may involve Sobolev embedding inequalities whose constants are hard to control across the smoothness scale. Indeed, the Sobolev embedding constants can be arbitrarily large if the regularity is low.
This forecloses their application in low-regularity situations that matter in three-dimensional vector analysis and electromagnetism~\cite{costabel1991coercive,amrouche1998vector}.

Interpolants that require less regularity are well-known for scalar finite elements.
The Cl\'ement interpolant~\cite{clement1975approximation} is defined even on functions in Lebesgue spaces, and it enables localized interpolation estimates. The Cl\'ement interpolant admits a simple definition and analysis, and it is widely used in the literature~\cite{apel1999interpolation,carstensen1999quasi,carstensen2000constants,melenk2005hp}. 
Notably, though, it is not a projection.
The interpolant by Scott and Zhang~\cite{scott1990finite} is defined on functions with some modest regularity, for instance on Sobolev or Sobolev--Slobodeckij spaces admitting traces~\cite{ciarlet2013analysis}, and unlike the classical Cl\'ement interpolant, it is a projection. 
The Scott--Zhang interpolant also admits a more delicate approximation of boundary values and was the first interpolation operator that facilitates a piecewise Bramble--Hilbert lemma~\cite{veeser2016approximating}. 
These operators can be seen as instances of a general construction principle, where one defines the operator in terms of rough degrees of freedom~\cite[Chapter 4.8]{brennerscott}.
Such interpolation operators on non-smooth functions are also known as \emph{quasi-interpolants}.

Ern and Guermond~\cite{ern2017finite} were the first to define a comparable non-smooth quasi-interpolation operator for curl- or divergence-conforming finite element spaces of vector fields. Their projection is bounded on Lebesgue spaces. Similar to the literature on the Scott--Zhang projection~\cite{scott1990finite,girault2003quasi,ciarlet2013analysis}, their discussion of approximation errors and boundary conditions relies on Sobolev trace theory. The projection by Ern and Guermond averages the degrees of freedom of local polynomial projections and extends earlier work by Oswald on scalar functions~\cite{oswald1993bpx}.

Recent research has contributed new quasi-interpolation operators for the vector field spaces $\bfH(\curl,\Omega)$ and $\bfH(\divergence,\Omega)$, generalizing many results known for scalar fields to vector fields. 
Importantly, this regularity regime differs from the classical Sobolev--Slobodeckij scale. 
Quasi-interpolants within the framework of finite element exterior calculus~\cite{licht2021} include the Cl\'ement and Scott--Zhang interpolants for vector-valued finite element spaces. 
The generalized Cl\'ement interpolant is defined on Lebesgue spaces of vector fields and satisfies approximation error estimates completely analogous to the scalar setting. 
The generalized Scott--Zhang projection is defined on the vector field spaces $\bfH(\curl,\Omega)$ and $\bfH(\divergence,\Omega)$. In contrast to the averaging-based projection by Ern and Guermond, the generalized Scott--Zhang projection relies on the representation of the degrees of freedom through integration-by-parts formulas. The error analysis of this operator was the first to state a piecewise Bramble--Hilbert lemma for $\bfH(\curl,\Omega)$ and $\bfH(\divergence,\Omega)$. The present manuscript continues that line of work.

Roughly speaking, the piecewise Bramble--Hilbert lemma states that any function in $H^{1}(\Omega)$ can be approximated by conforming finite elements just as well as by non-conforming (i.e., discontinuous) finite elements. This is also known as the equivalence of the global and local best-approximation errors.
While the piecewise Bramble--Hilbert lemma can be found implicitly in the original paper by Scott and Zhang, it has attracted explicit attention in recent work, first by Veeser~\cite{veeser2016approximating} and then by Camacho and Demlow~\cite{camacho2015L2}, who prove the estimate via an error analysis of the Scott--Zhang projection. 
The generalized Scott--Zhang projections are only defined on $\bfH(\curl,\Omega)$ and $\bfH(\divergence,\Omega)$, just like their scalar-valued counterpart is only defined on sufficiently smooth Sobolev spaces~\cite{ciarlet2013analysis}. 
Likewise, they generalize the piecewise Bramble--Hilbert lemma to curl- or divergence-conforming spaces. %
This makes the piecewise Bramble--Hilbert lemma available in finite element exterior calculus and finite element vector calculus. 
In contrast, this manuscript contributes an approximation operator that ensures a piecewise Bramble--Hilbert lemma but is defined on Lebesgue spaces.

A different line of work has brought forth a sequence of projections onto curl- or divergence-conforming spaces~\cite{ern2022equivalence,chaumontfrelet2021equivalence,chaumontfrelet2024stable}.
These projectors are defined on $H^{1}(\Omega)$, $\bfH(\curl,\Omega)$, and $\bfH(\divergence,\Omega)$, respectively. They preserve fixed homogeneous partial boundary conditions, are locally stable, commute with the differential operators, and enable the piecewise Bramble--Hilbert lemma.

We emphasize that our projection does not commute with the exterior derivative. 
Most of the commuting interpolants and projections discussed in the context of finite element de~Rham complexes {}~\cite{demkowicz2005h1,AFW1,christiansen2011topics,christiansen2008smoothed,falk2014local,ern2016mollification,licht2019smoothed,licht2019mixed,ern2022equivalence,chaumontfrelet2024stable,chaumontfrelet2026quasi} serve as generalized Fortin interpolants. 
Their existence implies stability and a~priori error estimates for mixed finite element methods,
which show that finite element solutions satisfy quasi-optimal error estimates.
However, this does not yet establish error estimates and convergence rates in terms of the mesh size. 
Our projection does not commute with the differential operators and is not a Fortin interpolant. 
Instead, it serves a complementary purpose, namely, to establish specific convergence rates of the best-approximation errors in terms of the mesh size.

\subsection{Applications}

Systems in mathematical electromagnetism with discontinuous electric permittivity, discontinuous magnetic permeability, and mixed boundary conditions may lead to electric and magnetic fields with very low Sobolev--Slobodeckij regularity~\cite{costabel2000singularities}. However, the fields are in the curl-conforming space $\bfH(\curl,\Omega)$ and may have higher regularity on subregions~\cite{costabel1999singularities}. 
If the mesh is fitted to those subregions, then the piecewise Bramble--Hilbert lemma can recover convergence rates based on the regularity within subregions, even if the global regularity is much lower.

The piecewise Bramble--Hilbert lemma has utility in situations where the domain is approximated. 
Suppose a scalar function on a \emph{physical domain} is approximated in a finite element space on a triangulated \emph{parametric domain}. We need a transformation between the physical and the parametric domains to compare the original function with any finite element approximation. In practice, such transformations are bi-Lipschitz; see also Figure~\ref{fig:geometry}. 
But then we face a dilemma: on the one hand, transforming any finite element approximation onto the physical domain generally does not preserve polynomials; on the other hand, transforming the original function onto the parametric domain generally does not preserve any higher global regularity beyond mere conformity with the differential operators.
Neither of the two approaches admits the familiar piecewise polynomial interpolation estimates for finite element approximations
that rely on global Sobolev--Slobodeckij regularity, e.g., through the canonical interpolant or the Scott--Zhang projection. 

Situations like these arise in finite element error analysis on manifolds, surfaces, and domains with non-polyhedral boundary, irrespective of whether the transformation is explicitly known or implicitly assumed.
What resolves the aforementioned dilemma is that, in practice, the transformations are piecewise smooth and thus preserve the original regularity \emph{piecewise}.
The averaging-based interpolant enables the piecewise Bramble--Hilbert lemma that exploits the piecewise regularity of the transformed solutions. 

\begin{figure}
    \centering
    \begin{tabular}{cc}
        \begin{tikzpicture}[scale=2.5]
            \coordinate (P0) at (  0.0,  0.0, 0.0 );
            \coordinate (P1) at (  0.5,  0.0, 0.0 );
            \coordinate (P2) at (  0.0,  0.5, 0.0 );
            \coordinate (P3) at (  1.0,  0.0, 0.0 );
            \coordinate (P4) at (  0.0,  1.0, 0.0 );
            \coordinate (P5) at (  0.5,  0.5, 0.0 );
            \filldraw[fill={rgb:black,1;white,9}, draw=black] (P0) -- (P1) -- (P2) -- cycle;
            \filldraw[fill={rgb:black,1;white,9}, draw=black] (P1) -- (P2) -- (P5) -- cycle;
            \filldraw[fill={rgb:black,1;white,9}, draw=black] (P1) -- (P3) -- (P5) -- cycle;
            \filldraw[fill={rgb:black,1;white,9}, draw=black] (P2) -- (P4) -- (P5) -- cycle;
            \coordinate (Q0) at (  -0.0,  0.0, 0.0 );
            \coordinate (Q1) at (  -0.5,  0.0, 0.0 );
            \coordinate (Q2) at (  -0.0,  0.5, 0.0 );
            \coordinate (Q3) at (  -1.0,  0.0, 0.0 );
            \coordinate (Q4) at (  -0.0,  1.0, 0.0 );
            \coordinate (Q5) at (  -0.5,  0.5, 0.0 );
            \filldraw[fill={rgb:black,1;white,9}, draw=black] (Q0) -- (Q1) -- (Q2) -- cycle;
            \filldraw[fill={rgb:black,1;white,9}, draw=black] (Q1) -- (Q2) -- (Q5) -- cycle;
            \filldraw[fill={rgb:black,1;white,9}, draw=black] (Q1) -- (Q3) -- (Q5) -- cycle;
            \filldraw[fill={rgb:black,1;white,9}, draw=black] (Q2) -- (Q4) -- (Q5) -- cycle;
        \end{tikzpicture}
        \begin{tikzpicture}[scale=2.5]
            \filldraw[fill={rgb:black,1;white,9}, draw=black] (0,0) -- ( 1.0,0) arc (0: 90: 1.0) -- (0,0);
            \filldraw[fill={rgb:black,1;white,9}, draw=black] (0,0) -- (-1.0,0) arc (0:-90:-1.0) -- (0,0);
            
            \draw[draw=black]
                (0.5,0)
                .. controls (  0.5/2 + 0.70710678118/2 + 0.1, 0.70710678118/2 - 0.1 ) ..
                (0.70710678118,0.70710678118)
                .. controls (  0.70710678118/2 - 0.1, 0.5/2 + 0.70710678118/2 + 0.1) ..
                (0,0.5)
                .. controls ( -0.70710678118/2 + 0.1, 0.5/2 + 0.70710678118/2 + 0.1) ..
                (-0.70710678118,0.70710678118)
                .. controls ( -0.5/2 - 0.70710678118/2 - 0.1, 0.70710678118/2 - 0.1 ) ..
                (-0.5,0)
                ;
            \draw[draw=black]
                (0.5,0)
                .. controls ( 0.25+0.1,0.25+0.1) ..
                (0,0.5)
                .. controls (-0.25-0.1,0.25+0.1) ..
                (-0.5,0)
                ;

        \end{tikzpicture}
    \end{tabular} 
    \caption{
    A triangulated parametric domain (left) and a physical domain (right). The former is mapped onto the latter by a bi-Lipschitz piecewise smooth transformation. The image of the triangulation is also drawn inside the physical domain.
    }
    \label{fig:geometry}
\end{figure}

\subsection{Extension operators}

Universal extension operators for Lipschitz domains extend scalar functions to $\bbR^{n}$ while preserving Sobolev--Slobodeckij spaces~\cite{stein1970singular,rychkov1999restrictions}. 
There are extension operators that preserve membership in $\bfH(\curl,\Omega)$ or $\bfH(\divergence,\Omega)$ in addition to regularity in classical integral-order Sobolev classes~\cite{hiptmair2012universal}.
Here, the discussion has remained focused on Hilbert spaces. 
We contribute an alternative extension operator, defined for integrability exponents $1 < p < \infty$, which preserves not only membership in curl- or divergence-conforming spaces but also any Sobolev--Slobodeckij regularity with smoothness parameters $s \in [0,1]$. 
We only use it for an auxiliary result that relates Sobolev and weak traces, but this new extension operator appears to be of independent merit.

\subsection{Outline of the manuscript}

The introduction up to this point has addressed our results in the language of vector analysis. However, the remainder of the manuscript adopts the calculus of differential forms and the framework of finite element exterior calculus~\cite{AFW1,AFW2,hiptmair2002finite}. Only at the end do we return to the language of vector analysis to display our main results.

The remainder of this manuscript is structured as follows.
In Section~\ref{section:background}, we review background on triangulations, function spaces, exterior calculus, and finite element spaces.
In Section~\ref{section:roughdof}, we discuss biorthogonal systems of finite element bases and their degrees of freedom and introduce some notational conventions for the rest of the manuscript.
In Section~\ref{section:averaging-based-projection}, we construct the averaging-based projection, our main result.
In Section~\ref{section:estimates}, we develop stability and approximation error estimates.
In Section~\ref{section:equivalence-of-global-and-local}, we use the projection to compare global and local approximation errors.
Finally, we review applications of our results in the language of vector analysis in Section~\ref{section:anwendungen}.
In the appendix, Section~\ref{section:extension} discusses an extension operator for Sobolev differential forms on Lipschitz domains.

\section{Background} \label{section:background}

In this section, we review background on triangulations, function spaces, differential forms, and finite element de~Rham complexes.
We also establish notation, generally following the common conventions in the literature. 
Much of the content of this section summarizes the background in~\cite{licht2021}.
However, the discussion of low-regularity spaces and homogeneous partial boundary conditions goes beyond standard material. 
We let $\Omega \subseteq \bbR^{n}$ be an open set throughout this section; additional assumptions on $\Omega$ will be specified later.

\subsection{Triangulations} \label{subsection:background:triangulation}

A \emph{simplex} of dimension $d$ is the convex hull of $d+1$ affinely independent points, called its \emph{vertices}. 
Thus vertices, line segments, triangles, and tetrahedra are simplices of dimensions $0$, $1$, $2$, and $3$, respectively.
If $T$ is any simplex, then another simplex $F$ is called a \emph{subsimplex} of $T$ if all vertices of $F$ are vertices of $T$. 
We write $\subsimplices(T)$ for the set of all subsimplices of a simplex $T$ and $\subsimplices_{m}(T)$ for the set of its $m$-dimensional subsimplices.
We follow the convention of polyhedral theory and reserve the term \emph{facet} for the codimension-one subsimplices of a given simplex,
i.e., the facets of the $d$-simplex $T$ are the simplices in $\subsimplices_{d-1}(T)$.

A \emph{simplicial complex} $\calT$ is a finite set of simplices
such that $\subsimplices(T) \subseteq \calT$ for all $T \in \calT$
and, for any two simplices $T, T' \in \calT$ with non-empty intersection, $T \cap T'$ is a common subsimplex of $T$ and $T'$.
A \emph{simplicial subcomplex} of $\calT$ is any subset $\calU \subseteq \calT$ that is itself a simplicial complex.
We write $\subsimplices_{d}(\calT)$ for the set of $d$-dimensional simplices of $\calT$.
Given any simplex $T \in \calT$, we write $\supersimplices_{d}(\calT,T)$ for the set of $d$-simplices in $\calT$ that contain $T$:
\begin{align*}
    \supersimplices_{d}(\calT,T) := \left\{ S \in \subsimplices_{d}(\calT) \suchthat* T \subseteq S \right\}.
\end{align*}
We say that $\calT$ is $d$-dimensional\footnote{What we call a $d$-dimensional simplicial complex is a $d$-dimensional \emph{pure} simplicial complex in the terminology of topology.} if all its simplices are subsets in a simplex of $\subsimplices_{d}(\calT)$. 

We say that $\calT$ triangulates a closed set $A \subseteq \bbR^{n}$ if 
\begin{align*}
    A = \bigcup_{ T \in \calT } T.
\end{align*}
As is common in the finite element literature, we say that a finite simplicial complex $\calT$ triangulates an open set $\Omega \subseteq \bbR^{n}$ 
if $\Omega = \interior{\closure{\Omega}}$ and $\calT$ triangulates $\closure{\Omega}$.

\begin{remark}\label{remark:domain-triangulation} 
    The terminology for triangulations significantly differs between the finite element literature and the topology literature. %
    This discrepancy is largely technical, but it appears to remain largely unaddressed.
    In the terminology of geometry and topology, any finite triangulation can triangulate only bounded closed sets. 
    But in the terminology of the finite element literature, 
    a triangulation $\calT$ of an open set $\Omega \subseteq \bbR^{n}$ refers to a triangulation $\calT$ of the closure $\closure{\Omega}$.
    This is practical in most applications, where $\Omega$ is a \emph{regularly open} set~\cite[Chapter~2]{willard2004}, which means that $\Omega = \interior{\closure{\Omega}}$.
    For example, this condition excludes slit domains or domains with isolated point holes. 
\end{remark}

Whenever $T$ is a simplex of positive dimension $d > 0$, 
we write $h_{T}$ for its diameter and $\vol^{d}(T)$ for its $d$-dimensional volume (in the sense of Hausdorff measure).
Formally, for a vertex $V \in \calT$, we define its \enquote{diameter} $h_V$ as the maximum diameter of all simplices in $\calT$ that contain $V$, and we set $\vol^{0}(V) := 1$.
The \emph{shape measure} of any $d$-dimensional simplex $T$ is defined as 
\[
    \shapemeasure(T) := h_{T}^{d} / \vol^{d}(T).
\]
We define 
\begin{align*}
    \shapemeasure(\calT) := \sup_{T \in \calT} \shapemeasure(T),
    \quad 
    \comparisonconstant(\calT) := \sup_{ \substack{ S, S' \in \calT \\ S \cap S' \neq \emptyset } } \frac{ h_{S} }{ h_{S'} },
    \quad 
    \cardinalitybound(\calT) := \sup_{S \in \calT} \cardinality\left\{ S' \in \calT \suchthat* S \cap S' \neq \emptyset \right\}.
\end{align*}
We call $\shapemeasure(\calT)$ the \emph{shape measure} of $\calT$. 
Uniform boundedness of $\shapemeasure(\calT)$ is the usual shape-regularity condition for simplicial meshes.
While $\comparisonconstant(\calT) > 0$ bounds the ratio of the diameters of adjacent\footnote{Two simplices are \emph{adjacent} if their intersection is non-empty.} simplices, $\cardinalitybound(\calT) > 0$ bounds the maximum number of simplices adjacent to any given simplex in $\calT$.\footnote{One can show that, in the case of domain triangulations, $\comparisonconstant(\calT)$ and $\cardinalitybound(\calT)$ are bounded in terms of the shape measure $\shapemeasure(\calT)$.}

Lastly, we assume that all simplices are equipped with an arbitrary but fixed orientation.
Whenever $T$ is a $d$-simplex and $F \in \subsimplices_{d-1}(T)$ is one of its facets, 
we let $o(F,T) := 1$ if the orientation of $F$ is induced from $T$, and we let $o(F,T) := -1$ otherwise.

\subsection{Function spaces} \label{subsection:background:functionspaces}

We recall several function spaces, in particular the Banach spaces known as Sobolev and Sobolev--Slobodeckij spaces~\cite{slobodeckij1958generalized,burenkov1998sobolev,di2012hitchhiker,heuer2014equivalence}.

The space of smooth functions on $\Omega$ is denoted $C^{\infty}(\Omega)$.
Let $C^{\infty}(\closure{\Omega})$ be the space of restrictions of smooth functions on $\bbR^{n}$ to the set $\closure{\Omega}$.
We write $L^{p}(\Omega)$ for the Lebesgue space on $\Omega$ with integrability exponent $p \in [1,\infty]$, equipped with the norm ${\|\cdot\|}_{L^{p}(\Omega)}$.
For any $\theta \in (0,1)$, we define the seminorm
\begin{gather*}
    |\ding|_{W^{\theta,p}(\Omega)}
    :=
    \left\|
    | \ding(x) - \ding(y) | \cdot | x - y |^{ - \theta - \frac{n}{p} }
    \right\|_{L^{p}(\Omega \times \Omega)}
    ,
\end{gather*}
and we let $W^{\theta,p}(\Omega)$ be the subspace of $L^{p}(\Omega)$ for which that seminorm is finite. 

Write $A(n) := \bbN_{0}^{n}$ for the set of all multi-indices over $\{1,\dots,n\}$.
For any $k \in \bbN_0$,
let $W^{k,p}(\Omega)$ be the \emph{Sobolev space} of locally integrable functions on $\Omega$
for which all distributional $\alpha$-th derivatives with $\alpha \in A(n)$ and $|\alpha| \leq k$ are functions in $L^{p}(\Omega)$.
In particular, $W^{0,p}(\Omega) = L^{p}(\Omega)$.
On the space $W^{k,p}(\Omega)$, we define the seminorms and norms
\begin{gather*}
    | \ding |_{W^{k,p}(\Omega)}
    :=
    \sum_{ \substack{ \alpha \in A(n) \\ |\alpha| = k } }
    \| \partial^{\alpha} \ding \|_{L^{p}(\Omega)}
    ,
    \quad
    \| \ding \|_{W^{k,p}(\Omega)}
    :=
    \sum_{l = 0}^{k} | \ding |_{W^{l,p}(\Omega)}
    .
\end{gather*}
For $k \in \bbN_{0}$ and $\theta \in (0,1)$, the \emph{Sobolev--Slobodeckij space} $W^{k+\theta,p}(\Omega)$ is defined as the subspace of $W^{k,p}(\Omega)$ whose members have distributional derivatives of order $k$ belonging to $W^{\theta,p}(\Omega)$.
On the space $W^{k+\theta,p}(\Omega)$, we define the seminorms and norms
\begin{gather*}
    | \ding |_{W^{k+\theta,p}(\Omega)}
    :=
    \sum_{ \substack{ \alpha \in A(n) \\ |\alpha| = k } }
    | \partial^{\alpha} \ding |_{W^{\theta,p}(\Omega)}
    ,
    \\
    \|\ding\|_{W^{k+\theta,p}(\Omega)}
    :=
    \|\ding\|_{W^{k,p}(\Omega)}
    +
    |\ding|_{W^{k+\theta,p}(\Omega)}
    .
\end{gather*}
For every $m \in [0,\infty)$ and $p \in [1,\infty]$, 
we thus find that $W^{m,p}(\Omega)$ is a Banach space with norm $\|\cdot\|_{W^{m,p}(\Omega)}$. 

We will also use spaces of functions with piecewise higher Sobolev regularity. 
If $\calT$ is a triangulation of the domain $\Omega$, then for any $s \in [0,\infty)$ and $p \in [1,\infty]$ we define 
\begin{align*}
    W^{s,p}(\calT) 
    := 
    \left\{ \ding \in L^{p}(\Omega) \suchthat* \forall T \in \subsimplices_{n}(\calT) : \ding \restriction_{\interior{T}} \in W^{s,p}(\interior{T}) \right\}
    .
\end{align*}
We have $W^{s,p}(\Omega) \subseteq W^{s,p}(\calT) \subseteq L^{p}(\Omega)$, but we emphasize that piecewise Sobolev regularity generally does not imply global Sobolev regularity. 

\begin{remark}
    The Sobolev--Slobodeckij spaces are also known as fractional Sobolev spaces, and they can be defined in various ways, for example via interpolation theory or Fourier analysis. 
    The definition via the seminorm $|\cdot|_{W^{\theta,p}(\Omega)}$ is the most convenient one for our purposes.
\end{remark}

\subsection{Spaces of differential forms} \label{subsection:background:differentialforms}

We review spaces of differential forms on domains and simplices.
Since $0$-forms are functions, this generalizes the definitions in the preceding subsection.
We let $C^{\infty}\Lambda^{k}(\closure{\Omega})$ be the space of differential $k$-forms with coefficients in $C^{\infty}(\closure{\Omega})$, 
where $k$ is any integer. 
\footnote{Here and in what follows, we adhere to the common convention that spaces of differential $k$-forms equal zero unless $0 \leq k \leq n$.}

The Lebesgue spaces $L^{p}\Lambda^{k}(\Omega)$ and the Sobolev--Slobodeckij spaces $W^{m,p}\Lambda^{k}(\Omega)$ are defined accordingly for any $p \in [1,\infty]$ and $m \in [0,\infty)$,
and we let ${\|\cdot\|}_{L^{p}\Lambda^{k}(\Omega)}$, ${\|\cdot\|}_{W^{m,p}\Lambda^{k}(\Omega)}$, and ${|\cdot|}_{W^{m,p}\Lambda^{k}(\Omega)}$
denote the associated norms and seminorms.
In these Banach spaces, the norm is defined as the sum of the corresponding coefficient norms. 

We also recall the exterior derivative $\exteriorderivative : C^{\infty}\Lambda^{k}(\closure{\Omega}) \rightarrow C^{\infty}\Lambda^{k+1}(\closure{\Omega})$, which maps $k$-forms to $(k+1)$-forms.
The weak exterior derivative is a generalization of the classical exterior derivative, which is defined for smooth differential forms. 
We consider classes of differential $k$-forms with coefficients in Lebesgue spaces whose exterior derivative, a~priori defined only in the sense of distributions, lies again in a Lebesgue space.
For $p,q \in [1,\infty]$, we define 
\begin{align*}
    \Sobdiff^{p,q}\Lambda^{k}(\Omega)
    :=
    \left\{\;
        \ding \in L^{p}\Lambda^{k}(\Omega)
        \suchthat*
        \exteriorderivative\ding \in L^{q}\Lambda^{k+1}(\Omega)
    \;\right\}
    .
\end{align*}
The exterior derivative of any $\ding \in \Sobdiff^{p,q}\Lambda^{k}(\Omega)$ exists in the sense of distributions, 
and we then call $\exteriorderivative \ding$ the \emph{weak exterior derivative}. 
Note that $\Sobdiff^{p,q}\Lambda^{k}(\Omega)$ is a Banach space with the graph norm
\[
    \|\ding\|_{\Sobdiff^{p,q}\Lambda^{k}(\Omega)}
    :=
    \|\ding\|_{L^{p}\Lambda^{k}(\Omega)}
    +
    \|\exteriorderivative\ding\|_{L^{q}\Lambda^{k+1}(\Omega)}.
\]
Recall that the exterior product $\ding \wedge \phi$ of a $k$-form $\ding$ and an $l$-form $\phi$ satisfies $\ding \wedge \phi = (-1)^{kl} \phi \wedge \ding$ and is bilinear.
The following characterization of weak exterior derivatives is insightful. 
We have $\ding \in \Sobdiff^{p,q}\Lambda^{k}(\Omega)$
if and only if $\ding \in L^{p}\Lambda^{k}(\Omega)$ and there exists $\dong \in L^{q}\Lambda^{k+1}(\Omega)$ such that
\begin{align*}
    \int_{\Omega} \ding \wedge \exteriorderivative \phi = (-1)^{k+1} \int_{\Omega} \dong \wedge \phi
\end{align*}
for all $\phi \in C^{\infty}_{c}\Lambda^{n-k-1}(\Omega)$, and in that case $\exteriorderivative\ding = \dong$.

\begin{remark}
    The Sobolev--Slobodeckij spaces $W^{m,p}\Lambda^{k}(\Omega)$ and the spaces $\Sobdiff^{p,q}\Lambda^{k}(\Omega)$ are two different ways to define low-regularity spaces of differential forms.
    These two classes are not comparable in general, and they represent independent axes of regularity. 
    
    The normed space $\Sobdiff^{2,2}\Lambda^{k}(\Omega)$ corresponds to the Hilbert space $H\Lambda^{k}(\Omega)$, which has been at the center of attention in the literature on finite element exterior calculus~\cite{hiptmair2002finite,AFW1,AFW2}.
\end{remark}

A helpful integration-by-parts identity is the following; see~\cite[Lemma~3.1]{goldshtejn1982differential}.
If $\ding \in \Sobdiff^{p,q}\Lambda^{k}(\Omega)$ and $\phi \in \Sobdiff^{\infty,\infty}\Lambda^{n-k-1}(\bbR^{n})$ with compact support contained in $\Omega$, then
\begin{align}\label{math:integration-by-parts-against-flat-test-forms}
    \int_{\Omega} {\exteriorderivative \ding} \wedge \phi
    =
    (-1)^{k+1}
    \int_{\Omega} \ding \wedge \exteriorderivative \phi
    . 
\end{align}

\subsection{Strong boundary traces of differential forms}

We review boundary traces of differential forms by means of Sobolev trace theory. 
For reasons that will become clear in the next subsection, we refer to these as \emph{strong} boundary traces, 
in contrast to the \emph{weak} boundary traces discussed further below. 

Assume that $\Omega \subseteq \bbR^{n}$ is a bounded Lipschitz domain. 
Then its boundary $\partial\Omega$ is equipped with the surface measure of dimension $n-1$. 
Here, $L^{p}(\partial\Omega)$ denotes the Lebesgue space with respect to the Hausdorff surface measure~\cite{federer,ernguermond1}.

\begin{theorem} \label{theorem:scalar-trace}
    Let $\Omega$ be a bounded Lipschitz domain.
    Let $p \in [1,\infty]$ and $m \in [0,\infty)$ with $m > 1/p$ or $m \geq 1$.
    Then there exists a bounded linear operator 
    \begin{align*}
        \gamma : W^{m,p}(\Omega) \rightarrow L^{p}(\partial\Omega)
    \end{align*}
    such that $\gamma \ding = \ding \restriction_{\partial\Omega}$
    for all $\ding \in C^{\infty}(\closure{\Omega})$.  
\end{theorem}

\begin{proof}
    The case $1 \leq p < \infty$ is covered by Theorem~3.10 in~\cite{ernguermond1}.
    For $1 < p < \infty$, see also Theorem~B in~\cite{mitrea2008traces}.
    The case $p = \infty$ follows because $W^{m,\infty}(\Omega)$ coincides with a classical H\"older or Lipschitz space for $0 < m \leq 1$. 
    If $m > 1$, then all members of $W^{m,\infty}(\Omega)$ are Lipschitz.
\end{proof}

Since $\Omega$ is a bounded Lipschitz domain, the boundary $\partial\Omega$ is a Lipschitz manifold of dimension $n-1$ 
along which there exists a measurable unit normal vector field $\normalfield : \partial\Omega \rightarrow \bbR^{n}$ defined almost everywhere with respect to the surface measure. 
We let $L^{p}\Lambda^{k}(\partial\Omega)$ denote the Lebesgue space with respect to the Hausdorff surface measure~\cite{ernguermond1} of functions with values in alternating $k$-forms that are tangential almost everywhere at the boundary.

Recall that there is a notion of tangential boundary trace given by a linear mapping
\begin{align}\label{math:tangential-trace}
    \Trace : C^{\infty}\Lambda^{k}(\closure{\Omega}) \rightarrow L^{\infty}\Lambda^{k}(\partial\Omega).
\end{align}
Formally, for all $\ding \in C^{\infty}\Lambda^{k}(\closure{\Omega})$, one has\footnote{Here, $\normalfield^{\flat}$ denotes the $1$-form (or covector) associated to $\normalfield$ via the Euclidean inner product, and $\normalfield \contract$ denotes the interior product with a vector.} 
\begin{align*}
    \Trace \ding = \normalfield \contract (\normalfield^{\flat} \wedge ( \ding \restriction_{\partial\Omega} ))
\end{align*}
almost everywhere on $\partial\Omega$. 
Thus, this tangential trace is defined by taking the pointwise restriction of the $k$-form to the boundary and then projecting to the tangential directions almost everywhere along the boundary.

The tangential trace on smooth differential forms extends to Sobolev differential forms.

\begin{proposition} \label{proposition:spursatz-der-differentialformen}
    Let $\Omega$ be a bounded Lipschitz domain.
    Let $p \in [1,\infty]$ and $m \in [0,\infty)$ with $m > 1/p$ or $m \geq 1$.
    Then there exists a bounded linear operator 
    \begin{align}\label{math:spursatz-der-differentialformen}
        \Trace : W^{m,p}\Lambda^{k}(\Omega) \rightarrow L^{p}\Lambda^{k}(\partial\Omega)
    \end{align}
    that extends the tangential trace~\eqref{math:tangential-trace}. 
\end{proposition} 

If $\Omega$ is a bounded Lipschitz domain and $\Gamma \subseteq \partial\Omega$ is any relatively open set, 
$p \in [1,\infty]$, and $m \in [0,\infty)$ with $m > 1/p$ or $m \geq 1$,
then we define
\begin{align*}
    W^{m,p}\Lambda^{k}(\Omega,\Gamma) 
    := 
    \left\{ 
        \ding \in W^{m,p}\Lambda^{k}(\Omega) \suchthat* (\Trace\ding) \restriction_\Gamma = 0 
    \right\}
    .
\end{align*}

\subsection{Weak boundary traces of differential forms}

We define a notion of weak boundary traces of differential forms via an integration-by-parts formula. This builds on the classical integration-by-parts formula and Stokes' theorem for smooth differential forms.
For the purpose of this manuscript, we only need a notion of homogeneous weak boundary traces.
This is sufficient to define the relevant subspaces of $\Sobdiff^{p,q}\Lambda^{k}(\Omega)$ with homogeneous partial boundary conditions.

Assume that $\Omega$ is a bounded Lipschitz domain and $\Gamma \subseteq \partial \Omega$ is a relatively open subset of the domain boundary.
Write $\Gamma^{c} := \partial\Omega \setminus \closure{\Gamma}$ for the complementary boundary part. 
We define 
\begin{align*}
    \Sobdiff^{p,q}\Lambda^{k}(\Omega,\Gamma)
    :=
    \left\{\;
        \ding \in \Sobdiff^{p,q}\Lambda^{k}(\Omega) 
        \suchthat*
        \exists \widetilde\ding \in \Sobdiff^{p,q}\Lambda^{k}( \bbR^{n} \setminus \closure{\Gamma^{c}} ) 
        : 
        \widetilde\ding \restriction_{\Omega} = \ding, \; \widetilde\ding \restriction_{\bbR^{n} \setminus \Omega} = 0 
    \;\right\}
    .
\end{align*}
Equivalently, 
$\ding \in \Sobdiff^{p,q}\Lambda^{k}(\Omega,\Gamma)$ if and only if $\ding \in \Sobdiff^{p,q}\Lambda^{k}(\Omega)$ and its extension-by-zero onto $\bbR^{n} \setminus \closure{\Gamma^{c}}$ is contained in the space $\Sobdiff^{p,q}\Lambda^{k}( \bbR^{n} \setminus \closure{\Gamma^{c}} )$.

A third equivalent characterization directly uses an integration-by-parts formula: 
$\ding \in \Sobdiff^{p,q}\Lambda^{k}(\Omega,\Gamma)$ if and only if $\ding \in \Sobdiff^{p,q}\Lambda^{k}(\Omega)$ and 
\begin{align}\label{math:weak-trace-condition}
    \int_{\Omega} \ding \wedge \exteriorderivative \phi = (-1)^{k+1} \int_{\Omega} \exteriorderivative\ding \wedge \phi
\end{align}
for all $\phi \in C^{\infty}_{c}\Lambda^{n-k-1}(\bbR^{n})$ such that $\supp \phi \cap \partial\Omega \subseteq \Gamma$. 

We say that the members of $\Sobdiff^{p,q}\Lambda^{k}(\Omega,\Gamma)$ satisfy \emph{homogeneous partial boundary conditions} on $\Gamma$.
The space $\Sobdiff^{p,q}\Lambda^{k}(\Omega,\Gamma)$ is a closed subspace of $\Sobdiff^{p,q}\Lambda^{k}(\Omega)$, and hence a Banach space.
Clearly, $\Sobdiff^{p,q}\Lambda^{k}(\Omega,\emptyset) = \Sobdiff^{p,q}\Lambda^{k}(\Omega)$. 

Note that the exterior derivative preserves homogeneous partial boundary conditions, that is,
\[ 
    \exteriorderivative \Sobdiff^{p,q}\Lambda^{k}(\Omega,\Gamma) \subseteq \Sobdiff^{q,r}\Lambda^{k+1}(\Omega,\Gamma)
\]
for all $p,q,r \in [1,\infty]$.

\begin{remark}
    The distinction between strong and weak boundary traces corresponds to the two distinct regularity axes for differential forms, as represented by the Sobolev--Slobodeckij spaces $W^{m,p}\Lambda^{k}(\Omega)$ and the spaces $\Sobdiff^{p,q}\Lambda^{k}(\Omega)$. 
    Again, the two notions of traces are not comparable in general.
    
    In fact, it is not even trivial that the strong and weak traces of a differential form agree whenever both are defined. 
    A full discussion of their relationship is beyond the scope of this manuscript.
    The notions of strong and weak boundary traces of differential forms have emerged only recently in the literature on geometric harmonic analysis~\cite{mitrea2022geometricI,mitrea2022geometricII} but appear to be of utility for the discussion of analysis in low-regularity regimes. The definition of weak boundary traces via an integration-by-parts formula is also amenable to a discussion in abstract Hilbert space theory~\cite{hiptmair2023traces}.
    As we enlarge the class of test forms in the integration-by-parts formula, we further restrict the Sobolev spaces of differential forms and impose a notion of homogeneous partial boundary conditions.
\end{remark}

\subsection{Differential forms on simplices} \label{subsection:background:differentialforms-on-simplices}

So far, we have defined the spaces of functions and differential forms on open subsets of $\bbR^{n}$.
We can also develop these spaces on any $d$-dimensional simplex, subject only to minor technical modifications.

For any subset $S \subseteq \bbR^{n}$, we let $C^{\infty}(S)$ be the space of restrictions of smooth functions on $\bbR^{n}$ to the set $S$. 
The integral $\int_{S} \ding$ of any integrable $k$-form on a $k$-dimensional simplex $S$ is well-defined.

We briefly review the pullback of differential forms\footnote{The pullback of differential forms generalizes the Piola transformations in three-dimensional vector analysis. The pullback commutes with the exterior derivative.} along affine mappings between simplices and derive inequalities between Sobolev--Slobodeckij seminorms.
We include the result for completeness; see also~\cite[Lemma~11.7]{ernguermond1},~\cite[Lemma~2.2]{ern2017finite}, and~\cite{stern2013lp} for related results.
The main observation is that each Sobolev--Slobodeckij seminorm of the pullback can be estimated by the corresponding seminorm of the original differential form.

\begin{lemma}\label{lemma:pullback-estimate}
    Let $\widehat{T}$ and ${T}$ be $n$-dimensional simplices and let $\Phi : \widehat{T} \to {T}$ be an invertible affine map.
    Let $0 < \sigma_{\min} \leq \sigma_{\max}$ be the smallest and the largest singular values of its Jacobian. 
    Let $0 \leq k \leq n$, $p \in [1,\infty]$, and $s \in [0,\infty)$.
    For any $\ding \in W^{s,p}\Lambda^{k}({T})$, the pullback along $\Phi$ satisfies $\Phi^{\ast} \ding \in W^{s,p}\Lambda^{k}(\widehat{T})$.
    Moreover, the following inequalities hold.
    
    If $s \in \bbN_{0}$, then 
    \begin{align}\label{math:pullback-estimate:integral}
        |\Phi^{\ast} \ding|_{W^{s,p}\Lambda^{k}(\widehat{T})}
        \leq
        C
        \sigma_{\max}^{k+s}
        \sigma_{\min}^{-n/p}
        |\ding|_{W^{s,p}\Lambda^{k}({T})}
        .
    \end{align}
    If $s \in \bbR_{0}^{+} \setminus \bbN_{0}$, then
    \begin{align}\label{math:pullback-estimate:fractional}
        |\Phi^{\ast} \ding|_{W^{s,p}\Lambda^{k}(\widehat{T})}
        \leq
        C
        \sigma_{\max}^{k+s+n/p}
        \sigma_{\min}^{-2n/p}
        |\ding|_{W^{s,p}\Lambda^{k}({T})} .
    \end{align}
    Here, the constant $C > 0$ depends only on $n$, $k$, $p$, and $s$.
\end{lemma}

\begin{proof}
    Let $\ding \in W^{s,p}\Lambda^{k}({T})$. We interpret $\ding$ as a function with values in a vector space. 
    For any integer $m \in \bbN_{0}$, we write $\nabla^{m} \ding$ for the $m$-th weak derivative of $\ding$. 
    Since $\Phi$ is affine, its Jacobian is everywhere equal to a matrix $A \in \bbR^{n \times n}$. 
    We also recall Hadamard's inequality: 
    $| \det A | \leq \sigma_{\max}^{ n}$ and $| \det A^{-1} | \leq \sigma_{\min}^{-n}$.
    We use the following fact~\cite{horn2013matrix}: 
    There exists a matrix $B$ with singular values bounded by $\sigma_{\max}^{k}$ such that\footnote{The matrix $B$ is the matrix of $k \times k$ minors of $A$, up to a permutation of entries.} 
    $\Phi^{\ast} \ding(\hat{x}) = B \ding \circ \Phi(\hat{x})$.
    Consider the case $s \in \bbN_{0}$. 
    We derive the following estimate for almost all $\hat{x} \in \widehat{T}$:
    \begin{align*}
        \| \nabla^{s} \Phi^{\ast} \ding(\hat{x}) \| 
        \leq 
        C 
        \sigma_{\max}^{k+s} 
        \| \nabla^{s} \ding ( \Phi(\hat{x}) ) \|
        .
    \end{align*}
    This proves~\eqref{math:pullback-estimate:integral} if $p = \infty$,
    and in the case $p < \infty$, one uses a change-of-variables formula on $\widehat{T}$ and Hadamard's inequality.
    Now consider the case that $s = m + \sigma$ with $m \in \bbN_0$ and $0 < \sigma < 1$. 
    Let $\hat{x}, \hat{y} \in \widehat{T}$. 
    Since $\Phi$ has constant Jacobian, 
    \begin{align*}
        \| \nabla^{m} \Phi^{\ast} \ding(\hat{x}) - \nabla^{m} \Phi^{\ast} \ding(\hat{y}) \|
        \leq 
        C 
        \sigma_{\max}^{k+m} 
        \| ( \nabla^{m} \ding )( \Phi(\hat{x}) ) - ( \nabla^{m} \ding )( \Phi(\hat{y}) ) \|
        .
    \end{align*}
    With that in mind, one estimates 
    \begin{align*}
        &
        \left\|
            ( \nabla^{m} \ding )( \Phi(\hat{x}) ) - ( \nabla^{m} \ding )( \Phi(\hat{y}) )     
        \right\|
        \cdot 
        \left\| 
            \hat{x} - \hat{y}
        \right\|^{ -\sigma - \frac{n}{p} }
        \\&\qquad 
        \leq 
        \sigma_{\max}^{ \sigma + \frac{n}{p} }
        \left\| 
            ( \nabla^{m} \ding )( \Phi(\hat{x}) ) - ( \nabla^{m} \ding )( \Phi(\hat{y}) )   
        \right\| 
        \cdot 
        \left\| 
            \Phi(\hat{x}) - \Phi(\hat{y})
        \right\|^{ -\sigma - \frac{n}{p} }
        .
    \end{align*}
    This proves~\eqref{math:pullback-estimate:fractional} if $p = \infty$, 
    and in the case $p < \infty$, it holds by a change-of-variables formula on $\widehat{T} \times \widehat{T}$ and Hadamard's inequality.
\end{proof}   

We write $\trace_{S,F}$ for the trace operator\footnote{We use the term \emph{trace operator} for tangential traces onto sets of arbitrary codimension. Traces in the analysis of PDE and in numerical analysis are typically onto subsets of codimension one, such as domain boundaries, but traces onto lower dimensional sets have a long history~\cite{jonsson1984function,triebel1997fractals,triebel2001structure}. The tangential trace in this article can be identified with the pullback of a differential form~\cite{leesmooth}.} from any simplex $S$ onto any of its subsimplices $F \in \subsimplices(S)$.
Given any simplex $T$ and one of its facets (codimension-one simplices) $F \subseteq T$, 
the trace operator \[\trace_{T,F} : C^{\infty}\Lambda^{k}(T) \to C^{\infty}\Lambda^{k}(F)\] between spaces of smooth differential forms on simplices extends continuously to Sobolev spaces of differential forms. 

\begin{lemma} \label{lemma:trace-estimate-on-simplices} 
    Let $T$ be an $n$-dimensional simplex, and let $F \in \subsimplices_{n-1}(T)$ be one of its facets.
    Let $p \in [1,\infty]$ and $m \in [0,\infty)$ with $m > \frac 1 p$ or $m = 1$.
    Then 
    \[
        \trace_{T,F} : W^{m,p}\Lambda^{k}(T) \rightarrow L^{p}\Lambda^{k}(F)
    \]
    is a bounded operator. Moreover,\footnote{Here and in what follows, we stipulate $1/\infty = 0$.}
    \begin{align*}
        \| \trace_{T,F} \ding \|_{L^{p}\Lambda^{k}(F)}
        \leq
        C_{\trace}
        \left(
        h_{T}^{-\frac 1 p}
        \| \ding \|_{L^{p}\Lambda^{k}(T)}
        +
        h_{T}^{m-\frac 1 p}
        \left| \ding \right|_{W^{m,p}\Lambda^{k}(T)}
        \right)
        ,
        \quad
        \ding \in W^{m,p}\Lambda^{k}(T)
        ,
    \end{align*}
    where $C_{\trace} > 0$ depends only on $p$, $m$, $n$, $k$, and $\shapemeasure(T)$.
\end{lemma}

\begin{proof}
    If $T$ is a reference simplex,
    then the statement follows from Proposition~\ref{proposition:spursatz-der-differentialformen}.
    The general case then follows by a scaling argument (Lemma~\ref{lemma:pullback-estimate}). 
\end{proof}

\subsection{Finite element spaces on triangulations} \label{subsection:background:femspaces}

In this manuscript, we adopt the framework of finite element exterior calculus~\cite{hiptmair2002finite,AFW1}. We now review its basic definitions and notation. 
We consider finite element spaces of polynomial differential forms, their construction by means of traces and extension operators, and spaces of degrees of freedom.

Throughout the remainder of this section, let $\calT$ be an $n$-dimensional simplicial complex, and let $\calU \subseteq \calT$ be any simplicial subcomplex. 

For any simplex $S$,
let $\calP_{r}\Lambda^{k}(S)$ be the space of polynomial differential $k$-forms of degree at most $r \geq 0$ on $S$,
and let $\calP^{-}_{r}\Lambda^{k}(S)$ be the space of trimmed polynomial differential $k$-forms
of degree at most $r \geq 1$ on $S$;
we refer to the literature~\cite{AFW1} for details of their definitions. Here, we only recall 
\[
    \calP_{r-1}\Lambda^{k}(S) \subseteq \calP^{-}_{r}\Lambda^{k}(S) \subseteq \calP_{r}\Lambda^{k}(S),
    \quad 
    \exteriorderivative \calP^{-}_{r}\Lambda^{k}(S) = \exteriorderivative \calP_{r}\Lambda^{k}(S).
\]
Generally speaking, $\calP^{-}_{r}\Lambda^{k}(S)$ is the smallest subspace of $\calP_{r}\Lambda^{k}(S)$ with the same range under the exterior derivative.\footnote{For technical convenience, we define $\calP_{r}\Lambda^{k}(S) := 0$ for $r < 0$ and $\calP^{-}_{r}\Lambda^{k}(S) := 0$ for $r \leq 0$.}

We define spaces of polynomial differential forms satisfying boundary conditions as follows. On any simplex $S$, set 
\begin{align*}
    \mathring{\calP}_{r}\Lambda^{k}(S)
     & :=
    \left\{
    \ding \in \calP_{r}\Lambda^{k}(S)
    \suchthat*
    \forall F \in \subsimplices(S), F \neq S : \trace_{S,F} \ding = 0
    \right\},
    \\
    \mathring{\calP}^{-}_{r}\Lambda^{k}(S)
     & :=
    \left\{
    \ding \in \calP^{-}_{r}\Lambda^{k}(S)
    \suchthat*
    \forall F \in \subsimplices(S), F \neq S : \trace_{S,F} \ding = 0
    \right\}.
\end{align*}
These \enquote{bubble spaces} are just the polynomial differential forms whose traces vanish along the entire boundary.
We introduce finite element spaces on triangulations as collections of polynomial differential forms on simplices that are single-valued on subsimplices.
Recall that a family $\{ \ding_{T} \in C^{\infty}\Lambda^{k}(T) \suchthat T \in \subsimplices_{n}(\calT) \}$ is called \emph{single-valued} if 
\[
    \forall T, T' \in \subsimplices_{n}(\calT), \; T \cap T' \neq \emptyset 
    : 
    \trace_{T, T \cap T'} \ding_{T} = \trace_{T', T \cap T'} \ding_{T'} 
    .
\]
We define 
\begin{align*}
    \calP_{r}\Lambda^{k}(\calT)
     & :=
    \left\{
        \ding \in \bigoplus_{ T \in \subsimplices_{n}(\calT) } \calP_{r}\Lambda^{k}(T)
        \suchthat*
        \ding \text{ is single-valued }
    \right\}
    ,
    \\
    \calP^{-}_{r}\Lambda^{k}(\calT)
    & :=
    \left\{
        \ding \in \bigoplus_{ T \in \subsimplices_{n}(\calT) } \calP^{-}_{r}\Lambda^{k}(T)
        \suchthat*
        \ding \text{ is single-valued }
    \right\}
    .
\end{align*}
The definitions ensure that the following trace mappings are well-defined.
For each $F \in \calT$, we have the \emph{(global) trace operators} 
\begin{align*}
    \Trace_{F} : \calP_{r}^{ }\Lambda^{k}(\calT) \rightarrow \calP_{r}    \Lambda^{k}(F),
    \quad
    \Trace_{F} : \calP^{-}_{r}\Lambda^{k}(\calT) \rightarrow \calP^{-}_{r}\Lambda^{k}(F).
\end{align*}
Given the simplicial subcomplex $\calU \subseteq \calT$, we define 
\begin{gather*}
    \calP_{r}\Lambda^{k}(\calT,\calU)
    :=
    \left\{\;
    u \in \calP_{r}\Lambda^{k}(\calT)
    \suchthat*
    \forall F \in \calU : \Trace_{F} u = 0
    \;\right\}
    ,
    \\
    \calP^{-}_{r}\Lambda^{k}(\calT,\calU)
    :=
    \left\{\;
    u \in \calP^{-}_{r}\Lambda^{k}(\calT)
    \suchthat*
    \forall F \in \calU : \Trace_{F} u = 0 
    \;\right\}
    .
\end{gather*}
In the case $\calU = \emptyset$,
we have
\begin{align*}
    \calP_{r}    \Lambda^{k}(\calT,\emptyset) = \calP_{r}    \Lambda^{k}(\calT),
    \qquad
    \calP^{-}_{r}\Lambda^{k}(\calT,\emptyset) = \calP^{-}_{r}\Lambda^{k}(\calT).
\end{align*}

The structure of the finite element spaces is governed by the geometric decomposition. 
That decomposition applies to the spaces of degrees of freedom as well.

For each $F \in \calT$, we fix the \emph{extension operators}
\begin{align*}
    \Ext_{F,\calT}^{r,k  } : \calP_{r}\Lambda^{k}(F)     \rightarrow \calP_{r}^{ }\Lambda^{k}(\calT),
    \quad
    \Ext_{F,\calT}^{r,k,-} : \calP^{-}_{r}\Lambda^{k}(F) \rightarrow \calP^{-}_{r}\Lambda^{k}(\calT)
\end{align*}
satisfying the following two properties.
On the one hand,
they are right inverses of the traces,
\begin{align*}
    \Trace_{F} \Ext_{F,\calT}^{r,k} = \Id, \quad \Trace_{F} \Ext_{F,\calT}^{r,k,-} = \Id.
\end{align*}
On the other hand, they are localized in the sense that, for all $S \in \calT$ with $F \nsubseteq S$, we have
\begin{gather*}
    \forall \ding \in \mathring{\calP}_{r}\Lambda^{k}(F) 
    : 
    \Trace_{S} \Ext_{F,\calT}^{r,k} \ding = 0
    ,
    \\
    \forall \ding \in \mathring{\calP}^{-}_{r}\Lambda^{k}(F)
    : 
    \Trace_{S} \Ext_{F,\calT}^{r,k,-} \ding = 0
    .
\end{gather*}
Extension operators of this kind are discussed in the literature~\cite{AFW1,afwgeodecomp,christiansen2008construction,licht2022basis}.

The finite element spaces decompose into direct sums:
\begin{align*}
    \calP_{r}\Lambda^{k}(\calT,\calU)
    =
    \bigoplus_{ \substack{ F \in \calT \\ F \notin \calU } }
    \Ext_{F,\calT}^{r,k} \mathring{\calP}_{r}\Lambda^{k}(F)
    ,
    \quad
    \calP^{-}_{r}\Lambda^{k}(\calT,\calU)
    =
    \bigoplus_{ \substack{ F \in \calT \\ F \notin \calU } }
    \Ext_{F,\calT}^{r,k,-} \mathring{\calP}^{-}_{r}\Lambda^{k}(F)
    .
\end{align*}
The geometric decomposition corresponds to the common notion that the shape functions of finite element spaces are associated with specific simplices of the triangulation and are localized near those simplices.

We introduce the dual spaces of the finite element spaces, whose members are commonly known as \emph{degrees of freedom}. 
The degrees of freedom are functionals that enter the construction of averaging-based projections below. 
Given $F \in \calT$, we define
\begin{align*}
    \calC_{r}\Lambda^{k}(F)
     & :=
    \left\{
    \ding \mapsto \int_{F} \eta \wedge \Trace_{F} \ding
    \suchthat*
    \eta \in \calP_{r+k-\dim(F)}^{-}\Lambda^{\dim(F)-k}(F)
    \right\},
    \\
    \calC^{-}_{r}\Lambda^{k}(F)
     & :=
    \left\{
    \ding \mapsto \int_{F} \eta \wedge \Trace_{F} \ding
    \suchthat*
    \eta \in \calP_{r+k-\dim(F)-1}\Lambda^{\dim(F)-k}(F)
    \right\}.
\end{align*}
One can show~\cite{AFW1,licht2017priori} that these spaces of functionals provide a canonical direct decomposition of the dual spaces of the respective finite element spaces:
\begin{align*}
    \calP_{r}\Lambda^{k}(\calT,\calU)^{\ast}
    =
    \bigoplus_{ \substack{ F \in \calT \\ F \notin \calU } }
    \calC_{r}\Lambda^{k}(F)
    ,
    \quad
    \calP^{-}_{r}\Lambda^{k}(\calT,\calU)^{\ast}
    =
    \bigoplus_{ \substack{ F \in \calT \\ F \notin \calU } }
    \calC^{-}_{r}\Lambda^{k}(F)
    .
\end{align*}
By analogy with the geometric decomposition of the finite element spaces, we have geometric decompositions of the degrees of freedom:
we work with functionals on finite element spaces that are associated with simplices of the triangulation.

\begin{remark}
    The spaces $\calC_{r}\Lambda^{k}(S)$ and $\calC^{-}_{r}\Lambda^{k}(S)$ are introduced above as spaces of functionals on the simplexwise finite element spaces. Using the global trace operator, the functionals are extended onto the conforming finite element space. 
    Each such functional $\phi^\ast$ is given by an explicit trace formula on the simplex $S$ of the form 
    \begin{align*}
        \phi^\ast(\ding)
        = 
        \int_{S} \eta \wedge \Trace_{S} \ding
        .
    \end{align*}
    Hence, if $S$ is a subsimplex of another simplex $T$, then the same formula defines a functional on $C^{\infty}\Lambda^{k}(T)$ by taking the trace from $T$ to $S$. 
    We then use the same notation for this local functional:
    \begin{align*}
        \phi^\ast(\ding)
        :=
        \int_{S} \eta \wedge \trace_{T,S} \ding, 
        \qquad
        \ding \in C^{\infty}\Lambda^{k}(T). 
    \end{align*}
    This is only a notational convention, in line with the conventions in the literature, and only involves traces on classical function spaces.
\end{remark}

\section{Rough degrees of freedom} \label{section:roughdof}

The degrees of freedom are initially defined as functionals on conforming finite element spaces and on local polynomial spaces. 
In this section, we extend the degrees of freedom to bounded linear functionals on much larger spaces of differential forms
by rewriting them in terms of integrals over elements and facets. 
This extension is non-standard (and not necessarily unique) and relies on an integration-by-parts formula. 
We call the resulting functionals \emph{rough degrees of freedom}. 
The rough degrees of freedom are a key component of the projection constructed below. 

We assume that $\calT$ is an $n$-dimensional simplicial complex, 
and we let $\calU \subseteq \calT$ be a simplicial subcomplex of dimension $n-1$. 

We later construct projection operators for the two families of finite element differential forms. The two constructions are completely analogous, apart from degree conventions.
To avoid duplication when there are only minor technical differences, we establish the following convention.

\begin{convention} \label{convention:twofamilies}
    We let $k \in \bbZ$ and $r \in \bbN$. 
    We fix a family of finite element spaces of differential forms by adopting one of the following two notational conventions:
    \begin{gather} \tag{A}
        \begin{cases}
            \qquad
            \calP\Lambda^{k}(\calT) := \calP_{r}\Lambda^{k}(\calT),
            \quad
            \calP\Lambda^{k}(\calT,\calU) := \calP_{r}\Lambda^{k}(\calT,\calU), &
            \\
            \qquad
            \text{and for all $S \in \calT$:}                                             &
            \\
            \qquad
            \calP\Lambda^{k}(S) := \calP_{r}\Lambda^{k}(S),
            \;
            \mathring{\calP}\Lambda^{k}(S) := \mathring{\calP}_{r}\Lambda^{k}(S),
            \;
            \calC\Lambda^{k}(S) := \calC_{r}\Lambda^{k}(S),                      &
        \end{cases}
    \end{gather}
    or
    \begin{gather} \tag{B}
        \begin{cases}
            \qquad
            \calP\Lambda^{k}(\calT) := \calP^{-}_{r}\Lambda^{k}(\calT),
            \quad
            \calP\Lambda^{k}(\calT,\calU) := \calP^{-}_{r}\Lambda^{k}(\calT,\calU), &
            \\
            \qquad
            \text{and for all $S \in \calT$:}                                                 &
            \\
            \qquad
            \calP\Lambda^{k}(S) := \calP^{-}_{r}\Lambda^{k}(S),
            \;
            \mathring{\calP}\Lambda^{k}(S) := \mathring{\calP}^{-}_{r}\Lambda^{k}(S),
            \;
            \calC\Lambda^{k}(S) := \calC^{-}_{r}\Lambda^{k}(S).                      &
        \end{cases}
    \end{gather}
    For every $S \in \calT$, we fix an index set
    \[
        I(S) 
        := 
        \left\{ i \in \bbN \suchthat* 1 \leq i \leq \dim\mathring{\calP}\Lambda^{k}(S) \right\}
        .
    \]
    We choose the first option when $k=0$ and the second option when $k=n$.\footnote{Recall that $\calP_{r}\Lambda^{0} = \calP^{-}_{r}\Lambda^{0}$ and $\calP_{r-1}\Lambda^{n} = \calP^{-}_{r}\Lambda^{n}$.} 
\end{convention}

We choose bases of the shape functions and of the degrees of freedom that are localized around simplices and constitute a biorthogonal system.
In fact, this is a common assumption in finite element methods and can be implemented easily in algorithms~\cite{ernguermond2}.
The following theorem (see~\cite[Theorem~5.2, Theorem~5.4]{licht2021}) formalizes the idea and also establishes some important inverse estimates.

\begin{theorem} \label{theorem:biorthogonal}
    There exist
    \begin{enumerate}
    \renewcommand{\labelenumi}{\alph{enumi})}
    \item 
    a basis
    $\left\{ \phi_{S,i} \right\}_{S \in \calT, i \in I(S)}$ of $\calP\Lambda^{k}(\calT)$,
    \item 
    bases
    $\left\{ \phi^{\ast}_{S,i} \right\}_{i \in I(S)}$ of $\calC\Lambda^{k}(S)$ for each $S$,
    \end{enumerate}
    such that the following conditions are satisfied:
    \begin{enumerate}
        \renewcommand{\labelenumi}{(\roman{enumi})}
        \item The set $\{ \phi_{S,i} \}_{ S \in \calT \setminus \calU, i \in I(S) }$ is a basis of $\calP\Lambda^{k}(\calT,\calU)$.
        \item $\forall S, S' \in \calT : S \nsubseteq S' \implies \forall i \in I(S) : \Trace_{S'} \phi_{S,i} = 0$.
        \item $\forall S, S' \in \calT : \forall i \in I(S) : \forall j \in I(S') : \phi^{\ast}_{S,i}( \phi_{S',j} ) = 
        \begin{cases}
            1 & \text{ if } i=j \text{ and } S = S'
            \\
            0 & \text{ otherwise. }
        \end{cases}$
    \end{enumerate}
    Moreover, 
    \begin{enumerate}
    \renewcommand{\labelenumi}{(\roman{enumi})}
    \item 
    Given $p \in [1,\infty]$ and $s \in [0,\infty)$,
    there exist $C_{\mathrm{A},s}, C_{\mathrm{A}} > 0$ 
    such that for all $T \in \calT$ with $\dim(T) = n$ and $S \in \subsimplices(T)$:
    \begin{subequations} 
        \begin{gather}
            | \phi_{S,i} |_{W^{s,p}\Lambda^{k}(T)}
            \leq 
            C_{\mathrm{A},s} h_{S}^{ \frac{n}{p} - k - s },
            \quad
            i \in I(S)
            ,
            \label{math:biorthogonal:boundphi:primal}
            \\
            | \phi^{\ast}_{S,i}( \ding ) |
            \leq
            C_{\mathrm{A}}
            h_{S}^{ k - \frac{n}{p}}
            \| \ding \|_{L^{p}\Lambda^{k}(T)},
            \quad
            \ding \in \calP\Lambda^{k}(T),
            \quad
            i \in I(S)
            .
            \label{math:biorthogonal:boundphi:dual}  
        \end{gather}
    \end{subequations}
    Here, $C_{\mathrm{A},s}, C_{\mathrm{A}} > 0$ depend only on $p$, $n$, $r$, $s$, $\comparisonconstant(\calT)$, and $\shapemeasure(\calT)$. 
    \item 
    Given $p \in [1,\infty]$,
    there exists $C_{\mathrm{B}} > 0$ %
    such that for all $n$-dimensional $T \in \calT$:
    \begin{gather}\label{math:biorthogonal:inverseinequality}
        \| \exteriorderivative \ding \|_{L^{p}\Lambda^{k+1}(T)}
        \leq 
        C_{\mathrm{B}} 
        h_{T}^{-1}
        \| \ding \|_{L^{p}\Lambda^{k}(T)},
        \quad 
        \ding \in \calP\Lambda^{k}(T).
    \end{gather}
    Here, $C_{\mathrm{B}} > 0$ depends only on $p$, $n$, $r$, $\comparisonconstant(\calT)$, and $\shapemeasure(\calT)$. 
    \end{enumerate}
\end{theorem}

\begin{proof}
    We give an outline of the proof and refer to~\cite{licht2021} for additional details. 
    Let $S \in \calT$ be a simplex. 
    We choose a basis $\left\{ \phi^{\ast}_{S,i} \right\}_{i \in I(S)}$ of $\calC\Lambda^k(S)$ by transformation from the reference simplex. 
    This also yields a pre-dual basis $\left\{ \psi_{S,i} \right\}_{i \in I(S)}$ of $\mathring{\calP}\Lambda^k(S)$. 
    
    Let $\Ext_{S,\calT}$ denote the extension operator associated with the chosen finite element family. 
    It can be defined elementwise via transformation from reference simplices. 
    For each $T \in \subsimplices_{n}(\calT)$, we define 
    \[
        \phi_{T,i} := \psi_{T,i}, \qquad i \in I(T).
    \]
    For any simplex $S \in \subsimplices_{d}(\calT)$ with $0 \leq d < n$, we recursively define 
    \[
        \phi_{S,i} 
        := 
        \Ext_{S,\calT} \psi_{S,i} 
        - 
        \sum_{\ell=d+1}^{n} \sum_{ \substack{ S' \in \subsimplices_{\ell}(\calT) \\ j \in I(S') } }
        \phi_{S',j}^{\ast} \left( \Ext_{S,\calT} \psi_{S,i} \right)
        \phi_{S',j} 
        , 
        \qquad i \in I(S).
    \]
    By induction on the dimension, 
    we observe that for all $F \in \calT$ with $S \nsubseteq F$ and for all $i \in I(S)$, one has $\Trace_{F} \phi_{S,i} = 0$.
    Moreover, $\phi_{S',j}^{\ast} ( \phi_{S,i} ) = 1$ if and only if $S=S'$ and $i=j$, and otherwise that term equals zero. 
    It follows that $\left\{ \phi_{S,i} \right\}_{S \in \calT, i \in I(S)}$ is a basis of $\calP\Lambda^{k}(\calT)$.
    Likewise, it is easily seen that 
    $\left\{ \phi_{S,i} \right\}_{S \in \calT \setminus \calU, i \in I(S)}$ is a basis of $\calP\Lambda^{k}(\calT,\calU)$.
    
    Finally,~\eqref{math:biorthogonal:boundphi:primal},~\eqref{math:biorthogonal:boundphi:dual}, and~\eqref{math:biorthogonal:inverseinequality} follow via a scaling argument from the reference simplex. 
    The constants depend only on $p$, $n$, $r$, $s$, $\comparisonconstant(\calT)$, and $\shapemeasure(\calT)$.
\end{proof}

We have defined the degrees of freedom as functionals on conforming finite element spaces in terms of weighted integrals over simplices. 
There are many possible formulas for each degree of freedom that produce the same functional on the finite element space but will generally differ on smooth differential forms. 
We choose particular representations of the degrees of freedom that extend the functionals to non-smooth spaces of differential forms by means of an integration-by-parts formula~\cite[Theorem~7.1]{licht2021}.
This construction is central to this manuscript.

\begin{theorem} \label{theorem:roughdof}
    Let $F, S \in \calT$ with $S \subseteq F$ and $\dim(F) = n - 1$, and let $i \in I(S)$.
    There exists $\xi_{F,S,i} \in C^{\infty}\Lambda^{n-k-1}(F)$ with support in the interior of $F$ such that 
    for all $T \in \supersimplices_{n}(\calT,F)$ we have 
    \begin{align} \label{math:roughdof:facet}
        \phi^{\ast}_{S,i}( \ding )
        =
        \int_{F} \xi_{F,S,i} \wedge \trace_{T,F} \ding,
        \quad
        \ding \in \calP\Lambda^{k}(T). 
    \end{align}
    For each $T \in \calT$ with $\dim(T) = n$ and $F \subseteq T$ 
    there exists $\Xi_{T,F,S,i} \in C^{\infty}\Lambda^{n-k-1}(T)$ whose support has positive distance from all facets of $T$ except $F$
    and which satisfies
    \begin{gather} \label{math:roughdof:integration-by-parts}
        \trace_{T,F} \Xi_{T,F,S,i} = \xi_{F,S,i},
        \\
        \phi^{\ast}_{S,i}( \ding )
        =
        o(F,T)
        \int_{T} 
        \exteriorderivative\Xi_{T,F,S,i} \wedge \ding + (-1)^{n-k-1} \Xi_{T,F,S,i} \wedge \exteriorderivative\ding
        ,
        \quad
        \ding \in \calP\Lambda^{k}(T)
        .
    \end{gather}
    Moreover, for every $q \in [1,\infty]$, there exists $C_{\Xi} > 0$ such that
    \begin{subequations}
        \begin{align}
            \| \Xi_{T,F,S,i} \|_{L^{q}\Lambda^{n-k-1}(T)}
            & \leq
            C_{\Xi}
            h_{S}^{ \frac{n}{q} - n+k+1}, \label{math:roughdof:scaling:xi}
            \\
            \| \exteriorderivative\Xi_{T,F,S,i} \|_{L^{q}\Lambda^{n-k}(T)}
            & \leq
            C_{\Xi}
            h_{S}^{ \frac{n}{q} - n+k}, \label{math:roughdof:scaling:exterior-derivative-xi}
            \\
            \| \trace_{T,F} \Xi_{T,F,S,i} \|_{L^{q}\Lambda^{n-k-1}(F)}
            & \leq
            C_{\Xi}
            h_{S}^{ \frac{n-1}{q} - n+k+1}.\label{math:roughdof:scaling:trace}
        \end{align}
    \end{subequations}
    Here, $C_{\Xi} > 0$ depends only on $r$, $n$, $q$, $\comparisonconstant(\calT)$, and $\shapemeasure(\calT)$. 
\end{theorem}

\begin{proof}
    We give an outline of the proof and refer to~\cite{licht2021} for additional details. 
    For each $F \in \subsimplices_{n-1}(\calT)$, 
    using transformations from reference simplices, we can choose a non-negative smooth function $\theta_{F}$ with compact support in the interior of $F$
    such that the following bilinear form is non-degenerate:
    \begin{align*}
        B_{F}(\eta,\ding)
        = 
        \int_{F} ( \star_{F} \theta_{F} \eta ) \wedge \ding 
        , 
        \quad 
        \ding, \eta \in \calP\Lambda^{k}(F)
        ,
    \end{align*}
    where $\star_{F}$ denotes the Hodge star operator on $F$. 
    
    Now, with $S, F \in \calT$ as in the statement, 
    we can write $\phi^{\ast}_{S,i}(\ding) = \phi^{\ast}_{S,i}( \trace_{T,F} \ding )$. 
    Then there exists $\eta \in \calP\Lambda^{k}(F)$ satisfying 
    \begin{align*}
        \phi^{\ast}_{S,i}( \ding )
        =
        B_{F}( \eta, \trace_{T,F} \ding ),
        \quad
        \ding \in \calP\Lambda^{k}(T),
        \quad 
        T \in \supersimplices_{n}(\calT,F)
        .
    \end{align*} 
    The existence of $\xi_{F,S,i} \in C^{\infty}\Lambda^{n-k-1}(F)$ follows with $\xi_{F,S,i} := \star_{F} \theta_{F} \eta$. 
    Notice that $\xi_{F,S,i}$ can be obtained by a reference transformation. 
    Likewise, if $T \in \subsimplices_{n}(\calT)$ with $F \subseteq T$, 
    then we obtain $\Xi_{T,F,S,i} \in C^{\infty}\Lambda^{n-k-1}(T)$ whose support has positive distance from all facets of $T$ except $F$
    such that 
    \[
        \xi_{F,S,i} = \trace_{T,F} \Xi_{T,F,S,i}.
    \]
    This gives the facet integral representation~\eqref{math:roughdof:integration-by-parts}. The volume integral representation~\eqref{math:roughdof:integration-by-parts} follows from Stokes' theorem applied to $\Xi_{T,F,S,i}\wedge\ding$. 
    
    Finally,~\eqref{math:roughdof:scaling:xi},~\eqref{math:roughdof:scaling:exterior-derivative-xi}, and~\eqref{math:roughdof:scaling:trace} follow via a scaling argument from the reference simplex. 
    The constants depend only on $n$, $r$, $q$, $\comparisonconstant(\calT)$, and $\shapemeasure(\calT)$.
\end{proof}

\begin{remark}
    Theorem~\ref{theorem:roughdof} represents degrees of freedom in a manner that is also defined for two regularity regimes: differential forms with weak exterior derivative, and differential forms with regularity that admits traces onto faces. 
    We point out that the degrees of freedom associated with the full-dimensional simplices already have a representation as integrals that is well-defined for differential forms with coefficients in Lebesgue spaces. In this sense, those degrees of freedom are already rough by definition.
\end{remark}

We now explain how to represent degrees of freedom by integrals over facets and elements.
Most importantly, our treatment of boundary conditions will require degrees of freedom associated with simplices in $\calU$ to be represented using facets that also belong to $\calU$.
We introduce a convention that fixes a representative pair consisting of a facet and an element for each degree of freedom.

\begin{convention} \label{convention:dofrepresentative} 
    For every simplex $S \in \calT$, we fix an $n$-simplex $T_{S} \in \subsimplices_{n}(\calT)$ such that $S \subseteq T_{S}$.
    If $\dim(S) = n$, then we require $T_{S} = S$.

    For any simplex $S \in \calT$ of dimension at most $n-1$, we also fix a facet $F_{S} \in \subsimplices_{n-1}(\calT)$ such that $S \subseteq F_{S} \subseteq T_{S}$.
    If $S \in \calU$, then we require $F_{S} \in \calU$.
    
    For any simplex $S \in \calT$ of dimension at most $n-1$ and $i \in I(S)$, we write
    \begin{align} \label{math:festsetzen}
        \Xi_{S,i} := \Xi_{T_{S},F_{S},S,i} \in C^{\infty}\Lambda^{n-k-1}(T_{S}).
    \end{align} 
\end{convention}

\section{Averaging-based projection} \label{section:averaging-based-projection}

This section defines an averaging-based finite element projection onto conforming finite element spaces.
The projection consists of two stages.
First, we take local projections onto a discontinuous piecewise polynomial finite element space. 
Second, we take weighted averages of the degrees of freedom of those piecewise projections to define a member of the conforming finite element space.

We assume that $\Omega \subseteq \bbR^{n}$ is a bounded Lipschitz domain endowed with a triangulation $\calT$, 
and let $\Gamma \subseteq \partial\Omega$ be a relatively open set that is endowed with a triangulation\footnote{Analogous to the discussion in Remark~\ref{remark:domain-triangulation}, a triangulation of the relatively open set $\Gamma \subseteq \partial\Omega$ commonly refers to a triangulation of $\closure{\Gamma}$, implicitly assuming that $\Gamma$ is at least a topological submanifold-with-boundary of $\partial\Omega$.} $\calU \subset \calT$. 

The following proposition establishes the local projections for the individual simplices.
We will only require them on full-dimensional simplices for the first stage.

\begin{proposition} \label{proposition:localprojection}
    Let $T \in \calT$ be an $n$-simplex.
    There exist bounded projections
    \begin{gather*}
        \Localproj_{T} :
        L^{p}\Lambda^{k}(T)
        \rightarrow
        \calP\Lambda^{k}(T)
        \subset
        L^{p}\Lambda^{k}(T),
        \quad 
        p \in [1,\infty],
        \\
        \Localproj_{T,\exteriorderivative} :
        L^{p}\Lambda^{k+1}(T)
        \rightarrow
        \calP_{r-1}\Lambda^{k+1}(T)
        \subset
        L^{p}\Lambda^{k+1}(T),
        \quad 
        p \in [1,\infty],
    \end{gather*}
    that satisfy the following conditions with some constant $C_{\BH} > 0$:
    \\
    \begin{subequations}
        \noindent 
        If $\calP\Lambda^{k}(T) = \calP_{r}\Lambda^{k}(T)$ and $\ding \in W^{s,p}\Lambda^{k}(T)$ with $s \in [0,r+1]$,
        and $m \in [0,s]$,
        then
        \begin{gather} \label{math:localprojection:full}
            \left| \ding - \Localproj_{T} \ding \right|_{W^{m,p}\Lambda^{k}(T)}
            \leq
            C_{\BH}
            h_{T}^{s-m}
            | \ding |_{W^{s,p}\Lambda^{k}(T)}
            .
        \end{gather}
        If $\calP\Lambda^{k}(T) = \calP^{-}_{r}\Lambda^{k}(T)$ and $\ding \in W^{s,p}\Lambda^{k}(T)$ with $s \in [0,r]$,
        and $m \in [0,s]$,
        then
        \begin{gather} \label{math:localprojection:trimmed}
            \left| \ding - \Localproj_{T} \ding \right|_{W^{m,p}\Lambda^{k}(T)}
            \leq
            C_{\BH}
            h_{T}^{s-m}
            | \ding |_{W^{s,p}\Lambda^{k}(T)}
            .
        \end{gather}
        If $\dong \in W^{s,p}\Lambda^{k+1}(T)$, $s \in [0,r]$, 
        and $m \in [0,s]$,
        then
        \begin{gather} \label{math:localprojection:exteriorderivative}
            \left| \dong - \Localproj_{T,\exteriorderivative} \dong \right|_{W^{m,p}\Lambda^{k+1}(T)}
            \leq
            C_{\BH}
            h_{T}^{s-m}
            | \dong |_{W^{s,p}\Lambda^{k+1}(T)}
            .
        \end{gather}
        If $\ding \in \Sobdiff^{1,1}\Lambda^{k}(T)$, 
        then
        \begin{gather} \label{math:localprojection:commutativity}
            \Localproj_{T,\exteriorderivative} \exteriorderivative \ding = \exteriorderivative \Localproj_{T} \ding
            .
        \end{gather}
    \end{subequations}
    Here, $C_{\BH} > 0$ depends only on $n$, $p$, $m$, $s$, the polynomial degree $r$, and $\shapemeasure(T)$.
\end{proposition}

\begin{proof}
    We recall averaged Taylor polynomials. 
    These are well-known for scalar distributions and commute with distributional partial derivatives~\cite{dupont1980polynomial}. 
    We define them on spaces of differential forms componentwise and notice that they commute with the exterior derivative in the sense of distributions.
    For all $k, r \in \bbN_{0}$, this produces bounded projections
    \begin{align*}
        \Localproj_{T}^{k,r} : L^{p}\Lambda^{k}(T) \rightarrow \calP_{r}\Lambda^{k}(T) \subseteq L^{p}\Lambda^{k}(T), \quad p \in [1,\infty],
    \end{align*}
    which satisfy the following: 
    if $\ding \in \Sobdiff^{p,p}\Lambda^{k}(T)$ and $r \geq 1$, then 
    \[
        \exteriorderivative \Localproj_{T}^{k,r} \ding = \Localproj_{T}^{k+1,r-1} \exteriorderivative \ding.
    \]
    There exists a linear projection~\cite{AFW1} 
    \begin{align*}
        I_{T}^{k,r} : \calP_{r}\Lambda^{k}(T) \rightarrow \calP^{-}_{r}\Lambda^{k}(T)
    \end{align*}
    such that $\exteriorderivative I_{T}^{k,r} \ding = \exteriorderivative \ding$ for all $\ding \in \calP_{r}\Lambda^{k}(T)$.
    We set \[\Localproj_{T}^{k,r,-} := I_{T}^{k,r} \Localproj_{T}^{k,r}.\]
    For any $\ding \in \Sobdiff^{p,p}\Lambda^{k}(T)$ it holds that
    \begin{align*}
        \exteriorderivative \Localproj_{T}^{k,r,-} \ding
        =
        \exteriorderivative I_{T}^{k,r} \Localproj_{T}^{k,r} \ding
        =
        \exteriorderivative \Localproj_{T}^{k,r} \ding
        =
        \Localproj_{T}^{k+1,r-1} \exteriorderivative \ding
        .
    \end{align*}
    We use the Bramble--Hilbert lemma:
    if $m,s \in [0,\infty)$ with $m \leq s \leq r+1$,
    then there exists $C_{\BH,0} > 0$ such that for all $\ding \in W^{s,p}\Lambda^{k}(T)$ it holds that 
    \begin{gather}\label{math:localprojection:scottdupont:full}
        \left| \ding - \Localproj_{T}^{k,r} \ding \right|_{W^{m,p}\Lambda^{k}(T)}
        \leq
        C_{\BH,0}
        h_{T}^{s-m}
        | \ding |_{W^{s,p}\Lambda^{k}(T)}
        .
    \end{gather}
    Likewise, 
    if $m,s \in [0,\infty)$ with $m \leq s \leq r$ and $r \geq 1$,
    then there exists $C_{\BH,-} > 0$ such that for all $\ding \in W^{s,p}\Lambda^{k}(T)$ it holds that 
    \begin{gather}\label{math:localprojection:scottdupont:trimmed}
        \left| \ding - \Localproj_{T}^{k,r,-} \ding \right|_{W^{m,p}\Lambda^{k}(T)}
        \leq
        C_{\BH,-}
        h_{T}^{s-m}
        | \ding |_{W^{s,p}\Lambda^{k}(T)}
        .
    \end{gather}
    We construct both interpolants first on a reference simplex and then extend the result to arbitrary simplices via a scaling argument.
    In both cases, the constant depends only on $p$, $m$, $s$, $n$, $r$, and $\shapemeasure(T)$.
    The desired result follows by combining these estimates with the commuting property above.
\end{proof}

A projection inspired by the work of Ern and Guermond~\cite{ern2017finite}, who build on earlier contributions by Oswald~\cite{oswald1993bpx}, has two stages.
First, we take the piecewise polynomial projections $\Localproj_{T}$ of the original field $\ding$ on each element $T$.
That produces a discontinuous finite element approximation.
Second, we compute a conforming finite element approximation by averaging the degrees of freedom of these piecewise projections (with boundary degrees of freedom omitted).
The composition of these two stages is the projection:
\begin{align*}
    \Projection^{EG}
    :
    L^{p}\Lambda^{k}(\Omega)
    \rightarrow
    \calP\Lambda^{k}(\calT,\calU),
    \quad
    \ding
    \mapsto
    \sum_{ \substack{ S \in \calT \\ i \in I(S) \\ S \notin \calU } }
    \card{\supersimplices_{n}(\calT,S)}^{-1}
    \sum_{ \substack{ T \in \supersimplices_{n}(\calT,S) } }
    \phi^{\ast}_{S,i}\left( \Localproj_{T}\ding \right)
    \phi_{S,i}
    .
\end{align*}
Another construction is inspired by the Cl\'ement interpolant and the original Scott--Zhang projection.
For every degree of freedom associated with some simplex $S \in \calT$,
we evaluate it on the piecewise projections $\Localproj_{T_S}$ on the fixed element $T_{S}$ that contains $S$, as defined above.
Thus, we obtain the projection
\begin{align*}
    \Projection^{C}
    :
    L^{p}\Lambda^{k}(\Omega)
    \rightarrow
    \calP\Lambda^{k}(\calT,\calU),
    \quad
    \ding
    \mapsto
    \sum_{ \substack{ S \in \calT \\ i \in I(S) \\ S \notin \calU } }
    \phi^{\ast}_{S,i}\left( \Localproj_{T_{S}}\ding \right)
    \phi_{S,i}
    .
\end{align*}

\begin{remark}
    The operator $\Projection^{EG}$ generalizes the Ern--Guermond projection~\cite{ern2017finite} to differential forms.
    The operator $\Projection^{C}$ appears to be new: it resembles the Cl\'ement interpolant but differs in some consequential ways.

    The Cl\'ement interpolant~\cite{clement1975approximation} is standard for the mathematical theory of finite element methods.
    In most variations, it evaluates each degree of freedom on a patchwise projection, 
    where the local patch contains the subsimplex associated with that degree of freedom.
    In contrast, $\Projection^C$ evaluates each degree of freedom not on patchwise projections but on elementwise projections.

    A comparison with the Scott--Zhang projection is instructive.
    We have derived representations of the degrees of freedom (Theorem~\ref{theorem:roughdof}), which define these functionals not only on polynomial differential forms but also on forms with weak exterior derivative. 
    If we instead apply them directly to the original field, leaving out the piecewise polynomial projection, then operator becomes a Scott--Zhang-type projection, which is well-defined on $\Sobdiff^{p,p}\Lambda^{k}(\Omega)$.

    As we shall see, this minor-looking change to the Cl\'ement interpolant preserves many important features,
    but the modified interpolant is a projection.
\end{remark}

The two operators introduced above are examples of a general construction.
We henceforth assume that for every $S \in \calT$ and every $T \in \supersimplices_{n}(\calT,S)$
we have a fixed non-negative weight $c(S,T)$ such that
\begin{align} \label{math:weight-property}
    1 = \sum_{ \substack{ T \in \supersimplices_{n}(\calT,S) } } c(S,T)
    .
\end{align}
We introduce the operator
\begin{align*}
    \Projection
    :
    L^{p}\Lambda^{k}(\Omega)
    \rightarrow
    \calP\Lambda^{k}(\calT,\calU),
    \quad
    \ding
    \mapsto
    \sum_{ \substack{ S \in \calT \\ i \in I(S) \\ S \notin \calU } }
    \sum_{ \substack{ T \in \supersimplices_{n}(\calT,S) } }
    c(S,T)
    \phi^{\ast}_{S,i}\left( \Localproj_{T}\ding \right)
    \phi_{S,i}
    .
\end{align*}
In other words, 
we initially project the original differential form onto a space of piecewise polynomial differential forms, without imposing any continuity assumptions, 
and then construct a conforming projection by taking weighted averages of degrees of freedom of the local projections.

\begin{example}
    The operators $\Projection^{EG}$ and $\Projection^{C}$ are recovered as examples of the general construction.
    Each example corresponds to a specific choice of weights. 
    The construction yields the Ern--Guermond-type projection $\Projection^{EG}$ if we choose uniform weights for each simplex $S \in \calT$,
    that is,
    \begin{align}
        c(S,T) = \card{\supersimplices_{n}(\calT,S)}^{-1}
        .
    \end{align}
    The construction yields the Cl\'ement-type interpolant if all coefficients vanish except one: for each simplex $S \in \calT$, we set 
    \begin{align}
        c(S,T) = \begin{cases} 1 & \text{ if } T = T_S, \\ 0 & \text{ otherwise. } \end{cases} 
        .
    \end{align}
\end{example}

Before we study the analytical properties of the operator in more detail,
we verify that it is a projection. 

\begin{lemma} \label{lemma:projectionproperty} 
    If $\ding \in \calP\Lambda^{k}(\calT,\calU)$, then $\Projection \ding = \ding$.
\end{lemma}
\begin{proof}
    Since $\ding$ is piecewise polynomial,
    \begin{align*}
        \Projection \ding
        =
        \sum_{ \substack{ S \in \calT \\ i \in I(S) \\ S \notin \calU } }
        \sum_{ \substack{ T \in \supersimplices_{n}(\calT,S) } }
        c(S,T) %
        \phi^{\ast}_{S,i}\left( \ding \restriction_{T} \right)
        \phi_{S,i}
        .
    \end{align*}
    We compute the degrees of freedom. 
    Let $S' \in \calT \setminus \calU$ and $j \in I(S')$. 
    Then 
    \begin{align*}
        \phi_{S',j}^{\ast}\left( \Projection \ding \right)
        & =
        \sum_{ \substack{ S \in \calT \\ i \in I(S) \\ S \notin \calU } }
        \sum_{ \substack{ T \in \supersimplices_{n}(\calT,S) } }
        c(S,T) %
        \phi^{\ast}_{S,i}\left( \ding \restriction_{T} \right)
        \phi_{S',j}^{\ast}\left( \phi_{S,i} \right)
        \\&
        =
        \sum_{ \substack{ T' \in \supersimplices_{n}(\calT,S') } }
        c(S',T')
        \phi_{S',j}^{\ast}\left( \ding \restriction_{T'} \right)
        \\&
        =
        \sum_{ \substack{ T' \in \supersimplices_{n}(\calT,S') } }
        c(S',T')
        \phi_{S',j}^{\ast}\left( \ding \right)
        =
        \phi_{S',j}^{\ast}\left( \ding \right) 
        .
    \end{align*}
    Here, we have used the biorthogonality property from Theorem~\ref{theorem:biorthogonal},
    the conformity of $\ding \in \calP\Lambda^{k}(\calT,\calU)$, and the weight condition~\eqref{math:weight-property}.
    This shows that $\Projection$ is a projection.
\end{proof}

\begin{remark}
    The projection property is also satisfied by Scott--Zhang-type projections~\cite{scott1990finite}.
    The generalized Scott--Zhang projections for differential forms~\cite{licht2021} also have this property, though this was not explicitly proven in that publication.
\end{remark}

\section{Approximation error estimates} \label{section:estimates}

We study the stability of the averaging-based projection and estimate the approximation error for various smoothness classes of differential forms. 
We continue to assume that $\Omega \subseteq \bbR^{n}$ is a bounded Lipschitz domain endowed with a triangulation $\calT$, 
and that $\Gamma \subseteq \partial\Omega$ is a relatively open set that is endowed with a triangulation $\calU \subset \calT$. 

The following stability result establishes that the projection is bounded on Lebesgue spaces.

\begin{theorem}\label{theorem:stability-of-projection}
    Let $p \in [1,\infty]$. 
    For any $T \in \subsimplices_{n}(\calT)$ and $\ding \in L^{p}\Lambda^{k}(\Omega)$, we then have 
    \begin{align*}
        \left\| \Projection \ding \right\|_{L^{p}\Lambda^{k}(T)}
        & \leq
        C
        \sum_{ \substack{ T' \in \subsimplices_{n}(\calT) \\ T \cap T' \neq \emptyset } }
        \left\| \ding \right\|_{L^{p}\Lambda^{k}(T')}
        .
    \end{align*}
    Here, $C > 0$ depends only on $p$, $r$, $n$, and the mesh regularity\footnote{We henceforth say that a constant depends on the mesh regularity if it is increasing in $\shapemeasure(\calT)$, $\comparisonconstant(\calT)$, and $\cardinalitybound(\calT)$.}.
\end{theorem}

\begin{proof}
    As follows from definitions,
    \begin{align*}
        \left\| \Projection \ding \right\|_{L^{p}\Lambda^{k}(T)}
        & \leq
        \sum_{ \substack{ S \in \subsimplices(T) \\ i \in I(S) \\ S \notin \calU } }
        \sum_{ \substack{ T' \in \supersimplices_{n}(\calT,S)} }
        c(S,T') 
        \left\| \phi^{\ast}_{S,i}\left( \Localproj_{T'}\ding \right) \phi_{S,i} \right\|_{L^{p}\Lambda^{k}(T)}
        .
    \end{align*}
    For any $T, T' \in \calT$ sharing a common simplex $S$ and $i \in I(S)$ we have
    \begin{align*}
        \left\| \phi^{\ast}_{S,i}\left( \Localproj_{T'}\ding \right) \phi_{S,i} \right\|_{L^{p}\Lambda^{k}(T)}
        \leq
        \left| \phi^{\ast}_{S,i}\left( \Localproj_{T'}\ding \right) \right|
        \cdot
        \left\| \phi_{S,i} \right\|_{L^{p}\Lambda^{k}(T)}
        .
    \end{align*}
    Theorem~\ref{theorem:biorthogonal} implies
    \begin{gather*}
        \left\| \phi_{S,i} \right\|_{L^{p}\Lambda^{k}(T)}
        \leq
        C_{\mathrm{A},0} h_{T}^{\frac n p - k }
        .
    \end{gather*}
    Together with $\Localproj_{T'}\ding \in \calP\Lambda^{k}(T')$,
    Theorem~\ref{theorem:biorthogonal} and Proposition~\ref{proposition:localprojection} show that
    \begin{align*}
        \left| \phi^{\ast}_{S,i}\left( \Localproj_{T'}\ding \right) \right|
        & \leq
        C_{\mathrm{A},0} 
        h_{T'}^{k - \frac n p}
        \left\| \Localproj_{T'}\ding \right\|_{L^{p}\Lambda^{k}(T')}
        \\&\leq
        C_{\mathrm{A},0} 
        C_{\BH}
        h_{T'}^{k - \frac n p}
        \left\| \ding \right\|_{L^{p}\Lambda^{k}(T')}
        .
    \end{align*}
    Adjacent simplices have comparable diameters, and the number of simplices touching any fixed simplex is uniformly bounded. This proves the claim.
\end{proof}

\begin{remark}\label{remark:stability-of-derivative}
    Under the assumptions of Theorem~\ref{theorem:stability-of-projection}, the following additional inequality holds. 
    If $p \in [1,\infty]$, $T \in \subsimplices_{n}(\calT)$, and $\ding \in L^{p}\Lambda^{k}(\Omega)$, then 
    \begin{align*}
        h_{T} 
        \left\| \exteriorderivative \Projection \ding \right\|_{L^{p}\Lambda^{k+1}(T)}
        & \leq
        C
        \sum_{ \substack{ T' \in \subsimplices_{n}(\calT) \\ T \cap T' \neq \emptyset } }
        \left\| \ding \right\|_{L^{p}\Lambda^{k}(T')}
        ,
    \end{align*}
    where $C > 0$ depends only on $p$, $r$, $n$, and the mesh regularity. 
    This directly follows from the inverse inequality~\eqref{math:biorthogonal:inverseinequality}.
    It is analogous to known results~\cite{ern2017finite,chaumontfrelet2024stable,chaumontfrelet2026quasi}. 
\end{remark}

We next record an additional property of simplicial complexes.
Suppose that we have two $n$-dimensional simplices $T_{0}, T \in \calT$ with non-empty intersection. Given $S \subseteq T_{0} \cap T$, 
we define a \emph{local face-connection from $T_{0}$ to $T$ around $S$} as a sequence $T_{0}, T_{1}, \dots, T_{N}$ 
of pairwise distinct $n$-dimensional simplices of $\calT$ with $T_N = T$ 
such that for all $1 \leq i \leq N$ the intersection $F_{i} := T_{i} \cap T_{i-1}$
satisfies $F_{i} \in \subsimplices_{n-1}(T_{i}) \cap \subsimplices_{n-1}(T_{i-1})$ and $S \subseteq F_{i}$.
We then call $T_{0}$ and $T$ \emph{locally face-connected},
and the triangulation $\calT$ is called \emph{locally face-connected}
if any two $n$-simplices with non-empty intersection are locally face-connected.
The length of any local face-connection is bounded in terms of the mesh regularity.
The following example is important for this manuscript: 
any simplicial complex that triangulates a Lipschitz domain\footnote{In fact, it suffices to have a triangulation of a topological manifold with boundary.} is locally face-connected~\cite{scott1990finite,veeser2016approximating,licht2021}.

The following proposition provides a preliminary error estimate for the averaging-based projection. 
We will explore specializations of the following estimate in different regularity regimes. 

\begin{proposition} \label{proposition:standard-error-representation}
    Let $p \in [1,\infty]$ and $m \in [0,\infty)$.
    There exists $C > 0$
    such that for every $\ding \in W^{m,p}\Lambda^{k}(\calT)$ and $T \in \subsimplices_{n}(\calT)$ we have 
    \begin{align*}
        | \ding - \Projection \ding |_{W^{m,p}\Lambda^{k}(T)}
        &
        \leq
        | \ding - \Localproj_{T} \ding |_{W^{m,p}\Lambda^{k}(T)}
        \\&\qquad
        +
        C
        h_{T}^{\frac n p - k - m}
        \sum_{ \substack{ S \in \subsimplices(T),\; i \in I(S)
        \\ %
        T_{1},T_{2} \in \supersimplices_{n}(\calT,S) \\ T_{1} \cap T_{2} \in \subsimplices_{n-1}(\calT) } } 
        | \phi^{\ast}_{S,i}\left( \Localproj_{T_{1}} \ding \right) - \phi^{\ast}_{S,i}\left( \Localproj_{T_{2}} \ding \right) |
        \\&\qquad
        +
        C
        h_{T}^{\frac n p - k - m}
        \sum_{ \substack{ S \in \subsimplices(T)     \\ i \in I(S) \\ S \in \calU } }
        | \phi^{\ast}_{S,i}\left( \Localproj_{T_{S}} \ding \right) |
        .
    \end{align*}
    Here, $C > 0$ depends only on $p$, $m$, $r$, $n$, and the mesh regularity.
\end{proposition}
\begin{proof}
    We begin with
    \begin{align*}
        | \ding - \Projection \ding |_{W^{m,p}\Lambda^{k}(T)}
        \leq
        | \ding - \Localproj_{T} \ding |_{W^{m,p}\Lambda^{k}(T)}
        +
        | \Localproj_{T} \ding - \Projection \ding |_{W^{m,p}\Lambda^{k}(T)}
        .
    \end{align*}
    We observe that
    \begin{align*}
        \Localproj_{T} \ding
        & =
        \sum_{ \substack{ S \in \subsimplices(T) \\ i \in I(S) \\ S \notin \calU } }
        \phi^{\ast}_{S,i}\left( \Localproj_{T} \ding \right)
        \phi_{S,i}
        +
        \sum_{ \substack{ S \in \subsimplices(T) \\ i \in I(S) \\ S \in \calU } }
        \phi^{\ast}_{S,i}\left( \Localproj_{T} \ding \right)
        \phi_{S,i}
        \\&
        =
        \sum_{ \substack{ S \in \subsimplices(T) \\ i \in I(S) \\ S \notin \calU } }
        \sum_{ \substack{ T' \in \supersimplices_{n}(\calT,S) } }
        c(S,T') %
        \phi^{\ast}_{S,i}\left( \Localproj_{T} \ding \right)
        \phi_{S,i}
        +
        \sum_{ \substack{ S \in \subsimplices(T) \\ i \in I(S) \\ S \in \calU } }
        \phi^{\ast}_{S,i}\left( \Localproj_{T} \ding \right)
        \phi_{S,i}
        .
    \end{align*}
    Therefore, $\Localproj_{T} \ding - (\Projection \ding)$ restricted to $T$ is given by  
    \begin{align*}
        \sum_{ \substack{ S \in \subsimplices(T) \\ i \in I(S) \\ S \notin \calU } }
        \sum_{ \substack{ T' \in \supersimplices_{n}(\calT,S) } }
        c(S,T') 
        \left(
        \phi^{\ast}_{S,i}\left( \Localproj_{T} \ding \right)
        -
        \phi^{\ast}_{S,i}\left( \Localproj_{T'} \ding \right)
        \right)
        \phi_{S,i}
        +
        \sum_{ \substack{ S \in \subsimplices(T) \\ i \in I(S) \\ S \in \calU } }
        \phi^{\ast}_{S,i}\left( \Localproj_{T} \ding \right)
        \phi_{S,i}
        .
    \end{align*}
    Hence, $| \Localproj_{T} \ding - \Projection \ding |_{W^{m,p}\Lambda^{k}(T)}$ is bounded by
    \begin{align*}
        &
        \sum_{ \substack{ S \in \subsimplices(T) \\ i \in I(S) \\ S \notin \calU } }
        \sum_{ \substack{ T' \in \supersimplices_{n}(\calT,S) } }
        \left|
        \phi^{\ast}_{S,i}\left( \Localproj_{T} \ding \right)
        -
        \phi^{\ast}_{S,i}\left( \Localproj_{T'} \ding \right)
        \right|
        \cdot | \phi_{S,i} |_{W^{m,p}\Lambda^{k}(T)}
        \\&\qquad
        +
        \sum_{ \substack{ S \in \subsimplices(T) \\ i \in I(S) \\ S \in \calU } }
        | \phi^{\ast}_{S,i}\left( \Localproj_{T} \ding \right) |
        \cdot | \phi_{S,i} |_{W^{m,p}\Lambda^{k}(T)}
        .
    \end{align*}
    Applying the scaling estimate~\eqref{math:biorthogonal:boundphi:primal} gives 
    \begin{align*}
        | \phi_{S,i} |_{W^{m,p}\Lambda^{k}(T)}
        \leq
        C_{\mathrm{A},m}
        h_{T}^{\frac n p - k - m}
        .
    \end{align*}
    We consider two cases. %
    If $S \in \calU$,
    then there exists a local face-connection $T_{0}, T_{1}, \dots, T_{N}$ from $T_{0} = T$ to $T_{N} = T_{S}$ around $S$.
    We then estimate
    \begin{align*}
        | \phi^{\ast}_{S,i}\left( \Localproj_{T} \ding \right) |
        \leq
        | \phi^{\ast}_{S,i}\left( \Localproj_{T_{S}} \ding \right) |
        +
        \sum_{j=1}^{N}
        | \phi^{\ast}_{S,i}\left( \Localproj_{T_{j-1}} \ding \right) - \phi^{\ast}_{S,i}\left( \Localproj_{T_{j}} \ding \right) |
        .
    \end{align*}
    If $S \notin \calU$,
    then there exists a local face-connection $T_{0}, T_{1}, \dots, T_{N}$ from $T_{0} = T$ to $T_{N} = T'$ around $S$,
    so that 
    \begin{align*}
        \left| 
            \phi^{\ast}_{S,i}\left( \Localproj_{T_{0}} \ding \right) - \phi^{\ast}_{S,i}\left( \Localproj_{T_{N}} \ding \right) 
        \right|
        \leq
        \sum_{j=1}^{N} 
        \left| 
        \phi^{\ast}_{S,i}\left( \Localproj_{T_{j-1}} \ding \right) - \phi^{\ast}_{S,i}\left( \Localproj_{T_{j}} \ding \right) 
        \right|
        .
    \end{align*}
    The length of the local face-connections and the number of adjacent simplices are bounded uniformly in terms of $\cardinalitybound(\calT)$,
    and adjacent simplices have comparable diameters in terms of $\comparisonconstant(\calT)$. The desired result follows. 
\end{proof}

The averaging-based projection restores conformity by averaging degrees of freedom. 
Broadly speaking, the price of averaging is measured by the jumps of the degrees of freedom that are removed.

\subsection{Estimates for weakly differentiable forms}

We estimate the approximation error of the averaging-based projection when the differential form has a weak exterior derivative. 
This relies on the rough degrees of freedom being represented as volume integrals. 
First, we bound the terms associated with degrees of freedom on the boundary part $\Gamma$ in our standard error representation.

\begin{lemma} \label{lemma:error-at-boundary}
    Let $p \in [1,\infty]$. 
    Let $T \in \subsimplices_{n}(\calT)$, $F \in \subsimplices_{n-1}(\calT)$, and $S \in \subsimplices(T)$ 
    with $S \subseteq F \subseteq T$ and $S, F \in \calU$, and $i \in I(S)$. 
    Then for any $\ding \in \Sobdiff^{p,p}\Lambda^{k}(\Omega,\Gamma)$ we have
    \begin{align*}
        | \phi^\ast_{S,i} \left( \Localproj_{T} \ding \right) |
        \leq
        C_{\Xi}
        h_{S}^{ -\frac{n}{p} + k  }
        \Big(
            \| \Localproj_{T} \ding - \ding \|_{L^{p}\Lambda^{k}(T)} 
            +
            h_{S}
            \| \exteriorderivative \Localproj_{T} \ding - \exteriorderivative\ding \|_{L^{p}\Lambda^{k+1}(T)}
        \Big)
        .
    \end{align*}
\end{lemma}
\begin{proof}
    We use the representation of the degrees of freedom in Theorem~\ref{theorem:roughdof}:
    \begin{align*}
        \phi^\ast_{S,i} \left( \Localproj_{T} \ding \right)
        & =
        \int_{F} \xi_{F,S,i} \wedge \trace_{T,F} (\Localproj_{T} \ding)
        \\&=
        o(F,T)
        \int_{T}
        \exteriorderivative\Xi_{T,F,S,i} \wedge (\Localproj_{T} \ding)
        +
        (-1)^{n-k-1} \Xi_{T,F,S,i} \wedge \exteriorderivative(\Localproj_{T} \ding)
        .
    \end{align*}
    By assumption, $\ding \in \Sobdiff^{p,p}\Lambda^{k}(\Omega,\Gamma)$ and $F \subseteq \closure{\Gamma}$. 
    The support properties of $\Xi_{T,F,S,i}$ lead to 
    \begin{align*}
        0
        =
        o(F,T)
        \int_{T}
        \exteriorderivative\Xi_{T,F,S,i} \wedge \ding
        +
        (-1)^{n-k-1} \Xi_{T,F,S,i} \wedge \exteriorderivative \ding
        .
    \end{align*}
    These two identities in combination yield 
    \begin{align*}
        &
        \phi^\ast_{S,i} \left( \Localproj_{T} \ding \right)
        \\&\quad
        =
        o(F,T)
        \int_{T}
        \exteriorderivative\Xi_{T,F,S,i} \wedge \left( \Localproj_{T} \ding - \ding \right)
        +
        (-1)^{n-k-1} \Xi_{T,F,S,i} \wedge \exteriorderivative\left( \Localproj_{T} \ding - \ding \right)
        .
    \end{align*}
    Let $p' \in [1,\infty]$ such that $1 = 1/p + 1/{p'}$.
    H\"older's inequality gives
    \begin{align*}
        \left| \phi^\ast_{S,i} \left( \Localproj_{T} \ding \right) \right|
        & \leq
        \| \exteriorderivative\Xi_{T,F,S,i} \|_{L^{{p'}}\Lambda^{n-k}(T)}
        \| \Localproj_{T} \ding - \ding \|_{L^{p}\Lambda^{k}(T)}
        \\&\qquad+
        \| \Xi_{T,F,S,i} \|_{L^{{p'}}\Lambda^{n-k-1}(T)}
        \| \exteriorderivative \Localproj_{T} \ding - \exteriorderivative\ding \|_{L^{p}\Lambda^{k+1}(T)}
        .
    \end{align*}
    Together with the scaling estimates~\eqref{math:roughdof:scaling:xi} and~\eqref{math:roughdof:scaling:exterior-derivative-xi},
    the desired result follows.
\end{proof}

Next, we bound the differences associated with degrees of freedom on neighboring elements
in our standard error representation as follows.

\begin{lemma} \label{lemma:error-in-between}
    Let $p \in [1,\infty]$.
    Let $T_{1}, T_{2} \in \subsimplices_{n}(\calT)$ be distinct and share a common facet,
    let $S \in \subsimplices(T_{1}) \cap \subsimplices(T_{2})$, and assume $i \in I(S)$.
    Then for every $\ding \in \Sobdiff^{p,p}\Lambda^{k}(\Omega)$ we have 
    \begin{align*}
        &
        | \phi^\ast_{S,i} \left( \Localproj_{T_{1}} \ding \right) - \phi^\ast_{S,i} \left( \Localproj_{T_{2}} \ding \right) |
        \\
        &\qquad
        \leq
        C_{\Xi} h_{S}^{ -\frac{n}{p}+k } 
        \bigg(
            \| \ding - \Localproj_{T_{1}} \ding \|_{L^{p}\Lambda^{k}(T_{1})}
            +
            h_{S}^{}
            \| \exteriorderivative \ding - \exteriorderivative \Localproj_{T_{1}} \ding \|_{L^{p}\Lambda^{k+1}(T_{1})}
            \\&\qquad\qquad\qquad\qquad 
            +
            \| \ding - \Localproj_{T_{2}} \ding \|_{L^{p}\Lambda^{k}(T_{2})}
            +
            h_{S}^{}
            \| \exteriorderivative \ding - \exteriorderivative \Localproj_{T_{2}} \ding \|_{L^{p}\Lambda^{k+1}(T_{2})}
        \bigg)
        .
    \end{align*}
\end{lemma}
\begin{proof}
    Let $F \in \calT$ be the common facet of $T_{1}$ and $T_{2}$.
    We recall the representation of the degrees of freedom from Theorem~\ref{theorem:roughdof} and define 
    \begin{align*}
        \xi_{F,S,i} := \trace_{T_{1},F} \Xi_{T_{1},F,S,i} = \trace_{T_{2},F} \Xi_{T_{2},F,S,i}.
    \end{align*}
    Thus, 
    \begin{align*}
        \phi^\ast_{S,i}( \Localproj_{T_{1}} \ding )
        &
        = 
        \int_{F} \xi_{F,S,i} \wedge \trace_{T_{1},F} \Localproj_{T_{1}} \ding 
        \\&
        =
        o(F,T_{1})
        \int_{T_{1}} \exteriorderivative\Xi_{T_{1},F,S,i} \wedge \Localproj_{T_{1}} \ding + (-1)^{n-k-1}\Xi_{T_{1},F,S,i} \wedge \exteriorderivative \Localproj_{T_{1}} \ding
        \\
        \phi^\ast_{S,i}( \Localproj_{T_{2}} \ding )
        &
        = 
        \int_{F} \xi_{F,S,i} \wedge \trace_{T_{2},F} \Localproj_{T_{2}} \ding 
        \\&
        =
        o(F,T_{2})
        \int_{T_{2}} \exteriorderivative\Xi_{T_{2},F,S,i} \wedge \Localproj_{T_{2}} \ding + (-1)^{n-k-1}\Xi_{T_{2},F,S,i} \wedge \exteriorderivative \Localproj_{T_{2}} \ding
        .
    \end{align*}
    Define $\Xi_{F,S,i} \in L^{\infty}\Lambda^{n-k-1}(\Omega)$ as the differential form satisfying  
    \begin{align*}
        \Xi_{F,S,i} \restriction_{T_{1}} 
        = 
        \Xi_{T_{1},F,S,i}, 
        \quad 
        \Xi_{F,S,i} \restriction_{T_{2}} 
        = 
        \Xi_{T_{2},F,S,i}, \quad \Xi_{F,S,i} \restriction_{ \Omega \setminus ( T_{1} \cup T_{2} ) } = 0. 
    \end{align*}
    We have $\Xi_{F,S,i} \in \Sobdiff^{\infty,\infty}\Lambda^{n-k-1}(\Omega)$ because it is piecewise smooth and has matching (tangential) traces on $F$.
    Therefore, $\Xi_{F,S,i}$ has an essentially bounded weak exterior derivative.
    Since $\Xi_{F,S,i}$ has compact support included in the interior of $T_{1} \cup T_{2}$, 
    we can apply~\eqref{math:integration-by-parts-against-flat-test-forms} to find 
    \begin{align}\label{math:the-integral-over-the-union}
        0
        =
        \int_{T_{1} \cup T_{2}} \exteriorderivative\Xi_{F,S,i} \wedge \ding + (-1)^{n-k-1} \Xi_{F,S,i} \wedge \exteriorderivative\ding
        .
    \end{align}
    Since $T_{1}$ and $T_{2}$ induce opposite orientations on their common facet $F$, we have $o(F,T_{1}) = - o(F,T_{2})$.
    We split the integral~\eqref{math:the-integral-over-the-union} and find
    \begin{align*}
        &
        o(F,T_{1})
        \int_{T_{1}} \exteriorderivative\Xi_{T_{1},F,S,i} \wedge \ding + (-1)^{n-k-1} \Xi_{T_{1},F,S,i} \wedge \exteriorderivative\ding
        \\&\qquad
        -
        o(F,T_{2})
        \int_{T_{2}} \exteriorderivative\Xi_{T_{2},F,S,i} \wedge \ding + (-1)^{n-k-1} \Xi_{T_{2},F,S,i} \wedge \exteriorderivative\ding
        = 0
        .
    \end{align*}
    The combination of these identities shows that
    \begin{align*}
        &
        \phi^\ast_{S,i} \left( \Localproj_{T_{1}} \ding \right) - \phi^\ast_{S,i} \left( \Localproj_{T_{2}} \ding \right)
        \\&\quad =
        o(F,T_{1})
        \int_{T_{1}}
        \exteriorderivative \Xi_{T_{1},F,S,i} \wedge \left( \Localproj_{T_{1}} \ding - \ding \right)
        +
        (-1)^{n-k-1}
        \Xi_{T_{1},F,S,i} \wedge \exteriorderivative \left( \Localproj_{T_{1}} \ding - \ding \right)
        \\&\quad \quad
        -
        o(F,T_{2})
        \int_{T_{2}}
        \exteriorderivative \Xi_{T_{2},F,S,i} \wedge \left( \Localproj_{T_{2}} \ding - \ding \right)
        +
        (-1)^{n-k-1}
        \Xi_{T_{2},F,S,i} \wedge \exteriorderivative \left( \Localproj_{T_{2}} \ding - \ding \right)
        .
    \end{align*}
    Let ${p'} \in [1,\infty]$ with $1/p + 1/{p'} = 1$. 
    H\"older's inequality gives 
    \begin{align*}
        \left|
        \phi^\ast_{S,i} \left( \Localproj_{T_{1}} \ding \right) - \phi^\ast_{S,i} \left( \Localproj_{T_{2}} \ding \right)
        \right|
        &
        \leq
        \| \exteriorderivative \Xi_{T_{1},F,S,i} \|_{L^{{p'}}\Lambda^{n-k}(T_{1})}
        \| \ding - \Localproj_{T_{1}} \ding \|_{L^{p}\Lambda^{k}(T_{1})}
        \\&\quad\quad
        +
        \| \Xi_{T_{1},F,S,i} \|_{L^{{p'}}\Lambda^{n-k-1}(T_{1})}
        \| \exteriorderivative \ding - \exteriorderivative \Localproj_{T_{1}} \ding \|_{L^{p}\Lambda^{k+1}(T_{1})}
        \\&\quad\quad
        +
        \| \exteriorderivative \Xi_{T_{2},F,S,i} \|_{L^{{p'}}\Lambda^{n-k}(T_{2})}
        \| \ding - \Localproj_{T_{2}} \ding \|_{L^{p}\Lambda^{k}(T_{2})}
        \\&\quad\quad
        +
        \| \Xi_{T_{2},F,S,i} \|_{L^{{p'}}\Lambda^{n-k-1}(T_{2})}
        \| \exteriorderivative \ding - \exteriorderivative \Localproj_{T_{2}} \ding \|_{L^{p}\Lambda^{k+1}(T_{2})}
        .
    \end{align*}
    The proof is completed with the inverse inequalities~\eqref{math:roughdof:scaling:xi} and~\eqref{math:roughdof:scaling:exterior-derivative-xi}.
\end{proof}

We can now state our first main result, the piecewise Bramble--Hilbert lemma for differential forms with weak exterior derivative.

\begin{theorem} \label{theorem:mainresult}
    Let $p \in [1,\infty]$ and $m \in [0,\infty)$.
    For every $\ding \in \Sobdiff^{p,p}\Lambda^{k}(\Omega,\Gamma) \cap W^{m,p}\Lambda^{k}(\calT)$ and $T \in \subsimplices_{n}(\calT)$ we have
    \begin{align*}
        \left| \ding - \Projection \ding \right|_{W^{m,p}\Lambda^{k}(T)}
        &
        \leq
        \left| \ding - \Localproj_{T} \ding \right|_{W^{m,p}\Lambda^{k}(T)}
        \\&\qquad
        +
        C
        h_{T}^{-m}
        \sum_{ \substack{ T' \in \subsimplices_{n}(\calT) \\ T' \cap T \neq \emptyset } }
        \| \ding - \Localproj_{T'} \ding \|_{L^{p}\Lambda^{k}(T')}
        +
        h_{T}
        \| \exteriorderivative \ding - \exteriorderivative \Localproj_{T'} \ding \|_{L^{p}\Lambda^{k+1}(T')}
        .
    \end{align*}
    Here, $C > 0$ depends only on $p$, $m$, $r$, $n$, and the mesh regularity.
\end{theorem}
\begin{proof}
    This follows by combining Proposition~\ref{proposition:standard-error-representation}, Lemma~\ref{lemma:error-at-boundary}, and Lemma~\ref{lemma:error-in-between}.
\end{proof}

The main result above opens the possibility for approximation error estimates where the rate of convergence is only determined by piecewise Sobolev--Slobodeckij regularity. 
The two different families of finite element spaces have different convergence properties. This is reflected in the following corollary,
which is an immediate consequence of Theorem~\ref{theorem:mainresult} together with Proposition~\ref{proposition:localprojection}.

\begin{corollary} \label{corollary:mainresult} 
    Let $p \in [1,\infty]$ and $m,s,t \in [0,\infty)$ with $m \leq s \leq r+1$ and $m \leq t \leq r$. 
    If $\calP\Lambda^{k}(\calT,\calU) = \calP_{r}\Lambda^{k}(\calT,\calU)$,
    assume $s \leq r+1$.
    If $\calP\Lambda^{k}(\calT,\calU) = \calP^{-}_{r}\Lambda^{k}(\calT,\calU)$,
    assume $s \leq r$.
    Let 
    \begin{align*}
         T \in \subsimplices_{n}(\calT), 
         \quad 
         \ding \in \Sobdiff^{p,p}\Lambda^{k}(\Omega,\Gamma) \cap W^{s,p}\Lambda^{k}(\calT), 
         \quad 
         \exteriorderivative\ding \in W^{t,p}\Lambda^{k+1}(\calT).
    \end{align*}
    Then
    \begin{align*}
        &
        \left| \ding - \Projection \ding \right|_{W^{m,p}\Lambda^{k}(T)}
        \leq
        C
        \sum_{ \substack{ T' \in \subsimplices_{n}(\calT) \\ T' \cap T \neq \emptyset } }
        h_{T}^{s-m}
        \left| \ding \right|_{W^{s,p}\Lambda^{k}(T')}
        +
        h_{T}^{t+1-m}
        \left| \exteriorderivative \ding \right|_{W^{t,p}\Lambda^{k+1}(T')}
        .
    \end{align*}
    Here, $C > 0$ depends only on $p$, $m$, $s$, $t$, $r$, $n$, and the mesh regularity.
\end{corollary}

\begin{remark}
    A close reading of the proof of Theorem~3.1 in the original contribution of Scott and Zhang~\cite{scott1990finite}, specifically the last inequality on p.~489, already displays the essential techniques for the piecewise Bramble--Hilbert lemma. The inequality does not appear to have been isolated as a separate result. 
    Veeser~\cite{veeser2016approximating} used the result in nonlinear approximation theory,
    and Camacho and Demlow~\cite{camacho2015L2} applied the inequality in the analysis of surface finite element methods. 
\end{remark}

\begin{remark}
    A sequence of commuting projections~\cite{ern2022equivalence,chaumontfrelet2024stable,chaumontfrelet2026quasi} satisfies a similar inequality as in Corollary~\ref{corollary:mainresult}. These projections are defined on the graph-norm spaces, similar to the Scott--Zhang projections for vector fields~\cite{licht2021}, but homogeneous boundary part is fixed in advance. In addition, we point out that these commuting projections involve constructions over entire local patches rather than only single simplices. 
\end{remark}

\subsection{Estimates for Sobolev--Slobodeckij forms}

We turn our attention to approximation error estimates based on Sobolev trace theory.
The construction relies on the rough degrees of freedom represented as integrals over facets. 

We use the trace inequality of Lemma~\ref{lemma:trace-estimate-on-simplices}.
Here, we require a modicum of smoothness to ensure that traces are well-defined; 
this is different from the trace theory based on an integration-by-parts formula, and neither contains the other in general.
The following analysis bears some similarity with Ciarlet's analysis of the Scott--Zhang projection~\cite{ciarlet2013analysis}.

\begin{lemma} \label{lemma:error-at-boundary:sobolevtracetheory}
    Let $p \in [1,\infty]$ and $\ell \in [0,\infty)$.
    Assume that $1/p < \ell \leq 1$ or $\ell = 1$.
    Let $T \in \subsimplices_{n}(\calT)$, $F \in \subsimplices_{n-1}(\calT)$, and $S \in \subsimplices(T)$ with $S \subseteq F \subseteq T$ and $S, F \in \calU$.
    Let $i \in I(S)$. 
    For every $\ding \in W^{\ell,p}\Lambda^{k}(\Omega,\Gamma)$ we have
    \begin{align*}
        | \phi^\ast_{S,i} \left( \Localproj_{T} \ding \right) |
        \leq
        C
        h_{S}^{ -\frac{n}{p} + k  }
        \Big(
            \| \ding - \Localproj_{T} \ding \|_{L^{p}\Lambda^{k}(T)} 
            +
            h_{S}^\ell
            | \ding - \Localproj_{T} \ding |_{W^{\ell,p}\Lambda^{k}(T)} 
        \Big)
        .
    \end{align*}
    Here, $C > 0$ depends only on $p$, $\ell$, $r$, $n$, and the mesh regularity. 
\end{lemma}
\begin{proof} 
    We use the representation of the degrees of freedom in Theorem~\ref{theorem:roughdof}.
    Write $\xi_{F,S,i} := \trace_{T,F} \Xi_{T,F,S,i}$.
    Since $F \in \calU$ and $\ding$ has homogeneous trace on $\Gamma$, we have $\trace_{T,F} \ding = 0$. 
    Therefore, 
    \begin{align*}
        \phi^\ast_{S,i} \left( \Localproj_{T} \ding \right)
        &
        =
        \int_{F} \xi_{F,S,i} \wedge \trace_{T,F} ( \Localproj_{T} \ding )
        \\&
        =
        \int_{F} \xi_{F,S,i} \wedge \trace_{T,F} ( \Localproj_{T} \ding - \ding )
        .
    \end{align*}
    Let $q \in [1,\infty]$ such that $1 = 1/p + 1/{q}$.
    By H\"older's inequality,
    \begin{align*}
        \left| \phi^\ast_{S,i} \left( \Localproj_{T} \ding \right) \right|
        & 
        \leq 
        \| \xi_{F,S,i} \|_{L^{{q}}\Lambda^{n-k-1}(F)}
        \| \ding - \Localproj_{T} \ding \|_{L^{p}\Lambda^{k}(F)}
        \\&
        \leq
        C_{\Xi} h_{S}^{ \frac{n-1}{q} - n+k+1 }
        \| \ding - \Localproj_{T} \ding \|_{L^{p}\Lambda^{k}(F)} 
        \\&
        \leq
        C_{\Xi} C_{\trace}
        h_{T}^{ \frac{n-1}{q} - n+k+1 }
        \left(
        h_{T}^{-\frac 1 p}
        \| \ding - \Localproj_{T} \ding \|_{L^{p}\Lambda^{k}(T)}
        +
        h_{T}^{\ell-\frac 1 p}
        \left| \ding - \Localproj_{T} \ding \right|_{W^{\ell,p}\Lambda^{k}(T)}
        \right)
        .
    \end{align*}
    Here, we have used the scaling inequality~\eqref{math:roughdof:scaling:trace} and the trace inequality on simplices, Lemma~\ref{lemma:trace-estimate-on-simplices}.
    We notice that $\tfrac{n-1}{q} - n+k+1 - \tfrac 1 p = k - n + \tfrac{n}{q} = k - \tfrac{n}{p}$. 
    The desired estimate follows. 
\end{proof}
\begin{lemma} \label{lemma:error-in-between:sobolevtracetheory}
    Let $p \in [1,\infty]$ and $\ell \in [0,\infty)$. 
    Assume that $1/p < \ell \leq 1$ or $\ell = 1$.
    Let $T_{1}, T_{2} \in \subsimplices_{n}(\calT)$ be distinct and share a common facet, let $S \in \subsimplices(T_{1}) \cap \subsimplices(T_{2})$, and $i \in I(S)$.
    Then for every $\ding \in W^{\ell,p}\Lambda^{k}(\Omega)$ we have
    \begin{align*}
        | \phi^\ast_{S,i} \left( \Localproj_{T_{1}} \ding \right) - \phi^\ast_{S,i} \left( \Localproj_{T_{2}} \ding \right) |
        &
        \leq
        C
        h_{S}^{ -\frac{n}{p}+k }
        \bigg(
        \| \ding - \Localproj_{T_{1}} \ding \|_{L^{p}\Lambda^{k}(T_{1})}
        +
        h_{S}^{\ell}
        | \ding - \Localproj_{T_{1}} \ding |_{W^{\ell,p}\Lambda^{k}(T_{1})}
        \\&\quad
        \qquad\qquad\qquad
        +
        \| \ding - \Localproj_{T_{2}} \ding \|_{L^{p}\Lambda^{k}(T_{2})}
        +
        h_{S}^{\ell}
        | \ding - \Localproj_{T_{2}} \ding |_{W^{\ell,p}\Lambda^{k}(T_{2})}
        \bigg)
        .
    \end{align*}
    Here, $C > 0$ depends only on $p$, $\ell$, $r$, $n$, and the mesh regularity. %
\end{lemma}
\begin{proof}
    Let $F \in \subsimplices_{n-1}(\calT)$ be the common facet of $T_{1}$ and $T_{2}$.
    The differential form $\ding$ satisfies $\trace_{T_{1},F} \ding = \trace_{T_{2},F} \ding$.
    Therefore,  
    \begin{align*}
        &
        \phi^\ast_{S,i} \left( \Localproj_{T_{1}} \ding \right) - \phi^\ast_{S,i} \left( \Localproj_{T_{2}} \ding \right)
        \\&\quad 
        =
        \int_F \xi_{F,S,i} \wedge \trace_{T_1,F}( \Localproj_{T_1} \ding )
        -
        \int_F \xi_{F,S,i} \wedge \trace_{T_2,F}( \Localproj_{T_2} \ding )
        \\&\quad 
        =
        \int_F \xi_{F,S,i} \wedge \trace_{T_1,F}( \Localproj_{T_1} \ding - \ding )
        -
        \int_F \xi_{F,S,i} \wedge \trace_{T_2,F}( \Localproj_{T_2} \ding - \ding )
        .
    \end{align*}
    Let $j \in \{1,2\}$ and let $q \in [1,\infty]$ with $\tfrac{1}{p} + \tfrac{1}{q} = 1$. 
    By H\"older's inequality, the scaling inequality~\eqref{math:roughdof:scaling:trace}, and the trace inequality from Lemma~\ref{lemma:trace-estimate-on-simplices},
    \begin{align*}
        &
        \left| \int_{F} \trace_{T_{j},F} \Xi_{T_{j},F,S,i} \wedge \trace_{T_{j},F} \left( \Localproj_{T_{j}} \ding - \ding \right) \right|
        \\&\quad
        \leq 
        \left\| \trace_{T_{j},F} \Xi_{T_{j},F,S,i} \right\|_{L^{q}\Lambda^{n-k-1}(F)}
        \left\| \trace_{T_{j},F} ( \Localproj_{T_{j}} \ding - \ding ) \right\|_{L^{p}\Lambda^{k}(F)}
        \\&\quad
        \leq 
        C_{\Xi} 
        h_{S}^{ \frac{n-1}{q} - n+k+1}
        \left\| \trace_{T_{j},F} ( \Localproj_{T_{j}} \ding - \ding ) \right\|_{L^{p}\Lambda^{k}(F)}
        \\&\quad
        \leq 
        C_{\Xi} 
        C_{\trace}
        h_{S}^{ \frac{n-1}{q} - n+k+1}
        \left(
        h_{F}^{-\frac 1 p}
        \| \Localproj_{T_{j}} \ding - \ding \|_{L^{p}\Lambda^{k}(T_{j})} 
        +
        h_{F}^{\ell-\frac 1 p}
        \left| \Localproj_{T_{j}} \ding - \ding \right|_{W^{\ell,p}\Lambda^{k}(T_{j})}
        \right)
        .
    \end{align*}
    Again, $\tfrac{n-1}{q} - n+k+1 - \tfrac 1 p = k - \tfrac{n}{p}$. 
    The desired estimate follows. 
\end{proof}

\begin{theorem} \label{theorem:mainresult:sobolevtracetheory}
    Let $p \in [1,\infty]$ and $\ell, m \in [0,\infty)$.
    Assume that $1/p < \ell \leq 1$ or $\ell = 1$.
    For every $\ding \in W^{\ell,p}\Lambda^{k}(\Omega,\Gamma) \cap W^{m,p}\Lambda^{k}(\calT)$ and $T \in \subsimplices_{n}(\calT)$ we have
    \begin{align*}
        \left| \ding - \Projection \ding \right|_{W^{m,p}\Lambda^{k}(T)}
        &
        \leq
        \left| \ding - \Localproj_{T} \ding \right|_{W^{m,p}\Lambda^{k}(T)}
        \\&\qquad
        +
        C
        \sum_{ \substack{ T' \in \subsimplices_{n}(\calT) \\ T' \cap T \neq \emptyset } }
        h_{T}^{-m}
        \| \ding - \Localproj_{T'} \ding \|_{L^{p}\Lambda^{k}(T')}
        +
        h_{T}^{\ell-m}
        \left| \ding - \Localproj_{T'} \ding \right|_{W^{\ell,p}\Lambda^{k}(T')}
        .
    \end{align*}
    Here, $C > 0$ depends only on $p$, $\ell$, $m$, $r$, $n$, and the mesh regularity.
\end{theorem}
\begin{proof}
    This follows by combining Proposition~\ref{proposition:standard-error-representation}, Lemma~\ref{lemma:error-at-boundary:sobolevtracetheory}, and Lemma~\ref{lemma:error-in-between:sobolevtracetheory}.
\end{proof}

\begin{corollary} \label{corollary:mainresult:sobolevtracetheory}
    Let $p \in [1,\infty]$ and $\ell, m, s \in [0,\infty)$ with $\max(\ell,m) \leq s$.
    Assume that $1/p < \ell \leq 1$ or $\ell = 1$.
    If $\calP\Lambda^{k}(\calT,\calU) = \calP_{r}\Lambda^{k}(\calT,\calU)$,
    assume $s \leq r+1$.
    If $\calP\Lambda^{k}(\calT,\calU) = \calP^{-}_{r}\Lambda^{k}(\calT,\calU)$,
    assume $s \leq r$.
    Let 
    \begin{align*}
         T \in \subsimplices_{n}(\calT), 
         \quad 
         \ding \in W^{\ell,p}\Lambda^{k}(\Omega,\Gamma) \cap W^{m,p}\Lambda^{k}(\calT) \cap W^{s,p}\Lambda^{k}(\calT).
    \end{align*}
    Then 
    \begin{align*}
        &
        \left| \ding - \Projection \ding \right|_{W^{m,p}\Lambda^{k}(T)}
        \leq
        C
        h_{T}^{s-m}
        \sum_{ \substack{ T' \in \subsimplices_{n}(\calT) \\ T' \cap T \neq \emptyset } }
        \left| \ding \right|_{W^{s,p}\Lambda^{k}(T')}
        .
    \end{align*}
    Here, $C > 0$ depends only on $p$, $\ell$, $m$, $s$, $r$, $n$, and the mesh regularity.
\end{corollary}
\begin{proof}
    This follows from Theorem~\ref{theorem:mainresult:sobolevtracetheory} and Proposition~\ref{proposition:localprojection}.
\end{proof}

\begin{remark}
    A differential form in $W^{1,p}\Lambda^{k}(\Omega)$ has well-defined traces in the sense of Sobolev theory
    and via the integration-by-parts formula, and both traces agree.
    Those two approaches allow us to define traces, 
    either in the weak sense on differential forms in $\Sobdiff^{p,p}\Lambda^{k}(\Omega)$,
    or in the strong sense on Sobolev--Slobodeckij spaces of differential forms.
\end{remark}

\begin{remark}
    The following informal observation is of interest.
    An interpolant that is bounded on Lebesgue spaces can respect homogeneous boundary conditions on $\Gamma$
    only by \enquote{hard-coding} them in the definition of the interpolant, and will satisfy a Bramble--Hilbert-type inequality near that boundary part only for sufficiently regular differential forms. The interpolant constructed here is an example. 
    In contrast, if we want an interpolant to satisfy such an inequality for all differential forms of sufficient regularity, regardless of whether they satisfy the boundary conditions or not,
    then we must restrict the interpolant to differential forms with some modest amount of smoothness.
\end{remark}

\subsection{Further estimates for Sobolev--Slobodeckij forms}

We have proven approximation error estimates when the differential form has a weak exterior derivative or enough global smoothness to define traces. 
In particular, these regularity regimes admit the definition of homogeneous boundary conditions that are incorporated into the definition of the averaging-based projection. 

A third regularity regime enables approximation error estimates closer in spirit to the original analysis of the Cl\'ement interpolant, where the regularity can be arbitrarily low but must be global. 
However, we estimate local approximation errors only away from the boundary part where homogeneous boundary conditions are imposed; 
this restriction is to be expected in a regime of very low regularity, when boundary traces are no longer defined. 

For any facet $F \in \subsimplices_{n-1}(\calT)$, we let $\Omega_{F}$ denote the polyhedral domain triangulated by the $n$-dimensional simplices of the triangulation that contain $F$.
If $F$ is a facet at the boundary, then there is only one $n$-simplex $T$ containing $F$ and hence $\closure{\Omega_{F}}=T$.
If instead $F$ is an interior facet, then there are exactly two $n$-simplices $T_{1}, T_{2} \in \calT$
that contain $F$, and hence $\closure{\Omega_{F}} = T_{1} \cup T_{2}$.

\begin{theorem} \label{theorem:lowregularityestimate}
    Let $p \in [1,\infty]$ and $m,s \in [0,\infty)$ with $m \leq s$.
    If $\calP\Lambda^{k}(\calT,\calU) = \calP_{r}\Lambda^{k}(\calT,\calU)$, assume $s \leq r+1$.
    If $\calP\Lambda^{k}(\calT,\calU) = \calP^{-}_{r}\Lambda^{k}(\calT,\calU)$, assume $s \leq r$.
    Let $T \in \subsimplices_{n}(\calT)$ with $\subsimplices(T) \cap \calU = \emptyset$.
    Then for any $\ding \in W^{s,p}\Lambda^{k}(\Omega)$ we have
    \begin{align*}
        &
        \left| \ding - \Projection \ding \right|_{W^{m,p}\Lambda^{k}(T)}
        \leq
        C
        h^{s-m}_{T}
        \sum_{ \substack{ F \in \subsimplices_{n-1}(\calT) \\ T \cap F \neq \emptyset } }
        \left| \ding \right|_{W^{s,p}\Lambda^{k}(\Omega_{F})}
        .
    \end{align*}
    Here, $C > 0$ depends only on $p$, $m$, $s$, $r$, $n$, and the mesh regularity.
\end{theorem}
\begin{proof}
    By Proposition~\ref{proposition:standard-error-representation} and the assumption $\subsimplices(T) \cap \calU = \emptyset$, we get 
    \begin{align*}
        &
        | \ding - \Projection \ding |_{W^{m,p}\Lambda^{k}(T)}
        \\&\qquad 
        \leq
        | \ding - \Localproj_{T} \ding |_{W^{m,p}\Lambda^{k}(T)} 
        +
        C
        h_{T}^{\frac n p - k - m}
        \sum_{ \substack{ S \in \subsimplices(T),\; i \in I(S)
        \\
        T_{1},T_{2} \in \supersimplices_{n}(\calT,S) \\ T_{1} \cap T_{2} \in \subsimplices_{n-1}(\calT) } } 
        | \phi^{\ast}_{S,i}\left( \Localproj_{T_{1}} \ding \right) - \phi^{\ast}_{S,i}\left( \Localproj_{T_{2}} \ding \right) |
        .
    \end{align*}
    Consider now any $T_{1}, T_{2} \in \supersimplices_{n}(\calT,S)$ 
    sharing a common facet $F = T_1 \cap T_2 \in \subsimplices_{n-1}(\calT)$.
    Let $i \in I(S)$.
    
    We define $\calQ\Lambda^{k}(\Omega_{F}) \subseteq C^{\infty}\Lambda^{k}(\Omega_{F})$ as the space of smooth $k$-forms whose coefficients are polynomials of either degree $r$ if $\calP\Lambda^{k}(T_{1}) = \calP_{r}\Lambda^{k}(T_{1})$ or of degree $r-1$ if $\calP\Lambda^{k}(T_{1}) = \calP^{-}_{r}\Lambda^{k}(T_{1})$.
    
    For any $\ding_{F} \in \calQ\Lambda^{k}(\Omega_{F})$, 
    the degree of freedom $\phi^\ast_{S,i} ( \ding_{F} )$ is well-defined, and thus we derive 
    \begin{align*}
        \phi^\ast_{S,i} \left( \Localproj_{T_{1}} \ding \right)
        -
        \phi^\ast_{S,i} \left( \Localproj_{T_{2}} \ding \right)
        =
        \phi^\ast_{S,i} \left( \Localproj_{T_{1}} \ding \right)
        -
        \phi^\ast_{S,i} \left( \ding_{F} \right)
        +
        \phi^\ast_{S,i} \left( \ding_{F} \right)
        -
        \phi^\ast_{S,i} \left( \Localproj_{T_{2}} \ding \right)
        .
    \end{align*}
    Let $j \in \{1,2\}$. We use the scaling inequality~\eqref{math:biorthogonal:boundphi:dual} to see that
    \begin{align*}
        \left|
        \phi^\ast_{S,i} \left( \ding_{F} \right)
        -
        \phi^\ast_{S,i} \left( \Localproj_{T_{j}} \ding \right)
        \right|
        & \leq
        C_{\mathrm{A}}
        h_{S}^{-\frac n p + k} 
        \left\|
            \ding_{F}
            -
            \Localproj_{T_{j}} \ding
        \right\|_{L^{p}\Lambda^{k}(T_{j})}
        .
    \end{align*}
    The triangle inequality gives 
    \begin{align*}
        \left\| 
        \ding_{F}
        -
        \Localproj_{T_{j}} \ding
        \right\|_{L^{p}\Lambda^{k}(T_{j})}
        &
        \leq
        \left\|
        \ding_{F}
        -
        \ding
        \right\|_{L^{p}\Lambda^{k}(T_{j})}
        +
        \left\|
        \ding
        -
        \Localproj_{T_{j}} \ding
        \right\|_{L^{p}\Lambda^{k}(T_{j})}
        \\&
        \leq 
        \left\|
            \ding_{F}
            -
            \ding
        \right\|_{L^{p}\Lambda^{k}(\Omega_{F})}
        +
        \left\|
            \ding
            -
            \Localproj_{T_{j}} \ding
        \right\|_{L^{p}\Lambda^{k}(T_{j})}
        .
    \end{align*}
    We apply Proposition~\ref{proposition:localprojection} to the norms over single simplices. 
    The norms over face patches are handled slightly differently. 
    
    The domain $\Omega_{F}$ is star-shaped with respect to a ball centered at the barycenter of $F$ and whose radius is proportional to $h_{T_1}$ and $h_{T_{2}}$ up to a constant that depends on $n$ and $\shapemeasure(\calT)$; see~\cite{licht2025computablepoincarefriedrichsconstantslp}. 
    Taking averaged Taylor polynomials with respect to that ball defines an operator 
    \begin{align*} 
        \pi_{F} : L^{p}\Lambda^{k}(\Omega_{F}) \to \calQ\Lambda^{k}(\Omega_{F}) \subseteq L^{p}\Lambda^{k}(\Omega_{F})
        .
    \end{align*}
    For any $s \in [0,\infty)$ under the conditions in the statement of the theorem, 
    there exists $C > 0$ depending only on $p$, $s$, $r$, $n$ and $\shapemeasure(\calT)$ such that 
    \[
        \left\| \ding - \pi_{F} \ding \right\|_{L^p\Lambda^k(\Omega_{F})} 
        \leq 
        C h_{F}^s \left| \ding \right|_{W^{s,p}\Lambda^k(\Omega_{F})}.
    \]
    If we choose $\ding_{F} = \pi_{F} \ding$, then the desired estimate follows. 
    The proof is complete. 
\end{proof}

\section{Global and local approximation errors} \label{section:equivalence-of-global-and-local}

Piecewise polynomial approximation of differential forms is well understood.
The approximations are sought within discontinuous, or \emph{non-conforming}, finite element spaces, where the approximation on each element uses only local data.
How much approximation quality is lost if we instead approximate by conforming finite element spaces, which impose boundary conditions and conformity? 
As it turns out, conforming and non-conforming finite element approximations have comparable errors in many situations.
This means that imposing boundary conditions or conformity does not degrade the approximation quality in comparison to a discontinuous finite element approximation. 
This phenomenon has received attention in the literature for $\bfH(\curl,\Omega)$ and $\bfH(\divergence,\Omega)$~\cite{ern2022equivalence,chaumontfrelet2021equivalence}. 
In this section, we prove similar results in the framework of finite element exterior calculus. 

We continue to assume that $\Omega \subseteq \bbR^{n}$ is a bounded Lipschitz domain endowed with a triangulation $\calT$, 
and that $\Gamma \subseteq \partial\Omega$ is a relatively open set that is endowed with a triangulation $\calU \subseteq \calT$.

We begin with the relevant error quantities. 
For $p \in [1,\infty]$ and $\ding \in \Sobdiff^{p,p}\Lambda^{k}(\Omega,\Gamma)$, we define
\begin{align*}
    E_{p}(\ding)
    &
    :=
    \min_{ \ding_{h} \in \calP\Lambda^{k}(\calT,\calU) }
    \left(
    \| \ding - \ding_{h} \|_{L^{p}\Lambda^{k}(\Omega)}^{p}
    + 
    \sum_{ T \in \subsimplices_{n}(\calT) }
    h_{T}^{p}
    \| \exteriorderivative\ding - \exteriorderivative\ding_{h} \|_{L^{p}\Lambda^{k+1}(T)}^{p}
    \right)^{\frac 1 p}
    , \quad p < \infty, 
    \\
    {}
    E_{\infty}(\ding)
    &
    :=
    \min_{ \ding_{h} \in \calP\Lambda^{k}(\calT,\calU) }
    \| \ding - \ding_{h} \|_{L^{\infty}\Lambda^{k}(\Omega)}
    + 
    \max_{ T \in \subsimplices_{n}(\calT) } 
    h_{T}
    \| \exteriorderivative\ding - \exteriorderivative\ding_{h} \|_{L^{\infty}\Lambda^{k+1}(T)}
    .
\end{align*}
These quantities measure the best-approximation errors of the differential form $\ding \in \Sobdiff^{p,p}\Lambda^{k}(\Omega,\Gamma)$
by members of $\calP\Lambda^{k}(\calT,\calU)$ in terms of a mesh-dependent weighted modification of the $\Sobdiff^{p,p}\Lambda^{k}(\Omega,\Gamma)$ norm.

We define local error terms on each simplex.
We minimize over the local polynomial space without imposing any local boundary or conformity conditions.
For $p \in [1,\infty]$, any full-dimensional simplex $T \in \subsimplices_{n}(\calT)$
and any differential form $\ding \in \Sobdiff^{p,p}\Lambda^{k}(\Omega,\Gamma)$,
we define
\begin{align*}
    e_{p,T}(\ding)
    &
    :=
    \min_{ \ding_{h} \in \calP\Lambda^{k}(T) }
    \left(
    \| \ding - \ding_{h} \|_{L^{p}\Lambda^{k}(T)}^{p}
    +
    h_{T}^{p}
    \| \exteriorderivative\ding - \exteriorderivative\ding_{h} \|_{L^{p}\Lambda^{k+1}(T)}^{p}
    \right)^{\frac 1 p}
    , \quad p < \infty,
    \\
    e_{\infty,T}(\ding)
    &
    :=
    \min_{ \ding_{h} \in \calP\Lambda^{k}(T) }
    \| \ding - \ding_{h} \|_{L^{\infty}\Lambda^{k}(T)}
    +
    h_{T}
    \| \exteriorderivative\ding - \exteriorderivative\ding_{h} \|_{L^{\infty}\Lambda^{k+1}(T)}
    .
\end{align*}
These quantities use the same finite element space on each simplex but no boundary conditions or conformity are imposed.
They measure the approximation error in local terms.

We want to compare the global and the local approximation errors.
The local approximation errors yield lower bounds for the global approximation errors: 
\begin{align*} 
    \sum_{ T \in \subsimplices_{n}(\calT) } e_{p,T}(\ding)^{p} 
    &
    \leq E_{p}(\ding)^{p},
    \quad \ding \in \Sobdiff^{p,p}\Lambda^{k}(\Omega,\Gamma), \quad p < \infty,
    \\
    \max_{ T \in \subsimplices_{n}(\calT) } e_{\infty,T}(\ding) 
    &
    \leq 
    E_{\infty}(\ding),
    \quad \ding \in \Sobdiff^{\infty,\infty}\Lambda^{k}(\Omega,\Gamma).
\end{align*}
This simply states that the best-approximation error with respect to fully discontinuous finite element spaces
is a lower bound for the best-approximation error with respect to conforming finite element spaces.
The following theorem, which is modeled after~\cite[Theorem~3.3]{ern2022equivalence} and~\cite[Theorem~2]{chaumontfrelet2021equivalence}, provides a converse to these inequalities.

\begin{theorem} \label{theorem:localvsglobal}
    Let $p \in [1,\infty]$.
    Let $\ding \in \Sobdiff^{p,p}\Lambda^{k}(\Omega,\Gamma)$. 
    Then
    \begin{align*}
        E_{p}(\ding)^{p} &\leq C \sum_{ T \in \subsimplices_{n}(\calT) } e_{p,T}(\ding)^{p}, \quad p < \infty 
        \\
        E_{\infty}(\ding) &\leq C \max_{ T \in \subsimplices_{n}(\calT) } e_{\infty,T}(\ding). 
    \end{align*}
    Here, $C > 0$ depends only on $p$, $r$, $n$, and the mesh regularity.
\end{theorem}

\begin{proof}
    In what follows, $C > 0$ depends only on $p$, $r$, $n$, and the mesh regularity.
    By definition of $E_{p}(\ding)$, we have 
    \begin{align*}
        E_{p}(\ding)^{p}
        &
        \leq
        \| \ding - \Projection\ding \|_{L^{p}\Lambda^{k}(\Omega)}^{p}
        +
        \sum_{ T \in \subsimplices_{n}(\calT) }
        h_{T}^{p}
        \| \exteriorderivative\ding - \exteriorderivative\Projection\ding \|_{L^{p}\Lambda^{k+1}(T)}^{p}
        , \quad p < \infty,
        \\
        E_{\infty}(\ding)
        &
        \leq
        \| \ding - \Projection\ding \|_{L^{\infty}\Lambda^{k}(\Omega)}
        +
        \max_{ T \in \subsimplices_{n}(\calT) }
        h_{T}
        \| \exteriorderivative\ding - \exteriorderivative\Projection\ding \|_{L^{\infty}\Lambda^{k+1}(T)}
        .
    \end{align*}
    For each simplex $T \in \subsimplices_{n}(\calT)$, we first use the triangle inequality:
    \begin{align*}
        \| \exteriorderivative\ding - \exteriorderivative\Projection\ding \|_{L^{p}\Lambda^{k+1}(T)}
        \leq 
        \| \exteriorderivative\ding - \exteriorderivative\Localproj_{T}\ding \|_{L^{p}\Lambda^{k+1}(T)}
        +
        \| \exteriorderivative\Localproj_{T}\ding - \exteriorderivative\Projection\ding \|_{L^{p}\Lambda^{k+1}(T)}
        .
    \end{align*}
    The difference $\Localproj_{T} \omega - \Projection\omega$ restricts to $\calP\Lambda^{k}(T)$. We use the inverse inequality~\eqref{math:biorthogonal:inverseinequality}:
    \begin{align*} 
        \| \exteriorderivative \Localproj_{T} \ding - \exteriorderivative\Projection\ding \|_{L^{p}\Lambda^{k+1}(T)}
        &
        \leq
        C_{\mathrm{B}}
        h_{T}^{-1}
        \| \Localproj_{T} \ding - \Projection\ding \|_{L^{p}\Lambda^{k}(T)}
        \\&
        \leq
        C_{\mathrm{B}}
        h_{T}^{-1}
        \left( 
            \| \Localproj_{T} \ding - \ding \|_{L^{p}\Lambda^{k}(T)}
            + 
            \| \ding - \Projection\ding \|_{L^{p}\Lambda^{k}(T)}
        \right)
        .
    \end{align*}
    Consequently, if $p < \infty$, then 
    \begin{align*}
        E_{p}(\ding)^{p}
        &
        \leq
        C
        \| \ding - \Projection\ding \|_{L^{p}\Lambda^{k}(\Omega)}^{p}
        \\&\qquad\qquad
        +
        C
        \sum_{ T \in \subsimplices_{n}(\calT) }
        \| \ding - \Localproj_{T} \ding \|_{L^{p}\Lambda^{k}(T)}^{p}
        +
        C
        \sum_{ T \in \subsimplices_{n}(\calT) }
        h_{T}^{p}
        \| \exteriorderivative\ding - \exteriorderivative\Localproj_{T}\ding \|_{L^{p}\Lambda^{k+1}(T)}^{p},
    \end{align*}
    and if $p = \infty$, then 
    \begin{align*}
        E_{\infty}(\ding)
        &
        \leq
        C
        \| \ding - \Projection\ding \|_{L^{\infty}\Lambda^{k}(\Omega)}
        \\&\qquad\qquad
        +
        C
        \max_{ T \in \subsimplices_{n}(\calT) }
        \| \ding - \Localproj_{T} \ding \|_{L^{\infty}\Lambda^{k}(T)}
        +
        C
        \max_{ T \in \subsimplices_{n}(\calT) }
        h_{T}
        \| \exteriorderivative\ding - \exteriorderivative\Localproj_{T}\ding \|_{L^{\infty}\Lambda^{k+1}(T)}
        .
    \end{align*}
    By Theorem~\ref{theorem:mainresult} with $m = 0$, for each $T \in \subsimplices_{n}(\calT)$:
    \begin{align*}
        \left\| \ding - \Projection \ding \right\|_{L^{p}\Lambda^{k}(T)}
        &
        \leq
        \left\| \ding - \Localproj_{T} \ding \right\|_{L^{p}\Lambda^{k}(T)}
        \\&\qquad
        +
        C
        \sum_{ \substack{ T' \in \subsimplices_{n}(\calT) \\ T' \cap T \neq \emptyset } }
        \| \ding - \Localproj_{T'} \ding \|_{L^{p}\Lambda^{k}(T')}
        +
        h_{T}
        \| \exteriorderivative \ding - \exteriorderivative \Localproj_{T'} \ding \|_{L^{p}\Lambda^{k+1}(T')}
        .
    \end{align*}
    Because the relevant element patches have uniformly bounded overlap,
    we find
    \begin{align*}
        E_{p}(\ding)^{p}
        & \leq
        C
        \sum_{ \substack{ T \in \subsimplices_{n}(\calT) } }
        \| \ding - \Localproj_{T} \ding \|_{L^{p}\Lambda^{k}(T)}^{p}
        + 
        C
        \sum_{ \substack{ T \in \subsimplices_{n}(\calT) } }
        h_{T}^{p}
        \| \exteriorderivative\ding - \exteriorderivative\Localproj_{T} \ding \|_{L^{p}\Lambda^{k+1}(T)}^{p}
        , \quad p < \infty,
        \\
        E_{\infty}(\ding)
        &\leq
        C
        \max_{ T \in \subsimplices_{n}(\calT) }
        \| \ding - \Localproj_{T} \ding \|_{L^{\infty}\Lambda^{k}(T)}
        + 
        C
        \max_{ \substack{ T \in \subsimplices_{n}(\calT) } }
        h_{T}
        \| \exteriorderivative\ding - \exteriorderivative\Localproj_{T} \ding \|_{L^{\infty}\Lambda^{k+1}(T)}
        .
    \end{align*}
    We use the fact that the local projections in Proposition~\ref{proposition:localprojection} commute with the exterior derivative.
    For any $T \in \subsimplices_{n}(\calT)$, we estimate
    \begin{align*}
        \| \ding - \Localproj_{T} \ding \|_{L^{p}\Lambda^{k}(T)}
        &
        \leq
        \min_{ \ding_{h} \in \calP\Lambda^{k}(T) }
        \| \ding - \ding_{h} + \Localproj_{T} \ding_{h} - \Localproj_{T} \ding \|_{L^{p}\Lambda^{k}(T)}
        \\&
        \leq 
        C
        \min_{ \ding_{h} \in \calP\Lambda^{k}(T) }
        \| \ding - \ding_{h} \|_{L^{p}\Lambda^{k}(T)}
        \\
        \| \exteriorderivative \ding - \exteriorderivative \Localproj_{T} \ding \|_{L^{p}\Lambda^{k+1}(T)}
        &
        \leq
        \min_{ \ding_{h} \in \calP\Lambda^{k}(T) }
        \| \exteriorderivative \ding - \exteriorderivative \ding_{h} + \exteriorderivative \Localproj_{T} \ding_{h} - \exteriorderivative \Localproj_{T} \ding \|_{L^{p}\Lambda^{k+1}(T)}
        \\&
        =
        \min_{ \ding_{h} \in \calP\Lambda^{k}(T) }
        \| \exteriorderivative \ding - \exteriorderivative \ding_{h} + \Localproj_{T,\exteriorderivative} \exteriorderivative \ding_{h} - \Localproj_{T,\exteriorderivative} \exteriorderivative \ding \|_{L^{p}\Lambda^{k+1}(T)}
        \\&
        \leq
        C
        \min_{ \ding_{h} \in \calP\Lambda^{k}(T) }
        \| \exteriorderivative \ding - \exteriorderivative \ding_{h} \|_{L^{p}\Lambda^{k+1}(T)}
        .
    \end{align*}
    Taking the minimum over $\ding_{h} \in \calP\Lambda^{k}(T)$ and then either summing over $T$ (if $p < \infty$) or taking the maximum over $T$ (if $p=\infty$) gives the desired estimate. 
    The proof is complete.
\end{proof}

The following corollary addresses the special case of approximating closed forms, or more generally forms whose exterior derivative lies in the exterior derivative of the finite element space. It is inspired by Theorem~4.2 in~\cite{chaumontfrelet2021equivalence}. 

\begin{corollary} \label{corollary:localvsglobal:derivatives}
    Let $p \in [1,\infty]$.
    Let $\ding \in \Sobdiff^{p,p}\Lambda^{k}(\Omega,\Gamma)$ 
    with $\exteriorderivative\ding \in \exteriorderivative\calP\Lambda^{k}(\calT,\calU)$.
    Then
    \begin{align*}
        \| \ding - \Projection \ding \|_{L^{p}\Lambda^{k}(\Omega)}^{p}
        &
        \leq
        C 
        \sum_{ \substack{ T \in \subsimplices_{n}(\calT) } }
        \| \ding - \Localproj_{T} \ding \|_{L^{p}\Lambda^{k}(T)}^{p}
        , \quad p < \infty,
        \\
        \| \ding - \Projection \ding \|_{L^{\infty}\Lambda^{k}(\Omega)}
        &
        \leq
        C
        \max_{ \substack{ T \in \subsimplices_{n}(\calT) } }
        \| \ding - \Localproj_{T} \ding \|_{L^{\infty}\Lambda^{k}(T)}
        .
    \end{align*}
    Here, $C > 0$ depends only on $p$, $r$, $n$, and the mesh regularity.
\end{corollary}

\begin{proof} 
    This follows from Theorem~\ref{theorem:localvsglobal} and the commutativity property of the local projections in Proposition~\ref{proposition:localprojection}.
    Since $\exteriorderivative\ding \in \exteriorderivative\calP\Lambda^{k}(\calT,\calU)$,
    the restriction of $\exteriorderivative\ding$ to each element lies in the local polynomial space in degree $k+1$.
    By commutativity and the projection property of the local projections,
    $
        \exteriorderivative \Localproj_{T} \ding
        =
        \Localproj_{T,\exteriorderivative} \exteriorderivative \ding
        =
        \exteriorderivative \ding
    $ on each $T$.
    Thus the derivative terms in Theorem~\ref{theorem:localvsglobal} vanish, and the stated estimates follow.
\end{proof}

\begin{remark}
    The analysis in the relevant references~\cite{ern2022equivalence,chaumontfrelet2021equivalence} also addresses the dependence of the constants on the polynomial degree, which is not within the scope of this manuscript.
\end{remark}

One way to interpret the results of this section and the preceding section is the following:
if a differential form has enough regularity to admit meaningful definitions of traces, conformity, and boundary conditions, 
then good piecewise polynomial approximations will already satisfy these conditions approximately. 
Thus, enforcing conformity directly does not significantly worsen the best-approximation error.

\section{Applications to vector analysis} \label{section:anwendungen}

For the reader's convenience, we restate our results in the language of three-dimensional vector analysis.
Let $\Omega \subseteq \bbR^{3}$ be a bounded Lipschitz domain endowed with a triangulation $\calT$, 
and let $\Gamma \subseteq \partial\Omega$ be a (possibly empty) two-dimensional topological submanifold of the boundary endowed with a triangulation $\calU \subset \calT$.

\subsection{Function spaces}

We focus on the Hilbert-space setting; the extension to the Banach-space setting with general $L^{p}$ spaces is straightforward.
We recall that $L^{2}(\Omega)$ is the space of square-integrable functions. 
Write $H^{s}(\Omega) := W^{s,2}(\Omega)$ with $s \geq 0$ for the Sobolev--Slobodeckij space. 
$\bfL^{2}(\Omega)$ is the space of square-integrable vector fields on $\Omega$,
and $\bfH^{s}(\Omega)$ with $s \geq 0$ is the space of vector fields with components in $H^{s}(\Omega)$.
Write $|\cdot|_{\bfH^{s}}$ for the associated seminorm.

Recall that $L^{2}(\partial\Omega)$ is the Lebesgue space of functions on $\partial\Omega$ that are square-integrable with respect to the surface measure. 
Whenever $s > \tfrac{1}{2}$, we let $\gamma : H^{s}(\Omega) \to L^{2}(\partial\Omega)$ be the scalar trace on the boundary.
Applying this componentwise, we likewise obtain the componentwise trace $\boldgamma : \bfH^{s}(\Omega) \to L^{2}(\partial\Omega)^{3}$ of vector fields on the boundary. 
These are well-defined due to the scalar trace theorem~\ref{theorem:scalar-trace}. 

We abbreviate $H^{s}(\Omega,\Gamma) := W^{s,2}(\Omega,\Gamma)$ if $s > \frac{1}{2}$.
Let $\normalfield : \partial\Omega \to \bbR^{3}$ be the outward unit normal field along $\partial\Omega$, defined almost everywhere with respect to the surface measure. 
In accordance with Proposition~\ref{proposition:spursatz-der-differentialformen}, for $s > \tfrac{1}{2}$ we define 
\begin{align*}
    \bfH^{s}_{\tan}(\Omega,\Gamma) 
    &
    :=
    \left\{ 
        \bfu \in \bfH^{s}(\Omega) 
        \suchthat* 
        ( \normalfield \times \boldgamma \bfu ) \restriction_{\Gamma} = 0 
    \right\}
    ,
    \\
    \bfH^{s}_{\nor}(\Omega,\Gamma) 
    &
    :=
    \left\{ 
        \bfu \in \bfH^{s}(\Omega) 
        \suchthat* 
        ( \normalfield \cdot \boldgamma \bfu ) \restriction_{\Gamma} = 0 
    \right\}
\end{align*}
for the subspaces of $\bfH^{s}(\Omega)$ that have vanishing tangential or normal traces on $\Gamma$, respectively.

The curl-conforming and divergence-conforming spaces of vector fields are defined, respectively, by 
\begin{gather*}
    \bfH(\curl,\Omega)
    :=
    \left\{ \bfu \in \bfL^{2}(\Omega) \suchthat* \curl \bfu \in \bfL^{2}(\Omega) \right\}
    ,
    \\
    \bfH(\divergence,\Omega)
    :=
    \left\{ \bfu \in \bfL^{2}(\Omega) \suchthat* \divergence \bfu \in L^{2}(\Omega) \right\}
    .
\end{gather*}
We define boundary conditions in a weak sense by means of integration-by-parts formulas.
Let $\bfH(\curl,\Omega,\Gamma)$ be the subspace of $\bfH(\curl,\Omega)$
consisting of all vector fields $\bfu$ for which 
\[
    \int_{\Omega} \langle \curl \bfu, \boldphi \rangle \dx
    =
    \int_{\Omega} \langle \bfu, \curl \boldphi \rangle \dx
\]
for all vector fields $\boldphi \in C^{\infty}(\bbR^{3})^3$ whose support is disjoint from $\partial\Omega \setminus \Gamma$.

Let $\bfH(\divergence,\Omega,\Gamma)$ be the subspace of
$\bfH(\divergence,\Omega)$ consisting of all vector fields $\bfu$ for which 
\[
    \int_{\Omega} \divergence \bfu\, \phi \,\dx
    =
    -
    \int_{\Omega} \langle \bfu, \grad \phi \rangle \dx
\]
for all scalar fields $\phi \in C^{\infty}(\bbR^{3})$ whose support is disjoint from $\partial\Omega \setminus \Gamma$.

We define spaces of scalar and vector fields satisfying piecewise Sobolev--Slobodeckij regularity with respect to the triangulation:
\begin{align*}
    H^{s}(\calT) := \left\{ 
    u \in L^{2}(\Omega) \suchthat* \forall T \in \subsimplices_{3}(\calT) : u \restriction_{\interior{T}} \in H^{s}(\interior{T})
    \right\}
    ,
    \\
    \bfH^{s}(\calT) := \left\{ 
        \bfu \in \bfL^{2}(\Omega) \suchthat* \forall T \in \subsimplices_{3}(\calT) : \bfu \restriction_{\interior{T}} \in \bfH^{s}(\interior{T})
    \right\}
    .
\end{align*}

\subsection{Finite element spaces}

We introduce finite element spaces, following common conventions in the literature; see~\cite{ernguermond1},~\cite[Table~5.2]{AFW1},~\cite{monk2003finite}.
The Lagrange space of degree $r$ is written $\Lagrange_{r}(\calT) = \calP_{r}\Lambda^{0}(\calT)$. 
We write $\Nedfst_{r}(\calT)$ and $\Nedsnd_{r}(\calT)$ for the curl-conforming N\'ed\'elec spaces of the first and second kind, respectively, of degree $r$,
and $\BDM_{r}(\calT)$ and $\RT_{r}(\calT)$ for the divergence-conforming Brezzi--Douglas--Marini and Raviart--Thomas spaces, respectively, of degree $r$.
With the convention adopted in this manuscript, 
the spaces $\Nedfst_{r}(\calT)$ and $\RT_{r}(\calT)$ correspond to the trimmed spaces $\calP^{-}_{r+1}\Lambda^{1}(\calT)$ and $\calP^{-}_{r+1}\Lambda^{2}(\calT)$, respectively,
and 
the spaces $\Nedsnd_{r}(\calT)$ and $\BDM_{r}(\calT)$ correspond to the spaces $\calP_{r}\Lambda^{1}(\calT)$ and $\calP_{r}\Lambda^{2}(\calT)$, respectively. 
All these finite element spaces contain vector fields with polynomial coefficients of degree at most $r$.
Thus,
\begin{gather*}
    \Nedsnd_{r}(\calT) \subseteq \Nedfst_{r}(\calT) \subseteq \Nedsnd_{r+1}(\calT),
    \quad
    \BDM_{r}(\calT) \subseteq \RT_{r}(\calT) \subseteq \BDM_{r+1}(\calT).
\end{gather*}
In addition, we use the following notation for the subspaces satisfying homogeneous boundary conditions on $\Gamma$:
\begin{align*}
    \Lagrange_{r}(\calT,\calU) &:= H^{1}(\Omega,\Gamma) \cap \Lagrange_{r}(\calT),
    \\
    \Nedfst_{r}(\calT,\calU) &:= \bfH(\curl,\Omega,\Gamma) \cap \Nedfst_{r}(\calT),
    \\
    \Nedsnd_{r}(\calT,\calU) &:= \bfH(\curl,\Omega,\Gamma) \cap \Nedsnd_{r}(\calT),
    \\
    \BDM_{r}(\calT,\calU) &:= \bfH(\divergence,\Omega,\Gamma) \cap \BDM_{r}(\calT),
    \\
    \RT_{r}(\calT,\calU) &:= \bfH(\divergence,\Omega,\Gamma) \cap \RT_{r}(\calT).
\end{align*}
Equivalently, these spaces are defined by setting all finite element degrees of freedom associated with the simplices in the boundary subcomplex $\calU$ equal to zero.

\subsection{Projections and approximation error estimates}

The common theme of the following results is that once a field has a minimal amount of global regularity, we obtain approximation error estimates only in terms of piecewise regularity.

Given any tetrahedron $T \in \calT$, 
we let $\calT(T)$ be the collection of tetrahedra of $\calT$ that have non-empty intersection with $T$,
and let $\Omega_{T} := \interior{\bigcup\calT(T)}$ denote the extended neighborhood of the tetrahedron $T$.

We first state approximation results for the finite element spaces whose restrictions to each tetrahedron are full polynomial finite element spaces of degree $r$. We illustrate the three different regularity regimes. 

\begin{proposition} 
    Let $r \geq 1$.
    There exists a projection
    \begin{gather*}
        \Projection_{\Lagrange} : L^{2}(\Omega) \rightarrow \Lagrange_{r}(\calT,\calU)
    \end{gather*}
    such that the following inequalities hold
    for $s \in [0,r+1]$
    and all tetrahedra $T \in \calT$.

    For all $u \in H^{1}(\Omega,\Gamma) \cap H^{s}(\calT)$ with $1 \leq s$, we have
    \begin{gather*}
        \| u - \Projection_{\Lagrange} u \|_{L^{2}(T)}
        \leq
        C
        \sum_{ \substack{ T' \in \calT(T) } }
        h^{s}_{T} \left| u \right|_{H^{s}(T')}
        +
        h^{s}_{T} \left| \grad u \right|_{H^{s-1}(T')}
        .
    \end{gather*}
    For all $u \in H^{\ell}(\Omega,\Gamma) \cap H^{s}(\calT)$ with $\frac{1}{2} < \ell \leq s$, we have
    \begin{gather*}
        \| u - \Projection_{\Lagrange} u \|_{L^{2}(T)}
        \leq
        C
        \sum_{ \substack{ T' \in \calT(T) } }
        h^{s}_{T} \left| u \right|_{H^{s}(T')}
        .
    \end{gather*}
    If $T \cap \closure{\Gamma} = \emptyset$, then for all $u \in H^{s}(\Omega)$ we have
    \begin{gather*}
        \left\| u - \Projection_{\Lagrange} u \right\|_{L^{2}(T)}
        \leq
        C
        h^{s}_{T} 
        \left| u \right|_{H^{s}(\Omega_{T})}
        .
    \end{gather*}
    Here, $C > 0$ depends only on the polynomial degree $r$, $\ell$, $s$, and the mesh regularity.
\end{proposition}

\begin{proposition}
    Let $r \geq 1$.
    There exist projections
    \begin{gather*}
        \Projection_{\BDM}    : \bfL^{2}(\Omega) \rightarrow \BDM_{r}(\calT,\calU),
        \\
        \Projection_{\Nedsnd} : \bfL^{2}(\Omega) \rightarrow \Nedsnd_{r}(\calT,\calU),
    \end{gather*}
    such that the following inequalities hold
    for $s \in [0,r+1]$, $t \in [0,r]$,
    and all tetrahedra $T \in \calT$. 

    For all $\bfu \in \bfH(\curl,\Omega,\Gamma) \cap \bfH^{s}(\calT)$ with $\curl \bfu \in \bfH^{t}(\calT)$, we have
    \begin{gather*}
        \| \bfu - \Projection_{\Nedsnd} \bfu \|_{\bfL^{2}(T)}
        \leq
        C
        \sum_{ \substack{ T' \in \calT(T) } }
        h^{s}_{T}\left| \bfu \right|_{\bfH^{s}(T')} + h^{t+1}_{T}\left| \curl \bfu \right|_{\bfH^{t}(T')}
        .
    \end{gather*}
    For all $\bfu \in \bfH(\divergence,\Omega,\Gamma) \cap \bfH^{s}(\calT)$ with $\divergence \bfu \in H^{t}(\calT)$, we have
    \begin{gather*}
        \| \bfu - \Projection_{\BDM} \bfu \|_{\bfL^{2}(T)}
        \leq
        C
        \sum_{ \substack{ T' \in \calT(T) } }
        h^{s}_{T}\left| \bfu \right|_{\bfH^{s}(T')} + h^{t+1}_{T}\left| \divergence \bfu \right|_{H^{t}(T')}
        .
    \end{gather*}
    For all $\bfu \in \bfH^{\ell}_{\tan}(\Omega,\Gamma) \cap \bfH^{s}(\calT)$, where $\frac{1}{2} < \ell \leq s$, we have
    \begin{gather*}
        \| \bfu - \Projection_{\Nedsnd} \bfu \|_{\bfL^{2}(T)}
        \leq
        C
        \sum_{ \substack{ T' \in \calT(T) } }
        h^{s}_{T}\left| \bfu \right|_{\bfH^{s}(T')} 
        .
    \end{gather*}
    For all $\bfu \in \bfH^{\ell}_{\nor}(\Omega,\Gamma) \cap \bfH^{s}(\calT)$, where $\frac{1}{2} < \ell \leq s$, we have
    \begin{gather*}
        \| \bfu - \Projection_{\BDM} \bfu \|_{\bfL^{2}(T)}
        \leq
        C
        \sum_{ \substack{ T' \in \calT(T) } }
        h^{s}_{T}\left| \bfu \right|_{\bfH^{s}(T')} 
        .
    \end{gather*}
    If $T \cap \closure{\Gamma} = \emptyset$, then for all $\bfu \in \bfH^{s}(\Omega)$ we have
    \begin{gather*}
        \left\| \bfu - \Projection_{\BDM} \bfu \right\|_{\bfL^{2}(T)}
        \leq
        C
        h^{s}_{T} \left| \bfu \right|_{\bfH^{s}(\Omega_{T})}
        ,
        \\
        \left\| \bfu - \Projection_{\Nedsnd} \bfu \right\|_{\bfL^{2}(T)}
        \leq
        C
        h^{s}_{T} \left| \bfu \right|_{\bfH^{s}(\Omega_{T})}
        .
    \end{gather*}
    Here, $C > 0$ depends only on the polynomial degree $r$, $\ell$, $s$, $t$, and the mesh regularity.
\end{proposition}

Next, we state approximation results for the trimmed finite element spaces. Among the different ways to describe them, we emphasize the following.
On each tetrahedron, $\Nedfst_{r}$ and $\RT_{r}$ contain all vector polynomial fields of degree at most $r$. 
Moreover, the curls of $\Nedfst_{r}$ and the divergences of $\RT_{r}$ have the same local ranges as the curls and divergences, respectively, of vector polynomial fields of degree at most $r+1$. 

\begin{proposition}
    Let $r \geq 0$.
    There exist projections
    \begin{gather*}
        \Projection_{\Nedfst} : \bfL^{2}(\Omega) \rightarrow \Nedfst_{r}(\calT,\calU),
        \\
        \Projection_{\RT}     : \bfL^{2}(\Omega) \rightarrow     \RT_{r}(\calT,\calU),
    \end{gather*}
    such that the following inequalities hold
    for $s \in [0,r+1]$, $t \in [0,r+1]$,
    and all tetrahedra $T \in \calT$.
    
    For all $\bfu \in \bfH(\curl,\Omega,\Gamma) \cap \bfH^{s}(\calT)$ with $\curl \bfu \in \bfH^{t}(\calT)$, we have
    \begin{gather*}
        \| \bfu - \Projection_{\Nedfst} \bfu \|_{\bfL^{2}(T)}
        \leq
        C
        \sum_{ \substack{ T' \in \calT(T) } }
        h^{s}_{T} | \bfu |_{\bfH^{s}(T')} + h^{t+1}_{T} | \curl \bfu |_{\bfH^{t}(T')}
        .
    \end{gather*}
    For all $\bfu \in \bfH(\divergence,\Omega,\Gamma) \cap \bfH^{s}(\calT)$ with $\divergence \bfu \in H^{t}(\calT)$, we have
    \begin{gather*}
        \| \bfu - \Projection_{\RT} \bfu \|_{\bfL^{2}(T)}
        \leq
        C
        \sum_{ \substack{ T' \in \calT(T) } }
        h^{s}_{T} | \bfu |_{\bfH^{s}(T')} + h^{t+1}_{T} | \divergence \bfu |_{H^{t}(T')}
        .
    \end{gather*}
    For all $\bfu \in \bfH^{\ell}_{\tan}(\Omega,\Gamma) \cap \bfH^{s}(\calT)$, where $\frac{1}{2} < \ell \leq s$, we have
    \begin{gather*}
        \| \bfu - \Projection_{\Nedfst} \bfu \|_{\bfL^{2}(T)}
        \leq
        C
        \sum_{ \substack{ T' \in \calT(T) } }
        h^{s}_{T}\left| \bfu \right|_{\bfH^{s}(T')} 
        .
    \end{gather*}
    For all $\bfu \in \bfH^{\ell}_{\nor}(\Omega,\Gamma) \cap \bfH^{s}(\calT)$, where $\frac{1}{2} < \ell \leq s$, we have
    \begin{gather*}
        \| \bfu - \Projection_{\RT} \bfu \|_{\bfL^{2}(T)}
        \leq
        C
        \sum_{ \substack{ T' \in \calT(T) } }
        h^{s}_{T}\left| \bfu \right|_{\bfH^{s}(T')} 
        .
    \end{gather*}
    If $T \cap \closure{\Gamma} = \emptyset$, then for all $\bfu \in \bfH^{s}(\Omega)$ we have
    \begin{gather*}
        \| \bfu - \Projection_{\Nedfst} \bfu \|_{\bfL^{2}(T)}
        \leq
        C
        h^{s}_{T} \left| \bfu \right|_{\bfH^{s}(\Omega_{T})}
        ,
        \\
        \| \bfu - \Projection_{\RT} \bfu \|_{\bfL^{2}(T)}
        \leq
        C
        h^{s}_{T} \left| \bfu \right|_{\bfH^{s}(\Omega_{T})}
        .
    \end{gather*}
    Here, $C > 0$ depends only on the polynomial degree $r$, $\ell$, $s$, $t$, and the mesh regularity.
\end{proposition}

\appendix

\section{Extension operators and traces}\label{section:extension}

Differential forms admit two different notions of traces. 
The \emph{strong tangential traces} of differential forms are defined in terms of Sobolev trace theory. 
In contrast, \emph{weak tangential traces} are defined by means of an integration-by-parts formula;
we have used weak tangential traces in Equation~\eqref{math:weak-trace-condition} to define homogeneous partial boundary conditions. 
The two notions are conceptually different but agree in many special cases. 
For example, given $\ding \in C^{\infty}\Lambda^{k}(\closure{\Omega})$ and $\eta \in C^{\infty}\Lambda^{n-k-1}(\closure{\Omega})$ on a Lipschitz domain $\Omega \subseteq \bbR^{n}$, 
the identity 
\begin{align*}
    \int_{\partial\Omega} \Trace\ding \wedge \Trace\eta 
    = 
    \int_{\Omega} \exteriorderivative\ding \wedge \eta + (-1)^{k} \ding \wedge \exteriorderivative\eta 
\end{align*}
relates two different notions of traces as functionals on $C^{\infty}\Lambda^{n-k-1}(\closure{\Omega})$. 
The left-hand side generalizes to the strong tangential trace of Sobolev--Slobodeckij differential forms. 
The right-hand side generalizes to the weak tangential trace of differential forms with weak exterior derivative. 

We want to show that the strong and weak tangential traces of differential forms agree whenever both are defined. 
This is not strictly necessary for this manuscript, but the observation appears to be of broader interest. 
We only address the case where the form and its weak exterior derivative have the same Lebesgue exponents.

For any $1 \leq p \leq \infty$ and $s \in [0,\infty)$, the intersection space $W^{s,p}\Lambda^{k}(\Omega) \cap \Sobdiff^{p,p}\Lambda^{k}(\Omega)$ is equipped with the intersection norm 
\begin{align*}
    \| \ding \|_{W^{s,p}\Lambda^{k}(\Omega) \cap \Sobdiff^{p,p}\Lambda^{k}(\Omega)}
    :=
    \| \ding \|_{W^{s,p}\Lambda^{k}(\Omega)}
    + 
    \| \ding \|_{\Sobdiff^{p,p}\Lambda^{k}(\Omega)}
    .
\end{align*}
The intersection space $W^{s,p}\Lambda^{k}(\Omega) \cap \Sobdiff^{p,p}\Lambda^{k}(\Omega)$ is a Banach space.

We first construct an extension operator compatible with both regularity structures. 

\begin{theorem}\label{theorem:universal-extension}
    Let $\Omega \subseteq \bbR^{n}$ be a bounded Lipschitz domain. 
    Let $1 < p < \infty$ and $s \in [0,1]$. 
    There exists a bounded operator  
    \begin{align}\label{math:universal-extension}
       \Extension^{k} : 
        W^{s,p}\Lambda^{k}(\Omega) \cap \Sobdiff^{p,p}\Lambda^{k}(\Omega) 
        \to 
        W^{s,p}\Lambda^{k}(\bbR^{n}) \cap \Sobdiff^{p,p}\Lambda^{k}(\bbR^{n}),
    \end{align}
    such that 
    $( \Extension^{k} \ding) \restriction_{\Omega} = \ding$ for all $\ding \in W^{s,p}\Lambda^{k}(\Omega) \cap \Sobdiff^{p,p}\Lambda^{k}(\Omega)$. 
\end{theorem}
\begin{proof}
    We use the regularized Poincar\'e operators~\cite{costabel2010}. 
    There exist bounded operators 
    \begin{gather*}
        \Poincare^{k} : W^{s,p}\Lambda^{k}(\Omega) \to W^{s+1,p}\Lambda^{k-1}(\Omega), \quad 0 \leq s \leq 1,
    \end{gather*}
    and bounded finite-rank smoothing operators 
    \begin{gather*}
        \Harmonic^{k} : L^{p}\Lambda^{k}(\Omega) \to W^{1,p}\Lambda^{k}(\Omega)
    \end{gather*}
    such that for all $\ding \in \Sobdiff^{p,p}\Lambda^{k}(\Omega)$ it holds that 
    \begin{align*}
        \ding = \exteriorderivative \Poincare^{k} \ding + \Harmonic^{k} \ding + \Poincare^{k+1} \exteriorderivative \ding,
        \quad 
        \exteriorderivative \Harmonic^{k} \ding = 0.
    \end{align*}
    We also use the universal extension operators for Sobolev spaces on Lipschitz domains~\cite{stein1970singular,rychkov1999restrictions}. 
    There exist bounded operators 
    \begin{align*}
        \mathscr{A}^{k} : W^{s,p}\Lambda^{k}(\Omega) \to W^{s,p}\Lambda^{k}(\bbR^{n}), 
        \qquad 
        s \in [0,\infty),
    \end{align*}
    such that $\ding = ( \mathscr{A}^{k} \ding) \restriction_{\Omega}$ for all $\ding \in W^{s,p}\Lambda^{k}(\Omega)$. 
    
    Fix $\ding \in W^{s,p}\Lambda^{k}(\Omega) \cap \Sobdiff^{p,p}\Lambda^{k}(\Omega)$. 
    Then $\exteriorderivative \ding \in L^{p}\Lambda^{k+1}(\Omega)$.
    It also follows that 
    $\Poincare^{k+1} \exteriorderivative \ding \in W^{1,p}\Lambda^{k}(\Omega)$
    and 
    $\Poincare^{k} \ding \in W^{s+1,p}\Lambda^{k-1}(\Omega)$.
    Moreover, $\Harmonic^{k} \ding \in W^{1,p}\Lambda^{k}(\Omega)$. 
    We define 
    \begin{align*} 
        \Extension^{k} \ding 
        :=
        \exteriorderivative \mathscr{A}^{k-1} \Poincare^{k} \ding 
        + 
        \mathscr{A}^{k} \Harmonic^{k} \ding 
        + 
        \mathscr{A}^{k} \Poincare^{k+1} \exteriorderivative \ding
        .
    \end{align*}
    By construction, $\Extension^{k} \ding \in W^{s,p}\Lambda^{k}(\bbR^{n}) \cap \Sobdiff^{p,p}\Lambda^{k}(\bbR^{n})$. 
    This defines~\eqref{math:universal-extension} as a bounded operator because $W^{1,p}\Lambda^{k}(\Omega)$ embeds into $W^{s,p}\Lambda^{k}(\Omega)$. 
    We observe  
    \begin{align*}
        ( \Extension^{k} \ding ) \restriction_{\Omega}
        &
        =
        ( \exteriorderivative \mathscr{A}^{k-1} \Poincare^{k} \ding ) \restriction_{\Omega}
        + 
        ( \mathscr{A}^{k} \Harmonic^{k} \ding ) \restriction_{\Omega}
        + 
        ( \mathscr{A}^{k} \Poincare^{k+1} \exteriorderivative \ding ) \restriction_{\Omega}
        \\&
        =
        \exteriorderivative \Poincare^{k} \ding 
        + 
        \Harmonic^{k} \ding
        + 
        \Poincare^{k+1} \exteriorderivative \ding
        =
        \ding
        .
    \end{align*}
    The proof is complete. 
\end{proof}

\begin{theorem}\label{theorem:smooth-forms-dense-in-intersection-space}
    Let $\Omega \subseteq \bbR^{n}$ be a bounded Lipschitz domain. 
    Let $1 < p < \infty$ and $s \in [0,1]$. 
    Then $C^{\infty}\Lambda^{k}(\closure{\Omega})$ 
    is dense in $W^{s,p}\Lambda^{k}(\Omega) \cap \Sobdiff^{p,p}\Lambda^{k}(\Omega)$.
\end{theorem}

\begin{proof}
    Let $\ding \in W^{s,p}\Lambda^{k}(\Omega) \cap \Sobdiff^{p,p}\Lambda^{k}(\Omega)$. 
    Define $\ding^{e} := \Extension^{k} \ding \in W^{s,p}\Lambda^{k}(\bbR^{n}) \cap \Sobdiff^{p,p}\Lambda^{k}(\bbR^{n})$ 
    using the extension operator of Theorem~\ref{theorem:universal-extension}. 
    For any $\epsilon > 0$, let $\ding^{e}_{\epsilon} \in C^{\infty}\Lambda^{k}(\bbR^{n})$ be the convolution of $\ding^{e}$ with a standard smooth compactly supported mollifier of unit integral, scaled by $\epsilon$~\cite{burenkov1998sobolev,ernguermond1,werner,alt_funktionalanalysis}.
    The convolution commutes with the exterior derivative. 
    Then $\ding^{e}_{\epsilon}$ is in the Banach space $W^{s,p}\Lambda^{k}(\bbR^{n}) \cap \Sobdiff^{p,p}\Lambda^{k}(\bbR^{n})$
    and $\ding^{e}$ is the limit of $\ding^{e}_{\epsilon}$ in that Banach space as $\epsilon$ goes to zero. 
    In particular, $\ding^{e}_{\epsilon} \restriction_{\Omega}$ converges to $\ding$ in $W^{s,p}\Lambda^{k}(\Omega) \cap \Sobdiff^{p,p}\Lambda^{k}(\Omega)$.
\end{proof}

The following proposition compares boundary traces, first via an integration-by-parts formula and second via Sobolev trace theory.  

\begin{proposition} \label{proposition:vergleich-spur-stark-schwach}
    Let $\Omega$ be a bounded Lipschitz domain.
    Let $p \in [1,\infty]$ and $m \in (0,1]$ with $m > 1/p$ or $m = 1$.
    If $\ding \in W^{m,p}\Lambda^{k}(\Omega) \cap \Sobdiff^{p,p}\Lambda^{k}(\Omega)$, 
    then for every $\eta \in C^{\infty}\Lambda^{n-k-1}(\closure{\Omega})$ 
    it holds that 
    \begin{align}\label{math:equivalence-of-weak-and-strong-traces}
        \int_{\partial\Omega} \Trace \ding \wedge \Trace \eta
        =
        \int_{\Omega} \exteriorderivative \ding \wedge \eta 
        +
        (-1)^{k}
        \int_{\Omega} \ding \wedge \exteriorderivative \eta 
        .
    \end{align}
\end{proposition}
\begin{proof}
    First, we notice that~\eqref{math:equivalence-of-weak-and-strong-traces} holds when $\ding \in C^{\infty}\Lambda^{k}(\closure{\Omega})$.
    We make a case distinction based on the value of $1 \leq p \leq \infty$. 
    
    Consider the case $p = 1$. Then $m = 1$,
    and hence the intersection space $W^{1,1}\Lambda^{k}(\Omega) \cap \Sobdiff^{1,1}\Lambda^{k}(\Omega)$  
    is the Sobolev space $W^{1,1}\Lambda^{k}(\Omega)$ but with a different, equivalent norm.
    Since $C^{\infty}\Lambda^{k}(\closure{\Omega})$ is dense in that space, the statement~\eqref{math:equivalence-of-weak-and-strong-traces} follows by continuity and density on both sides. 
     
    Consider the case $1 < p < \infty$ with $1/p < m \leq 1$. 
    Using Theorem~\ref{theorem:smooth-forms-dense-in-intersection-space}, 
    the density of $C^{\infty}\Lambda^{k}(\closure{\Omega})$ in $W^{m,p}\Lambda^{k}(\Omega) \cap \Sobdiff^{p,p}\Lambda^{k}(\Omega)$ and continuity on both sides 
    imply that~\eqref{math:equivalence-of-weak-and-strong-traces} holds.
    
    Finally, consider the case $p = \infty$. 
    Since $m > 0$, there exists $r \in (1,\infty)$ with $1/r < m$.
    Because $W^{m,\infty}\Lambda^{k}(\Omega) \cap \Sobdiff^{\infty,\infty}\Lambda^{k}(\Omega)$ 
    embeds into $W^{m,r}\Lambda^{k}(\Omega) \cap \Sobdiff^{r,r}\Lambda^{k}(\Omega)$,
    we conclude that~\eqref{math:equivalence-of-weak-and-strong-traces} holds. 
\end{proof}

We want to compare the spaces $W^{m,p}\Lambda^{k}(\Omega,\Gamma)$ and $\Sobdiff^{p,p}\Lambda^{k}(\Omega,\Gamma)$. 
The former space, which is the Sobolev--Slobodeckij space of $k$-forms with vanishing tangential trace on $\Gamma$, admits an equivalent characterization in the sense of distributions.

\begin{lemma} \label{lemma:support-property-along-boundary} 
    Let $\Omega$ be a bounded Lipschitz domain and let $\Gamma \subseteq \partial\Omega$ be relatively open.
    Let $p \in [1,\infty]$ and $m \in [0,\infty)$ with $m > 1/p$ or $m \geq 1$.
    Assume $\ding \in W^{m,p}\Lambda^{k}(\Omega)$. 
    Then $\ding \in W^{m,p}\Lambda^{k}(\Omega,\Gamma)$
    if and only if 
    \begin{align*}
        \int_{\partial\Omega} \Trace \ding \wedge \Trace \eta = 0 
    \end{align*}
    for all $\eta \in C^{\infty}\Lambda^{n-k-1}(\closure{\Omega})$ with $\supp \eta \cap \partial\Omega \subseteq \Gamma$.
\end{lemma}
\begin{proof}
    Suppose that $\beta \in L^{p}(\partial\Omega)$. 
    We define the linear functional 
    \begin{align*}
        T_{0}(\eta) = \oint_{\partial\Omega} \beta(x) \eta(x) \,\dif s, \qquad \eta \in C^{\infty}(\closure{\Omega})
        .
    \end{align*}
    The linear functional $T_{0}(\eta)$ is defined on $C^{\infty}(\closure{\Omega})$ and extends to a linear functional on $C^{\infty}(\bbR^{n})$, by restriction. 
    Notice that $T_{0}(\eta)$ is a bounded linear functional on, say, $W^{1,\infty}\Lambda^{k}(\bbR^{n})$, 
    and thus defines a distribution on $\bbR^{n}$. 
    The support of the distribution $T_{0}$ equals the essential support of $\beta$ with respect to the Hausdorff measure of dimension $n-1$. %
    Therefore, the support of the distribution $T_{0}$ on $\bbR^{n}$ is a subset of $\partial\Omega \setminus \Gamma$ 
    if and only if 
    $\beta = 0$ almost everywhere on $\Gamma$.

    With this preparation in place, consider any $\ding \in W^{m,p}\Lambda^{k}(\Omega,\Gamma)$.
    Then $\zeta := \Trace \ding$ exists as a function on $\partial\Omega$ with values in the $k$-th exterior power of $\bbR^{n}$,
    and the coefficients of $\zeta$ are in the Lebesgue space $L^{p}(\partial\Omega)$. 
    We define the linear functional 
    \begin{align*}
        T(\eta) = \int_{\partial\Omega} \Trace \ding \wedge \Trace \eta, \qquad \eta \in C^{\infty}\Lambda^{n-k-1}(\closure{\Omega})
        .
    \end{align*}
    Then $T(\eta)$ extends to a linear functional on $C^{\infty}\Lambda^{n-k-1}(\bbR^{n})$.
    Being bounded on, say, the space $W^{1,\infty}\Lambda^{n-k-1}(\bbR^{n})$, it defines a distribution on $\bbR^{n}$ with support in $\closure{\Omega}$. 
    
    On the one hand, if $\zeta = 0$ almost everywhere on $\Gamma$, 
    then $T(\eta) = 0$ for all $\eta \in C^{\infty}\Lambda^{n-k-1}(\closure{\Omega})$ with $\supp \eta \cap \partial\Omega \subseteq \Gamma$.
    On the other hand, 
    suppose $T(\eta) = 0$ for all $\eta \in C^{\infty}\Lambda^{n-k-1}(\closure{\Omega})$ with $\supp \eta \cap \partial\Omega \subseteq \Gamma$. 
    We represent $\zeta$ in coordinates. 
    Testing with forms that isolate the coordinates, we see that each coordinate $\zeta_{i}$ satisfies $\int_{\partial\Omega} \zeta_{i} \eta = 0$ 
    for all $\eta \in C^{\infty}(\closure{\Omega})$ with $\supp \eta \cap \partial\Omega \subseteq \Gamma$. 
    By our previous discussion, $\zeta_{i} = 0$ almost everywhere on $\Gamma$. 
    This means that $\zeta = 0$ almost everywhere on $\Gamma$. 
\end{proof}

\begin{theorem} \label{theorem:spursatz-verschwindende-spur}
    Let $\Omega$ be a bounded Lipschitz domain and let $\Gamma \subseteq \partial\Omega$ be relatively open. 
    Let $p \in [1,\infty]$ and $m \in [0,\infty)$ with $m > 1/p$ or $m \geq 1$.
    For all $\ding \in W^{m,p}\Lambda^{k}(\Omega) \cap \Sobdiff^{p,p}\Lambda^{k}(\Omega)$, we have the equivalence 
    \begin{align*}
        \ding \in \Sobdiff^{p,p}\Lambda^{k}(\Omega,\Gamma)
        \quad\Longleftrightarrow\quad
        \ding \in W^{m,p}\Lambda^{k}(\Omega,\Gamma)
        .
    \end{align*}
\end{theorem}
\begin{proof}
    Fix $\ding \in W^{m,p}\Lambda^{k}(\Omega) \cap \Sobdiff^{p,p}\Lambda^{k}(\Omega)$. 

    First assume that $\ding \in W^{m,p}\Lambda^{k}(\Omega,\Gamma)$. 
    Let $\eta \in C^{\infty}_c\Lambda^{n-k-1}(\bbR^{n})$ satisfy $\supp \eta \cap \partial\Omega \subseteq \Gamma$. 
    Using Proposition~\ref{proposition:vergleich-spur-stark-schwach}, 
    we obtain
    \[
        \int_{\Omega} \exteriorderivative \ding \wedge \eta
        =
        (-1)^{k+1}
        \int_{\Omega} \ding \wedge \exteriorderivative \eta
        +
        \int_{\partial\Omega} \Trace\ding \wedge \Trace\eta 
        .
    \]
    The boundary term vanishes because %
    \[
        (\Trace\ding) \restriction_{\Gamma} = 0,
        \qquad 
        \supp \Trace \eta \subseteq \Gamma
        .
    \]
    By definition, $\ding\in \Sobdiff^{p,p}\Lambda^{k}(\Omega,\Gamma)$.

    Conversely, assume that $\ding\in \Sobdiff^{p,p}\Lambda^{k}(\Omega,\Gamma)$. 
    For every $\eta \in C^{\infty}_c\Lambda^{n-k-1}(\bbR^{n})$ with $\supp \eta \cap \partial\Omega \subseteq \Gamma$,
    the defining integration-by-parts identity gives
    \[
        \int_{\Omega} \exteriorderivative \ding \wedge \eta
        =
        (-1)^{k+1}
        \int_{\Omega} \ding \wedge \exteriorderivative\eta
        .
    \]
    Using Proposition~\ref{proposition:vergleich-spur-stark-schwach} again, the following holds for every such $\eta$:
    \[
        \int_{\partial\Omega} \Trace\ding \wedge \Trace\eta = 0
        .
    \]
    Therefore, %
    for every $\eta \in C^{\infty}\Lambda^{n-k-1}(\closure{\Omega})$ with $\supp \eta \cap \partial\Omega \subseteq \Gamma$ it holds that 
    \[
        \int_{\partial\Omega} \Trace\ding \wedge \Trace\eta = 0.
    \]
    Invoking Lemma~\ref{lemma:support-property-along-boundary}, we conclude $\ding \in W^{m,p}\Lambda^{k}(\Omega,\Gamma)$. 
    This completes the proof. 
\end{proof}

\end{document}